\documentclass[aihp]{imsart}

\RequirePackage{amsthm,amsmath,amsfonts,amssymb}
\RequirePackage[numbers]{natbib}
\RequirePackage[colorlinks,citecolor=blue,urlcolor=blue]{hyperref}
\usepackage{color}
\usepackage{soul}   
\usepackage[dvipsnames]{xcolor}   

\usepackage{ifthen}
\usepackage{xkeyval}
\usepackage{todonotes}
\usepackage{ upgreek }
\usepackage{stmaryrd}
\SetSymbolFont{stmry}{bold}{U}{stmry}{m}{n}
\usepackage{amsthm}
\usepackage{float}

\usepackage{ bbm }
\usepackage{ stmaryrd }
\usepackage{ mathrsfs }
\usepackage{ frcursive }
\usepackage{ comment }

\usepackage{pgf, tikz}
\usetikzlibrary{shapes}
\usepackage{varioref}
\usepackage{enumitem}
\usepackage{longtable}

\usepackage{mathtools}

\usepackage{dsfont}

\startlocaldefs

\newtheorem{theorem}{Theorem}[section]
\newtheorem*{theorem*}{Theorem}
\newtheorem{lemma}[theorem]{Lemma}

\newtheorem{corollary}[theorem]{Corollary}
\newtheorem{proposition}[theorem]{Proposition}

\labelformat{hypothesis}{\textbf{M\kern-0.1mm#1}}

\newtheorem{condition}{Condition}
\newtheorem{conditionA}{A\kern-0.1mm}
\labelformat{conditionA}{\textbf{A\kern-0.1mm#1}}

\renewcommand\dots{\hbox to 1em{.\hss.\hss.}}

\theoremstyle{definition}

\newtheorem{remark}[theorem]{Remark}

\numberwithin{equation}{section}

\newcommand*{\abs}[1]{\left\lvert#1\right\rvert}
\newcommand*{\norm}[1]{\left\lVert#1\right\rVert}

\newcommand{\Mf}{\mathfrak{M}}

\def\bb#1{\mathbb{#1}}

\def\bf#1{\mathbf{#1}}
\def\scr#1{\mathscr{#1}}

\def\geq{\geqslant}
\def\leq{\leqslant}
\def\vphi{\varphi}
\def\EQ{\mathbb{E}_{\mathbb{Q}_s^x}} 
\newcommand\ee{\varepsilon}

\DeclareMathOperator{\Leb}{Leb}

\def\geq{\geqslant}
\def\leq{\leqslant}
\def\Rd {\mathbb{R}^d}
\def\Rd*{(\mathbb{R}^d)^*}
\def\Pd{{\mathbb{P}}^{d-1}}
\def\Pd*{(\bb P(V))^*}

\endlocaldefs

\begin{document}

\begin{frontmatter}

\title{The extremal position of a branching random walk\\ on the general linear group}
\runtitle{Extremal position of a branching random walk on ${\rm GL}(d, \bb R)$}

\begin{aug}
\author[A]{\inits{I.}\fnms{Ion}~\snm{Grama}\ead[label=e1]{ion.grama@univ-ubs.fr}},
\author[B]{\inits{S.}\fnms{Sebastian}~\snm{Mentemeier}\ead[label=e2]{mentemeier@uni-hildesheim.de}}
\and
\author[C]{\inits{H.}\fnms{Hui}~\snm{Xiao}\ead[label=e3]{xiaohui@amss.ac.cn}}
\address[A]{Univ Bretagne Sud, CNRS UMR 6205 LMBA, Campus de Tohannic 56017, Vannes, France\printead[presep={,\ }]{e1}}

\address[B]{Universit\"at Hildesheim, Institut f\"ur Mathematik und Angewandte Informatik, Hildesheim, Germany\printead[presep={,\ }]{e2}}

\address[C]{Academy of Mathematics and Systems Science, 
Chinese Academy of Sciences, Beijing 100190, China\printead[presep={,\ }]{e3}}
\end{aug}

\begin{abstract}
Consider a branching random walk $(G_u)_{u\in \mathbb T}$ on the general linear group $\textrm{GL}(V)$ of a finite dimensional space $V$, where $\mathbb T$ is the associated genealogical tree with nodes $u$. 
For any starting point $v \in V \setminus\{0\}$ with $\|v\|=1$ and $x = \bb R v \in \bb P(V)$, 
let $M^x_n=\max_{|u| = n} \log \| G_u v \|$ denote the maximal position of the walk $\log \| G_u v \|$ in the generation $n$. 
We first show that under suitable conditions, 
$\lim_{n \to \infty} \frac{M_n^x }{n} = \gamma$ almost surely, 
where $\gamma \in \mathbb R$ is a constant. 
Then, in the case when $\gamma = 0$, 
under appropriate {\it boundary conditions}, 
we refine the last statement by determining the rate of convergence at which $M_n^x$ 
converges to $-\infty$. We prove in particular that $\lim_{n \to \infty} \frac{M_n^x}{\log n} = -\frac{3}{2\alpha}$ in probability, where $\alpha >0$ is a constant determined by the boundary conditions. Analogous properties are established for the minimal position. As a consequence we derive the asymptotic speed of the maximal and minimal positions for the coefficients, the operator norm and the spectral radius of $G_u$. 
\end{abstract}

\begin{abstract}[language=french]
Consid\'erons une marche al\'eatoire branchante $(G_u)_{u\in \mathbb T}$ sur le groupe lin\'eaire g\'en\'eral $\textrm{GL}(V)$ 
d'un espace $V$ de dimension finie, o\`u  $\mathbb T$ est l'arbre g\'en\'ealogique associ\'e avec des n\oe uds $u$.
Pour tout point de d\'epart $v \in V \setminus{0}$ avec $\| v \|=1$ et $x = \bb R v \in \bb P(V)$,
soit $M^x_n = \max_{| u | = n} \log \| G_u v \|$ la position maximale de la marche $\log \| G_u v \|$ \`a la g\'en\'eration $n$.
Nous montrons d'abord que sous des conditions appropri\'ees,
$\lim_{n \to \infty} \frac{M_n^x }{n} = \gamma$ presque s\^urement,
o\`u $\gamma \in \mathbb R$ est une constante.
Ensuite, dans le cas o\`u $\gamma = 0$,
sous des {\it conditions fronti\`eres} appropri\'ees,
nous affinons la derni\`ere affirmation en d\'eterminant la vitesse de convergence \`a laquelle $M_n^x$
converge vers $-\infty$. Nous prouvons en particulier que $\lim_{n \to \infty} \frac{M_n^x}{\log n} = -\frac{3}{2\alpha}$ en probabilit\'e, o\`u $\alpha >0$ est une constante d\'etermin\'ee par les conditions fronti\`eres. 
Des propri\'et\'es analogues sont \'etablies pour la position minimale. Comme cons\'equence, 
nous d\'erivons la vitesse asymptotique des positions maximale et minimale pour les coefficients, la norme op\'erateur et le rayon spectral de $G_u$.
\end{abstract}



\end{frontmatter}


\section{Introduction}
Let $V=\bb R^d$ be a $d$-dimensional Euclidean vector space equipped with the norm $\| \cdot \|$, 
where $d\geq 1$ is an integer. 
Denote by  $\bb G = {\rm GL}(V)$ the general linear group of the vector space $V$. 
First at time $0$ we start with one root particle $\emptyset$. 
At time step 1, the root particle generates a random number of children.  
At subsequent time steps, this happens in an independent manner according to the same random mechanism
for every obtained child.
The above iterations generate a branching process whose genealogical tree is denoted by $\bb T$ with nodes denoted by $u$. 
We write $|u|=n$ when $u$ is a node of generation $n\geq 1$.
A branching random walk on the group $\bb G$ is then obtained as follows. 
We assign to the root particle $\emptyset$ the identity matrix ${\rm Id}$ and to each of its children $u$ with $|u|=1$ a random matrix $g_u$. 
Note that the sequence $g_u$, $|u|=1$ may have any dependence structure. 
On each of the further generations children are assigned matrices $g_u$ in an independent manner according to the same random mechanism. 
Starting with the root, by successive multiplication from the left of the random elements 
on the branch corresponding to a node $u\in \bb T$,  we obtain a branching random walk with values in $\bb G$ 
which we shall denote by $(G_u)_{u\in \bb T}$.

Denote by $\bb P(V)$ the projective space of $\bb R^d$. 
Let  $v\in \bb R^d$ be a vector with $\|v\| =1$ and let $x=\bb R v \in \bb P(V)$ be the direction of $v$. 
The first objective of the paper is to study the asymptotic behaviour of 
the maximal displacement $M_n^x=\max_{| u | = n} \log \|G_u v\|$ as $n \to \infty$.  
Moreover, we shall consider the following extension of this problem.  
Denote by $G_u\cdot x$ the Markov branching process associated   
to the projective action of the group $\bb G$ on the projective space $\bb P(V)$, 
which describes the behaviour of the directions of the walk $(G_u v)_{u\in \bb T}$.
Given a set of directions $A \subseteq \bb P(V)$, it is of interest to study 
the maximal position of  $\log \|G_u v\|$ provided $G_u\cdot x\in A$, 
i.e.\  the directed maximal displacement $M_n^x(A)= \sup_{G_u\cdot x\in A, | u | = n} \log \|G_u v\|$,
where, by convention we set $\sup \emptyset =-\infty$, to include the case when the condition $G_u \cdot x \in A$ is not satisfied.
We shall  as well study similar problems for the related characteristics of $(G_u)$ such as the coefficients, the operator norm and the spectral radius. 

Let us describe briefly our main results in the particular case when $A = \bb P(V)$.
Under suitable conditions, the following convergence holds (see Theorem \ref{Thm_LLN-001}): 
for any $x \in \bb P(V)$,  conditionally on the system's survival,
\begin{align} \label{LawLN0}
\lim_{n \to \infty} \frac{M_n^x }{n} = \gamma   \quad \mbox{almost surely,} 
\end{align} 
where $\gamma\in \bb R$ is a constant. 
Upon subtracting this linear drift $\gamma n$, we end up with a transformed process that is in the {\it boundary case} 
(for a precise formulation see condition \ref{Condi_ms} 
formulated in the next section). 
Our main interest is to investigate second-order asymptotics. 
We prove in particular that, for any $x \in \bb P(V)$, conditionally on the system's survival, 
$$ 
\lim_{n \to \infty} \frac{M_n^x - \gamma n}{\log n} = -\frac{3}{2\alpha}  \quad \text{in probability,}
$$
where $\alpha >0$ is a constant defined by the boundary condition \ref{Condi_ms}.
The full statement of our result for the maximal displacement is the content of Theorem \ref{Thm_In_Pro} below. 

The asymptotic behaviour of the maximal position of classical branching random walks in 
$\bb R^1$ has been investigated by many authors, 
see for example Hammersley \cite{Ham74}, Kingman \cite{Kin75},  Biggins \cite{Big76}. 
Significant progress has been made by Hu and Shi \cite{HS09}, Addario-Berry and Reed \cite{AR09},  A\"id\'ekon \cite{Aid13},
A\"id\'ekon and Shi \cite{AS10, AS14}, 
where  conditioned limit theorems for sums of independent and identically distributed (i.i.d.)\ random variables are used to 
establish a law of large numbers and limit theorems for its fluctuations. 

The branching random walk on the general linear group $\bb G$ studied here is a natural extension of the classical branching random walk in $\bb R^1$ and is of particular interest because the underlying group is non-commutative. 
The model was investigated in Buraczewski,  Damek, Guivarc'h and Mentemeier \cite{BDGM14}, Mentemeier \cite{Men16} and Bui, Grama and Liu \cite{BGL20}.
In view of the one-dimensional results, it is natural to ask for the behaviour of $M_n^x$ and, moreover, 
to study the maximal position of particles with a given direction $M_n^x(A)$. 
Both these questions have not been addressed in the works mentioned above and become considerably more involved for branching random walks on groups.  
The difficulty in obtaining such results is in a great part due to the lack of the corresponding 
conditioned local limit theory for products of random matrices, besides the 
inherently heavy argument in dealing with extremal position of the walk $(G_u v)_{u\in \bb T}$. 

Recent progress for products of random matrices has been made in Grama, Le Page and Peign\'e \cite{GLP17} and
Grama, Lauvergnat and Le Page  \cite{GLL18}, where some integral conditioned theorems have been established.
Local limit theorems for Markov chains with finite states have been considered in  Grama, Lauvergnat and Le Page \cite{GLL20},
however, establishing convenient conditioned local limit theorems 
for products of random matrices is still an open problem. 
Following some recent developments in Grama, Quint and Xiao  \cite{GQX21} and Grama and Xiao \cite{GX21}, 
in the present paper we shall establish some new conditioned limit theorems for products of random matrices
under the  assumption that the matrices have a density with respect to the Haar measure on $\bb G$.  
This conditioned local limit theory is the key point in studying the maximal displacement $M_n^x(A)$. 
Besides, it could also be useful for studying other problems like a version of the Seneta-Heyde 
theorem for branching random walks on groups. 
This question will be considered in a subsequent work.

\section{The setup and main results}

\subsection{Notations and Assumptions} \label{Sec-notation}
Let $\bb R_{+} = [0, \infty)$ and denote by  $\bb N$ the set of nonnegative integers.
For any integer $d \geq 1$, denote by $V = \bb R^d$ the $d$-dimensional Euclidean space.
We fix a  basis $(e_i)_{1 \leq i \leq d}$ of $V$ 
and define the associated norm on $V$ by $\|v\|^2 = \sum_{i=1}^d |v_i|^2$ for $v = \sum_{i=1}^d v_i e_i \in V$. 
Let $\bb G = \textup{GL}(V)$ be the general linear group of $V$.
The action of $g \in \bb G$ on a vector $v \in V$ is denoted by $gv$.
For any $g \in \bb G$, let $\| g \| = \sup_{v \in V \setminus \{0\} } \frac{\| g v \|}{\|v\|}$.
We write $\log_+ t = \max\{ \log t, 0\}$ for $t >0$, and $a \wedge b = \min \{ a, b \}$ for $a, b \in \bb R$. 
For any Borel set $B \subset \bb R$, denote by $B^{\circ}$ its interior. 
All over the paper, $c, C$  possibly supplied with indices will denote positive constants depending on their indices.

Let  $\bb P(V)$ be the projective space of $V$ 
equipped with the angular distance $\bf d$ defined by 
$\bf d(x, x') = \sqrt{1 - \frac{|\langle v, v' \rangle|^2}{ \|v\|^2 \|v'\|^2 }}$
for $x = \bb R v$ and $x' = \bb R v'$ with $v, v' \in \bb R^d \setminus \{0\}$. 
The action of $g\in \bb G$ on $x = \bb R v \in \bb P(V)$ is defined by $g\cdot x = \bb R gv \in \bb P(V)$.
For $g \in \bb G$ and $x = \bb R v \in \bb P(V)$ with $v \in V \setminus \{0\}$,
introduce the norm cocycle
\begin{align}\label{Def-dot-cocycle}
\sigma \, :\, \bb G \times \bb P(V) \to \bb R, \qquad \sigma(g, x) = \log \frac{\|gv\|}{\|v\|}. 
\end{align}

To define a branching random walk on the group $\bb G$, 
we assume that on  the probability space  $(\Omega,\mathscr F, \bb P)$ 
we are given a point process $\mathscr N$ on $\bb G$, that is a random counting measure on the Borel subsets of $\bb G$. 
At the time $0$ there is one root ancestor
to which we assign the identity matrix of the group $\bb G$. 
In the first generation the root ancestor gives birth to $\mathscr N(\bb G)$ children to which we assign 
the collection of matrices given by the point process $\mathscr N$. 
Each particle in the first generation gives birth to new particles with corresponding matrices 
according to the independent copies of the same point process $\mathscr N$, which form
the second generation. The system goes on according to the same mechanism. 
The genealogy of these particles forms a Galton-Watson tree $\bb T$ with the root $\emptyset$. We use Ulam-Harris notation, identifying $\bb T$ as a random subset of $\bigcup_{n=0}^\infty \bb N^n$. Hence, each node $u\in \bb T$ of generation $n$ will be identified with an $n$-tuple  $u=(u_1, \ldots, u_n)$, where $u_i\in \bb N$. For notational simplicity, we write $u=u_1\dots u_n$ and further denote by $u|k:=u_1 \dots u_k$ the restriction of $u$ to its first $k$ components. Given $u=u_1 \dots u_n$ and $v=v_1 \dots v_m$, we write $uv=u_1 \dots u_nv_1 \dots v_m$ for the concatenation. 
For a node $u$ in $\bb T$ by $|u|$ we denote its generation.

From the construction given above, to each node $u\in \bb T$ 
 corresponds a random element $g_u \in \bb G$. 
The branching random walk $(G_u)_{u\in \bb T}$  on the group $\bb G$
is defined then by taking the left product of the elements 
along the branch leading from $\emptyset$ to $u \in \bb T$:  
\begin{align*}
 G_u := g_u g_{u|n-1}\ldots g_{u|2} g_{u|1}, \ u\in \bb T.
\end{align*} 
Note that these factors are independent, but not necessarily identically distributed. 
There is a natural filtration given by $\scr{F}_n:= \sigma\big( \{ g_u \, : \, |u| \leq n\}\big)$. 
Obviously, the point process $\mathscr N$ can be written as $\mathscr N = \sum_{|u|=1}\delta_{ G_u }$. 
We denote by $N=\mathscr N(\bb G)$ the number of particles in the first generation as well as by $N_u$ the number of children of the particle at node $u \in \bb T$.

We define the following shift operator which will help us to  use the branching property. For any function $F=F\big((g_u)_{u \in \bb T} \big)$ of the branching process and any node $w \in \bb T$, we define 
\begin{equation}\label{eq:shift-operator}
[F]_w=F \big( (g_{wu})_{u \in \bb T_w} \big), 
\end{equation}
 where $\bb T_w$ denotes the random subtree rooted at $w$, its first generation being formed by the children of the particle at $w$. That is, $[F]_w$ is evaluated on the subtree started at $w$.
For example, $[G_u]_w$ means the product 
$g_{wu} g_{wu|n-1} \ldots g_{wu|2} g_{wu|1}$ for $u=u_1 u_2 \ldots u_n \in \bb T_{w}$, 
i.e. the product of the elements of $\bb G$ along the branch leading from $w$ to $wu$.

Denote the event corresponding to the system's survival after $n$ generations by
$\mathscr S_n   = \{ \sum_{|u|=n} 1 >0 \}.$ 
Then $\mathscr S = \cap_{n = 1}^{\infty} \mathscr S_n$ is the event corresponding to the system's ultimate survival.
We say that a sequence of random variables $(\xi_n)_{n\geq 1}$ converges in probability to the random variable $\xi$
under the system's survival if  
$\lim_{n\to\infty} \bb P( |  \xi_n - \xi | \geq \ee | \mathscr S) =0$ for any $\ee >0$. 
The convergence almost surely under the system's survival is defined in the same way: 
$\bb P( \lim_{n\to\infty} \xi_n = \xi  | \mathscr S) =1$.

For a non-zero vector  $v \in \bb R^d$, our goal is to study some extremal properties of the branching random 
walk $(G_u v )_{u\in \bb T}$ on $\bb R^d$ defined by the action of $G_u$ on $v$. 
We shall do it jointly with the  walk $(G_u \cdot  x)_{u\in \bb T}$ on $\bb P(V)$ defined by the action of $G_u$ 
on the projective element $x = \bb R v \in \bb P(V)$. 
To be more precise, we shall study the pair
\begin{align} \label{Def_Xnxu_Snxu}
X_u^x: = G_u \cdot x  = \bb R  G_u  v \in \bb P(V), 
	\quad 
S_u^x: = \sigma(G_u, x) = \log \frac{ \|G_u v\|}{\|v\|} \in \bb R, \quad u\in \bb T.
\end{align}
Given a Borel set $A \subseteq \bb P(V)$, 
 define respectively the maximal and minimal positions of $S_u^x$ provided 
that direction $X_u^x$ belongs to the set $A$ for nodes $u$ in the $n$-th generation  
by setting 
\begin{align} \label{MinPosition001A}
	M_n^x(A): = \sup \left\{ S_u^x:  \,  X_u^x \in A, \, |u| = n \right\},
	\qquad
	m_n^x(A): = \inf \left\{ S_u^x:   \,  X_u^x \in A, \, |u| = n \right\}, 
\end{align}
where $\sup \emptyset =-\infty$ and $\inf \emptyset = + \infty$. 
For conditions on $A$ under which the sets in \eqref{MinPosition001A} become nonempty eventually for large $n$,
we refer to Proposition \ref{Prop-SLLN-direction} and to the discussion before Theorem \ref{Thm_LLN-001}. 
The global maximal and minimal displacements 
in the $n$-th generation are defined by
\begin{align*} 
M_n^x: = M_n^x(\bb P(V)) = \sup \left\{ S_u^x: |u| = n  \right\},  
\qquad m_n^x: = m_n^x(\bb P(V)) =  \inf \left\{ S_u^x: |u| = n  \right\},   
\end{align*}
with the same convention $\sup \emptyset =-\infty$ and $\inf \emptyset = + \infty$.

Let $\mathbb E$ be the expectation corresponding to the probability measure $\mathbb P$. 
We need the following moment condition on $N$,
which implies that the branching random walk is in the supercritical regime.

\begin{conditionA}\label{Condi_N}
	There exists a constant $\delta>0$ such that $1 < \bb E N$ and  $\bb E (N^{1 + \delta}) <\infty$. 
\end{conditionA} 
This condition implies $\bb P(\scr S)>0$ for the survival event $\scr{S}$ and allows us to define a probability measure $\mu$ on $\bb G$ by\footnote{Note that since $\norm{\cdot}$ is submultiplicative, trying to avoid assumption \ref{Condi_N} by defining
	$$ \mu(B) ~=~ \frac{\bb E \big[ \sum_{|u|=1} \norm{G_u}^\theta \mathds 1_B(G_u)\big]}{\bb E \big[ \sum_{|u|=1} \norm{G_u}^\theta \big]}$$ for some $\theta>0$  does not lead to a convolution-stable definition of a measure $\mu$.}
\begin{align}\label{Def-mu}
\mu(B) := \frac{1}{\bb E N}\bb E \Big[ \sum_{|u|=1} \mathds 1_B(G_u) \Big], 
\end{align}
for any Borel measurable set $B \subseteq \bb G$.  
Our further assumptions are formulated in terms of the measure $\mu$. 
As announced before, we will work with a density assumption on the law of the matrices.

\begin{conditionA}\label{Condi-density}
The measure $\mu$ has a density $\dot{\mu}$ with respect to the Haar measure on $\bb G$ and $\dot{\mu}$ is bounded or continuous. 
Moreover,  there exist constants $c_0, \eta_0  > 0$ such that $\inf_{x \in \bb P(V)} \mu \{ g \in \bb G:  \sigma(g, x) > c_0 \} >0$ 
and 
\begin{align} \label{Moment001}
 \int_{\bb G} \max \big\{ \|g\|, \|g^{-1}\| \big\}^{\eta_0} \left( 1 + \|g\|^d   | \textup{det}(g^{-1})|  \right)  \mu(dg)< \infty.
\end{align}
\end{conditionA}

The condition $\inf_{x \in \bb P(V)} \mu \{ g \in \bb G:  \sigma(g, x) > c_0 \} >0$ 
is required to ensure that a certain harmonic function $V_s(x, y)$ (cf.\ Proposition \ref{Prop-harmonic}) 
is strictly positive for any $x \in \bb P(V)$ and $y \geq 0$. 

Let $(g_i)_{i \geq 1}$ be a sequence of i.i.d.\
random elements on $\bb G$ with the law $\mu$ (defined by \eqref{Def-mu})  and denote by 
\begin{align*}
	G_n := g_n \ldots g_1,   \quad  n\geq 1, 
\end{align*}
their left product. Our next assumption can be understood as an irreducibility assumption on the action of $G_n$ on $\bb P(V)$.

\begin{conditionA}\label{Condi-density-invariant}
For all $x \in \bb P(V)$ and all non-empty open sets $O \subset \bb P(V)$ 
there is $n \in \bb N$ such that $\bb P(G_n \cdot x \in O) >0$. 
\end{conditionA}

Condition \ref{Condi-density} is obviously satisfied when $\mu$ has a continuous density with compact support on $\bb G$.  Condition \ref{Condi-density-invariant} is satisfied for example when the support of $\mu$ contains an open subset of 
the special orthogonal group $SO(d,\bb R)$ in which case it is also possible for $\mu$ to have a compact support in $\bb G$. 
In fact, if the support of $\mu$ contains an open subset of $SO(d,\bb R)$, 
then by \cite[Lemma 2.1]{Bhatt72}, there exists $n\in \bb N$ such that the support of the $n$-fold convolution of $\mu$ contains the whole group $SO(d,\bb R)$, which means that the action of $G_n$ is transitive on $\bb P(V)$. 
Another example where  condition \ref{Condi-density-invariant} is satisfied is when the density of $\mu$ from \ref{Condi-density} is strictly positive on the whole group $\bb G$.

First let us note that, under conditions \ref{Condi-density} and \ref{Condi-density-invariant}, 
there exists a unique $\mu$-invariant probability measure $\nu$ on $\bb P(V)$ such that $\mu * \nu = \nu$ 
(this is a consequence of Proposition \ref{prop:Ps} stated below).
In fact, this conclusion holds under a more general condition (i-p) stated in Section \ref{Sec-SP-perturbed}; 
we refer to  \cite{BL85, GL16} for details.

Conditions \ref{Condi-density} and \ref{Condi-density-invariant} 
 play an important role in the paper. Besides being essential for the existence of a dual random walk, 
they ensure spectral gap properties of the transfer operators $P_s$, which we are going to define now. 
Set 
\begin{align}
& I_{\mu}^+  = \left\{ s \geq 0: \int_{\bb G} \norm{g}^s  \left( 1 + \|g\|^d   | \textup{det}(g^{-1})|  \right) \mu(dg)< \infty  \right\},  
 \label{def-I-mu-plus}  \\
& I_{\mu}^- = \left\{ s \leq 0:   \int_{\mathbb G}  \|g^{-1}\|^{-s}  
  \left( 1 + \|g\|^{d}   | \textup{det}(g^{-1})|  \right)  \mu(dg) < \infty \right\}.  \label{def-I-mu-minus}
\end{align}
By \ref{Condi-density} and  H\"{o}lder's inequality, 
 both $I_{\mu}^+$ and $I_{\mu}^-$ are non-empty intervals of $\bb R$. 
Let $\scr B$ be the Banach space of real-valued continuous functions on $\bb P(V)$
endowed with the supremum norm $\| \cdot \|_{\scr B}$.
For any $s \in I_{\mu} := I_{\mu}^+ \cup I_{\mu}^-$, define the transfer operator $P_s$ as follows: 
for $\varphi \in \scr B$ and $x \in \bb P(V)$,
\begin{align}\label{Def-Ps}
P_s \varphi(x) 
:= \int_{\bb G}  e^{s \sigma(g, x)} \varphi(g \cdot x) \mu(dg) 
= \frac{1}{\bb E (N)} \bb E \Big( \sum_{|u| = 1}  e^{s \sigma (G_u, x)} \varphi(G_u \cdot x) \Big),
\end{align}
where the second equality follows directly from the definition of $\mu$. 
By Proposition \ref{prop:Ps} below, for any $s \in I_{\mu}$, 
there exists a unique dominant eigenvalue $\kappa(s)>0$ of $P_s$ such that
the function $s \mapsto \kappa(s)$ is analytic on $I_{\mu}^\circ$ (the interior of $I_{\mu}$).
Moreover, it will be shown in Lemma \ref{lem:strict_convex}
that the  function $s \mapsto \log \kappa(s)$ is strictly convex on $I_{\mu}^{\circ}$.

For a branching random walk on the real line, many of its properties are encoded in the (log-) Laplace transform of the intensity measure of the point process. The corresponding object in our setting is the function $\mathfrak{m}$ defined by 
\begin{align}\label{Def-m-s}
\mathfrak m(s) =  \kappa (s)  \bb E N  \quad \mbox{and} \quad  \Mf(s) =\log \mathfrak{m}(s),  \quad  s \in I_{\mu}. 
\end{align}
The next condition is the analogue of the {\em boundary-case} condition (cf. \cite{BK05,HS09}) imposed for branching random walks on the real line in order to study second-order asymptotic terms of the maximal positions $M_n^x(A)$ and $M_n^x$. 
\begin{conditionA}\label{Condi_ms}
There exists a constant $\alpha \in (I_{\mu}^{+})^{\circ}$ with $\alpha >0$ such that 
$\mathfrak m (\alpha)  = 1$ and $\mathfrak m '(\alpha) = 0$. 
In addition, 
$$\bb E \Big(\sum_{|u|=1}  \|G_u\|^{\alpha} \big( 1+\log_+ \|G_u\| + \log_+ \|G_u^{-1}\| \big) \Big)^2 < \infty.$$ 
\end{conditionA}

The boundary condition \ref{Condi_ms} guarantees 
that, under some changed measure, the behaviour of our branching random walk is similar to that of 
a left product of random matrices whose Lyapunov exponent is $0$. 
The corresponding change of measure and a many-to-one
formula will be introduced in Section \ref{sec-LLN}.

In the one-dimensional case, the study of the minimal position of a branching random walk is equivalent to studying the maximal position of the same walk with negative increments. This simple relation is not true in our matrix-valued setting, where the maximal position is related to $\| G_u \|$, while the minimal position relates to $\| G_u^{-1} \|$. 
Hence, to study the second-order asymptotic terms of the minimal positions $m_n^x(A)$ and $m_n^x$, 
we need a separate boundary condition.

\begin{conditionA}\label{Condi_Neg_alpha}
There exists a constant $-\beta \in (I_{\mu}^{-})^{\circ}$ with $\beta >0$
such that $\mathfrak m (-\beta)  = 1$ and $\mathfrak m'(-\beta) = 0$.  
In addition, 
$$\bb E \Big(\sum_{|u|=1}  \|G_u^{-1}\|^{-\beta} \big( 1+\log_+ \|G_u\| + \log_+ \|G_u^{-1}\| \big) \Big)^2 < \infty.$$
\end{conditionA}

We will show below that - as in the one-dimensional case - many branching random walks on $\bb G$ 
can be transformed into the boundary case.

We finish with two moment assumptions of $Z \log Z$-type, which will only appear in Theorem \ref{Thm_LLN-001}. Denote $Y_s=\sum_{|u| = 1}  \| G_u \|^s$.  

\begin{conditionA}\label{Condition_Extra}
For any $s \in I_{\mu}^+$, it holds $\bb E (Y_s \log_+ Y_s ) < \infty.$
\end{conditionA}

\begin{conditionA}\label{Condition_Extra-mini}
For any $s \in I_{\mu}^-$, it holds $\bb E ( Y_s \log_+ Y_s ) < \infty.$ 
\end{conditionA}

Note that if $\mu$ has a compact support on $\bb G$, then conditions \ref{Condition_Extra} and \ref{Condition_Extra-mini}
are satisfied due to condition \ref{Condi_N}.

\subsection{Results for maximal and minimal positions}

Now we are ready to state the main results of the paper. Our first result is a law of large numbers for maximal and minimal positions of our branching random walk.  
Note that, in Proposition \ref{Prop-SLLN-direction} below, a strong law of large numbers for 
the counting measure of a branching random walk on $\bb G$ is established, from which
it follows that, for each Borel set $A \subseteq \bb P(V)$ satisfying $\nu(A) > 0$ and $\nu(\partial A) = 0$, 
there exists $n_0: = n_0(A) \geq 1$ such that in all generations after $n_0$,
there are particles $u$ such that $X_u^x$ is in $A$,
so that both $M_n^x(A)$ and $m_n^x(A)$ (cf.\ \eqref{MinPosition001A}) are real-valued.

\begin{theorem}\label{Thm_LLN-001}
Assume conditions \ref{Condi_N},  \ref{Condi-density}, \ref{Condi-density-invariant} and \ref{Condition_Extra}. 
Assume there is $s^* \in (I_\mu^+)^\circ$ such that $s^*\Mf'(s^*)= \Mf(s^*)$.  
Then, for any $x \in \bb P(V)$ and any Borel set $A \subseteq \bb P(V)$ satisfying $\nu(A) > 0$ and $\nu(\partial A) = 0$,  
conditionally on the system's survival, 
\begin{align}\label{eq_LLN-001}
\lim_{n \to \infty}  \frac{M_n^x(A)}{n} = \Mf'(s^*)  \quad  \bb P\mbox{-almost surely}. 
\end{align}
Similarly, assume conditions \ref{Condi_N},  \ref{Condi-density}, \ref{Condi-density-invariant} and  \ref{Condition_Extra-mini}.  
Assume there is $s^* \in (I_\mu^-)^\circ$ such that $s^*\Mf'(s^*)= \Mf(s^*)$. 
Then, for any $x \in \bb P(V)$ and any Borel set $A \subseteq \bb P(V)$ satisfying $\nu(A) > 0$ and $\nu(\partial A) = 0$,  
conditionally on the system's survival, 
\begin{align}\label{eq_LLN-001mini}
\lim_{n \to \infty}  \frac{m_n^x(A)}{n} = \Mf'(s^*)  \quad  \bb P\mbox{-almost surely}. 
\end{align}
\end{theorem}

From the proof of Theorem \ref{Thm_LLN-001} 
one has $\Mf'(s^*) = \inf_{s \in I_{\mu}^+} \frac{ \mathfrak M(s) }{s}$ in \eqref{eq_LLN-001},
and $\Mf'(s^*) =  \sup_{s \in I_{\mu}^-} \frac{ \mathfrak M(s) }{s}$ in \eqref{eq_LLN-001mini}. 

For the proof of \eqref{eq_LLN-001} and \eqref{eq_LLN-001mini}, we make use of the precise 
large deviation result for the counting measure which has been established recently in Bui, Grama and Liu \cite{BGL20}. 
We believe that if we do not account for the direction, 
with some extra effort conditions \ref{Condition_Extra} and \ref{Condition_Extra-mini}  can be removed, 
but we are refraining from doing so as our main goal is a second order asymptotic 
of the extreme position formulated below. 

The condition $s^*\Mf'(s^*)= \Mf(s^*)$ is the direct analogue of the condition which is required to transform a branching random walk on the real line into the boundary case, see \cite[Section 5.1]{Shi12} for a discussion. We can still deduce a strong law of large numbers for the maximal position without it (see Remark \ref{rem:linear speed extremal position}), but it will be required anyway for the study of second-order asymptotics. 
The following result shows that indeed any branching random walk $(\widetilde{G}_u)_{u \in \bb T}$ on $\bb G$ 
that satisfies $s^*\widetilde{\Mf}'(s^*)= \widetilde{\Mf}(s^*)$ can be converted to satisfy the boundary condition. 
For convenience, below we supply all quantities corresponding to 
the branching random walk $(\widetilde{G}_u)_{u \in \bb T}$ before the transformation with a tilde. 
	
\begin{lemma}\label{lem:transform}
Assume that there is $s^* \in I_\mu^\circ$ such that $s^*\widetilde{\Mf}'(s^*)= \widetilde{\Mf}(s^*)$.
Let $(\widetilde{G}_u)_{|u|=1}$ be the collection of matrices in the first generation.
Then the branching random walk on $\bb G$ with matrices 
$G_u := \exp\big(-{ \widetilde{\Mf} '(s^*)}\big) \widetilde{G}_u$ 
satisfies $\mathfrak m (s^*)  = 1$ and $\mathfrak m'(s^*) = 0$ (cf.\ condition \ref{Condi_ms}),
where $\mathfrak m$  corresponding to the transformed process $(G_u)_{u \in \bb T}$ is defined by \eqref{Def-m-s}. 
\end{lemma}

Recalling \eqref{Def_Xnxu_Snxu}, 
we see that we transform ${\widetilde{S}_u^x}: = \sigma (\widetilde{G}_u, x)$ by subtracting 
the drift $\widetilde{\Mf}'(s^*)$ of the maximal (or minimal) position, 
so that for the transformed process it holds $M_n^x(A)/n \to 0$, $\bb P$-almost surely. 
This motivates our next two theorems 
where we assume the boundary condition and normalize $M_n^x(A)$ and $m_n^x(A)$ by $\log n$.

\begin{theorem}\label{Thm_In_Pro}
Assume conditions \ref{Condi_N},  \ref{Condi-density}, \ref{Condi-density-invariant} and \ref{Condi_ms}. 
Then, for any $x \in \bb P(V)$ and any Borel set $A \subseteq \bb P(V)$ satisfying $\nu(A) > 0$ and $\nu(\partial A) = 0$, 
conditionally on the system's survival, 
\begin{align}\label{thm1_Minimal_aa_SetA}
\lim_{n \to \infty}  \frac{M_n^x(A)}{\log n} = - \frac{3}{2 \alpha}  \quad  \mbox{in probability} \ \bb P. 
\end{align}
In particular,  for any $x \in \bb P(V)$, conditionally on the system's survival, 
\begin{align}\label{thm1_Minimal_aa}
\lim_{n \to \infty}  \frac{M_n^x}{\log n} = - \frac{3}{2 \alpha}  \quad  \mbox{in probability} \ \bb P. 
\end{align}
\end{theorem}
As mentioned in Section \ref{Sec-notation}, 
conditions \ref{Condi-density} and \ref{Condi-density-invariant}
are satisfied for example when the measure $\mu$ has a compact support with a continuous density around the identity matrix.

Clearly the result \eqref{thm1_Minimal_aa} follows from \eqref{thm1_Minimal_aa_SetA} by taking $A = \bb P(V)$.
It is worth mentioning that, even though the limits are the same, the first assertion \eqref{thm1_Minimal_aa_SetA} 
is much stronger than the second one, since it shows what is the limit behaviour of the  maximal position
$\sigma (G_u, x)$ over the subset of particles with the direction $G_u \!\cdot\! x \in A$. 
We also note that convergence in probability cannot be sharpened to an almost sure convergence,  
which will be considered in a forthcoming work.

Our second result proves similar properties for the minimal position $m_n^x$.  

\begin{theorem}\label{Thm_In_Pro_min}
Assume conditions \ref{Condi_N},  \ref{Condi-density}, \ref{Condi-density-invariant} and \ref{Condi_Neg_alpha}. 
Then, for any $x \in \bb P(V)$ and any Borel set $A \subseteq \bb P(V)$ satisfying $\nu(A) > 0$  and $\nu(\partial A) = 0$,
conditionally on the system's survival,  
\begin{align}\label{thm1_Minimal_aa_minSetA}
\lim_{n \to \infty}  \frac{m_n^x(A)}{\log n} =  \frac{3}{2 \beta}  \quad  \mbox{in probability}\ \bb P. 
\end{align}
In particular, for any $x \in \bb P(V)$, conditionally on the system's survival,
\begin{align}\label{thm1_Minimal_aa_min}
\lim_{n \to \infty} \frac{m_n^x}{\log n} =  \frac{3}{2 \beta}  \quad  \mbox{in probability}\ \bb P. 
\end{align}
\end{theorem}

Note that the limit in  \eqref{thm1_Minimal_aa_minSetA} and \eqref{thm1_Minimal_aa_min} is strictly positive, 
while in the case of the maximal position the limit in \eqref{thm1_Minimal_aa_SetA} and \eqref{thm1_Minimal_aa} 
is strictly negative, since both $\beta>0$ and $\alpha>0$, see \ref{Condi_Neg_alpha} and \ref{Condi_ms}, respectively.
The appearance of $\alpha$ and $\beta$ in the formulas \eqref{thm1_Minimal_aa_SetA}, 
\eqref{thm1_Minimal_aa},  \eqref{thm1_Minimal_aa_minSetA} and \eqref{thm1_Minimal_aa_min}
is due to the fact that by the means of Lemma \ref{lem:transform}, 
we cannot transform to the case where $\alpha=1$ ($\beta=1$), 
because the corresponding transformation from the one-dimensional setting 
(a scaling of $S_u^x=\log \frac{\norm{G_u v}}{ \norm{v} }$) 
would amount to taking powers of matrix norms. 
This transformation has no interpretation on the underlying product of random matrices.

We complement the above results by 
 proving that the asymptotic behaviour of maximal (minimal) value of the logarithm of a given coefficient,  operator norm or  spectral radius is similar to that of maximal (minimal) value of the logarithm of the vector norm. 
Let $V^*$ be the dual vector space of $V$, i.e.\ the space of linear forms on $V$.  
For $v \in V$ and $f \in V^*$ denote by $\langle f, v \rangle=f(v)$  the duality bracket.  
Let $\norm{g}$ and $\rho(g)$ be respectively the operator norm and the spectral radius of the matrix $g\in \bb G$. 
Then we have the following results.
Let $v \in V$, $f \in V^*$ and $F_u$ be one of $|\langle f, G_u v \rangle|$, $\norm{G_u}$ or $\rho(G_u)$.
Theorems \ref{Thm_LLN-001}, \ref{Thm_In_Pro} and \ref{Thm_In_Pro_min} hold with $M_n^x$ or $m_n^x$ 
replaced by $\max_{|u|=n} \log F_u$ or $\min_{|u|=n} \log F_u$, respectively. 
We will only give precise formulations of the results corresponding to Theorems  \ref{Thm_In_Pro} and \ref{Thm_In_Pro_min}. 

\begin{theorem}\label{Thm_In_Pro_Coeff}
Assume conditions \ref{Condi_N},  \ref{Condi-density}, \ref{Condi-density-invariant} and \ref{Condi_ms}.
Let $A \subseteq \bb P(V)$ be any  Borel set satisfying $\nu(A) > 0$  and $\nu(\partial A) = 0$.  
Let $v \in V$, $f \in V^*$ and $F_u$ be one of $|\langle f, G_u v \rangle|$, $\norm{G_u}$ or $\rho(G_u)$. Then, conditionally on the system's survival,
\begin{align}\label{thm1_Minimal_aa_Coeff}  
\lim_{n \to \infty}  \frac{\sup_{G_u\cdot x\in A, | u | = n} \log F_u}{\log n}  
= - \frac{3}{2 \alpha}  \quad  \mbox{in probability} \ \bb P. 
\end{align} 
Similarly, under conditions \ref{Condi_N},  \ref{Condi-density}, \ref{Condi-density-invariant} and \ref{Condi_Neg_alpha}, 
conditionally on the system's survival, 
\begin{equation}
\lim_{n \to \infty}  \frac{\inf_{G_u\cdot x\in A, | u | = n} \log F_u}{\log n}
=  \frac{3}{2 \beta}  \quad  \mbox{in probability} \ \bb P.
\end{equation}
\end{theorem}

\section{Preliminaries} \label{sec-Prelim}

In this section we discuss spectral gap properties of the transfer operator $P_s$ and their application to a change of measure and a many-to-one formula.

\subsection{Spectral gap property of the transfer operator $P_s$}\label{subsect:spectral gap}

We are going to prove Proposition \ref{prop:Ps}, i.e., that the transfer operator $P_{s}$ has spectral gap properties 
for all $s \in I_{\mu} = I_{\mu}^+ \cup I_{\mu}^-$.  For positive $s$, these properties have been proved under various assumptions, see \cite{BM16} for a survey. So the main burden is to extend these results to the case of negative $s$. Recall that $\scr{B}$ denotes the space of real-valued continuous functions on $\bb P(V)$, equipped with the 
supremum norm $\| \cdot \|_{\scr B}$. 

\begin{lemma}\label{lem:compact}
Assume \ref{Condi-density}. Let $s \in I_\mu$. Then the operator $P_s : \scr{B} \to \scr{B}$ is a compact operator.
\end{lemma}

\begin{proof} 
We give the proof for negative $s$ ($s \in I_{\mu}^-$), 
the proof for positive $s$ ($s \in I_{\mu}^+$) being similar but easier. We have to show that 
the set $A: = \{ P_s \varphi: \|\varphi\|_{\scr B} \leq 1 \}$ is a compact subset of $\mathscr B$. 
To this end, by the theorem of Arzela-Ascoli, 
it is enough to prove that the set $A$ is uniformly bounded and equicontinuous. 
Note that, by \eqref{def-I-mu-minus}, for any $s \in I_{\mu}^-$, 
it holds that $\int_{\bb G}   e^{ |s| \log \| g^{-1}\| }  \mu(dg) < \infty.$
Since $\sigma(g, x) \geq \log \|g^{-1}\|^{-1}$, we have, for any $s \in I_{\mu}^-$, 
\begin{align*} 
\|P_s \varphi \|_{\mathscr B}  \leq  \int_{\bb G}   e^{ |s| \log \| g^{-1}\| }  \mu(dg), 
\end{align*}
uniformly in $\|\varphi\|_{\scr B} \leq 1$, so that the set $A$ is uniformly bounded.   
Next we prove that the set $A$ is equicontinuous, i.e. we show that 
for given $\ee >0$, there is $\delta>0$ such that $|P_s \varphi (x_1) - P_s \varphi (x_2)| < \ee$
whenever $\bf d (x_1, x_2) \leq \delta$ and $\varphi \in \scr B$ with $\| \varphi \|_{\scr B} \leq 1$. 
Since $s \in I_{\mu}^-$, there exists a constant $B: = B(\ee) >0$ such that, 
with $E=\{g \in \bb G: |\log \| g^{-1}\| | \leq B  \}$, 
\begin{align} \label{bound-Ec-dg}
\int_{E^c}   e^{ - s \log \| g^{-1}\| } \mu(dg) \leq \frac{\ee}{8}. 
\end{align}
By assumption \ref{Condi-density}, the measure $\mu$ has a density, so that $\mu(dg) = \overset{\cdot} \mu(g) dg$. 
Since $\overset{\cdot} \mu \in L^1$, there exists a continuous function $\overset{\cdot} \mu_E$ 
supported on $E$ such that
\begin{align}\label{bound-Ec-dg-002}
\int_{E} e^{ - s \log \| g^{-1}\| }  \left|  \overset{\cdot} \mu_E(g) - \overset{\cdot} \mu(g) \right| dg \leq \frac{\ee}{8}. 
\end{align} 
By \eqref{bound-Ec-dg}, \eqref{bound-Ec-dg-002} and the fact that $\sigma(g, x) \geq \log \|g^{-1}\|^{-1}$,
we get that for  $\varphi \in \scr B$ with $\| \varphi \|_{\scr B} \leq 1$, 
\begin{align}\label{Density_SP001}
|P_s \varphi (x_1) - P_s \varphi (x_2)|   
& = \left|  \int_{\bb G}  e^{s \sigma(g,  x_1)} \varphi(g \cdot x_1) \mu(dg) 
-   \int_{\bb G}  e^{s \sigma(g,  x_2)} \varphi(g \cdot  x_2))   \mu(dg)  \right|   \notag\\
& \leq \left|  \int_{E}  e^{s \sigma(g,  x_1)} \varphi(g \cdot x_1) \overset{\cdot} \mu(g) dg
-   \int_{E}  e^{s \sigma(g,  x_2)} \varphi(g \cdot  x_2))   \overset{\cdot} \mu(g) dg  \right| + \frac{\ee}{4} \notag\\
& \leq \left|  \int_{E}  e^{s \sigma(g,  x_1)} \varphi(g \cdot x_1) \overset{\cdot} \mu_E(g) dg
-   \int_{E}  e^{s \sigma(g,  x_2)} \varphi(g \cdot  x_2))   \overset{\cdot} \mu_E(g) dg  \right| + \frac{\ee}{2}
=: J + \frac{\ee}{2}. 
\end{align}
Note that 
there is a continuous map: $x \in  \bb P(V) \mapsto k(x) \in SO(d, \bb R)$ such that $x= \bb R  k (x) e_1$. 
Moreover, the set $E$ is invariant under the map $l: g \mapsto g l$ for any $l \in SO(d, \bb R)$. 
Hence, for any $x_1, x_2 \in \bb P(V)$, with $k_1 = k(x_1)$ and $k_2=k(x_2)$, 
\begin{align}\label{Density_SP001-aa}
J & =  \left| \int_{E}   e^{s \log \|g k_1 e_1\|} \varphi( \bb R g k_1 e_1 )  \overset{\cdot} \mu_E(g) dg 
-  \int_{E}   e^{s \log \|g k_2 e_1\|}  \varphi( \bb R g k_2 e_1 )   \overset{\cdot} \mu_E(g) dg   \right|    \notag\\
& =  \left| \int_{E}   e^{s \log \|g e_1\|} \varphi( \bb R g e_1 )  \overset{\cdot} \mu_E(g k_1^{-1}) dg 
-  \int_{E}   e^{s \log \|g e_1\|} \varphi( \bb R g e_1 )   \overset{\cdot} \mu_E(g k_2^{-1}) dg   \right|    \notag\\
& \leq  \int_{E}   e^{s \log \|g e_1\|}  |\varphi( \bb R g e_1 )| 
\left| \overset{\cdot} \mu_E(g k_1^{-1}) - \overset{\cdot} \mu_E(g k_2^{-1})  \right| dg   \notag\\
&  \leq   \int_{E}   e^{ - s \log \| g^{-1}\| } 
\left| \overset{\cdot} \mu_E(g k_1^{-1}) - \overset{\cdot} \mu_E(g k_2^{-1})  \right| dg. 
\end{align}
Since the map $x \mapsto k(x)$ and the function $\overset{\cdot} \mu_E$ are continuous on a compact set, 
hence uniformly continuous, by Lebesgue dominated convergence theorem,
there exists $\delta >0$ such that, for any $x_1, x_2 \in \bb P(V)$ with $\bf d (x_1, x_2) \leq \delta$, 
the last integral in \eqref{Density_SP001-aa} is bounded by $\ee/2$, for any $\varphi$ satisfying $\| \varphi \|_{\scr B} \leq 1$. 
This shows that the set  of functions $A$ is equicontinuous when $\overset{\cdot} \mu \in L^1$. 
\end{proof}

Before proving Proposition \ref{prop:Ps}, we need one auxiliary lemma. 	
Recall that $(g_i)_{i \ge 1}$ denotes a sequence of i.i.d. random elements of $\bb G$ with law $\mu$. 

\begin{lemma}\label{lem:continuous-density}
Under condition \ref{Condi-density}, the law of $g_2 g_1$ has a continuous density. 
\end{lemma}

\begin{proof}
By condition \ref{Condi-density}, the density $\dot{\mu}$ of $\mu$ is continuous or bounded. 
In the first case, there is nothing to prove. 
If  $\dot{\mu}$  is bounded, then $\dot{\mu}  \in L^1 \cap L^\infty$ w.r.t.~the Haar measure on $\bb G$. Let $C = \norm{\dot{\mu}}_\infty$.
	
The density of $g_2 g_1$ is given by
$$ \dot{\mu}^{(*2)}(g)  = \int_{\bb G} \dot{\mu}(gh^{-1}) \dot{\mu}(h) dh.$$
Given $\epsilon >0$, let $\dot{\mu}_\epsilon$ be a continuous function with compact support such that 
$\norm{\dot{\mu}-\dot{\mu}_\epsilon}_1 < \frac\epsilon{3C}$. 
Since $\dot{\mu}_\epsilon$ is uniformly continuous, there is $\delta>0$ 
such that $|\dot{\mu}_\epsilon(gh^{-1})-\dot{\mu}_\epsilon(g'h^{-1})| < \frac\epsilon3$ whenever $\| g - g' \| \leq \delta$. 
Hence, for all such $g$ and $g'$, 
\begin{align*}
& |\dot{\mu}^{(*2)}(g) - \dot{\mu}^{(*2)}(g')| 
\leq \int_{\bb G} |\dot{\mu}(gh^{-1})  -  \dot{\mu}(g'h^{-1})| \dot{\mu}(h) dh \\
& \leq \int_{\bb G} |\dot{\mu}(gh^{-1})  -  \dot{\mu}_\epsilon(gh^{-1})| \dot{\mu}(h) dh + \int_{\bb G} |\dot{\mu}_\epsilon(gh^{-1})  -  \dot{\mu}_\epsilon(g'h^{-1})| \dot{\mu}(h) dh + \int_{\bb G} |\dot{\mu}_\epsilon(g'h^{-1})  -  \dot{\mu}(g'h^{-1})| \dot{\mu}(h) dh  
\notag\\
& \leq C \norm{\dot{\mu}-\dot{\mu}_\epsilon}_1 + \int_{\bb G} \frac{\epsilon}{3} \mu(dh) + C \norm{\dot{\mu}-\dot{\mu}_\epsilon}_1  < \frac\epsilon3 + \frac\epsilon3 + \frac\epsilon3 = \epsilon,  
\end{align*}
concluding the proof. 
\end{proof}

Now we state the spectral gap property of the operator $P_s$ for $s \in I_{\mu}$:

\begin{proposition}\label{prop:Ps}
Assume conditions \ref{Condi-density} and \ref{Condi-density-invariant}. 
\begin{enumerate}
\item
There is a unique dominant eigenvalue $\kappa(s)>0$ of $P_s$; any other eigenvalue $\lambda$ of $P_s$ is strictly smaller in absolute value, i.e., $|\lambda| < \kappa(s)$. The eigenspace corresponding to $\kappa(s)$ is one-dimensional. 

\item
There is a unique probability measure $\nu_s$ satisfying $P_s \nu_s = \kappa(s) \nu_s$; it has a strictly positive density with respect to the uniform probability measure on $\bb P(V)$. 

\item
There is a unique function $r_s$ satisfying $P_s r_s = \kappa(s) r_s$ and $\nu_s(r_s)=1$. 
The function $r_s$ is strictly positive on $\bb P(V)$. 

\item
The function $s \mapsto \kappa(s)$ is analytic on $I_{\mu}^\circ$.
\end{enumerate}
\end{proposition}

This is well-known \cite{GL16,BM16} for $s>0$; we provide the corresponding arguments for $s<0$ in Section \ref{sec-Prelim} below. Note that  the unique probability measure $\nu$ which is invariant for the action of $\mu$ on $\bb P(V)$
equals $\nu_0$.

\begin{proof}[Proof of Proposition \ref{prop:Ps}] 
To prove the announced properties of $P_s$, we make use of a minorization property, 
see Eq. \eqref{eq:minorizationa} below. 
It holds subject to condition \ref{Condi-density-invariant} and the fact that, 
due to Lemma \ref{lem:continuous-density}, 
we may assume that $\bb P(G_n \in \cdot)$ has an absolutely continuous component 
with respect to Lebesgue measure on $\bb R^{d^2}$, 
and further there is an open set on which the density is bounded from below by a constant. 
These properties allow us to make use of \cite[Lemma 2.1]{AM12} and its extension \cite[Lemma 4.1]{BDMM13}. 
These lemmata give that there exist $n \in \bb N$, a constant $c_0>0$, 
a compact subset $D \subset \bb G$ and a probability measure $\rho$ 
with a strictly positive continuous density w.r.t. Lebesgue measure on $\bb P(V)$ such that
\begin{align} \label{eq:minorizationa}
 \bb P(G_n \cdot x \in \cdot) 
 \geq \bb P(G_n \cdot x \in \cdot, G_n \in D) 
\geq c_0 \rho(\cdot) \qquad \text{for all } x \in \bb P(V).
\end{align}
By the compactness of $D$, we can infer that for any nonnegative function $\varphi \in \scr B$ 
and for some constant $C>0$,
\begin{align}\label{eq:minorization}
P_s^n \varphi(x) = \bb E \big[ e^{-s \sigma (G_n, x)} \varphi(G_n \cdot x) \big] 
 \geq C \bb E \big[ \varphi(G_n \cdot x) \big] 
 \geq c_0 C \int_{\bb P(V)} \varphi(y) \rho(dy) 
\end{align}
for all $x \in \bb P(V)$. 
It follows that the operator $P_s^n$ is strictly positive in the sense 
that for any nonnegative function $\varphi \in \scr{B}$ with $\varphi \neq 0$, 
$P_s^n \varphi(x)>0$ for all $x \in \bb P(V)$. 
By Theorems 9 and 12 in \cite{K59}, $\kappa(s)^n$ is the unique eigenvalue of maximal modulus of $P_s^n$ 
and it is algebraically simple. 
By Lemma \ref{lem:compact}, the spectrum of $P_s$ is pure point, 
so all other eigenvalues are strictly smaller than $\kappa(s)$. 
Lemma 5.4 in \cite{AM12} gives that the corresponding eigenfunction $r_s$ is positive on $\bb P(V)$. 
Since any eigenfunction of $P_s$ is also an eigenfunction of $P_s^n$, 
it follows that the eigenspace of $P_s$ corresponding to $\kappa(s)$ is one-dimensional as well.
	
Turning to $\nu_s$, the minorization property \eqref{eq:minorization} gives that $\kappa(s) \nu_s \geq c_0 C \rho$, hence its support equals $\bb P(V)$ and it has a component with a continuous density. In fact, $\nu_s$ has to be absolutely continuous, since applying $P_s$ will smooth any singular component of $\nu_s$.
	
The final assertion about analyticity of $\kappa(s)$ will be proved in the subsequent Lemma \ref{Lem_spectral_gap}.
\end{proof}

Now we are ready to prove a perturbation theorem for the operator $P_s$, showing in particular that $\kappa(s)$ is an analytic function of $s$ with nice properties.
Denote by $\mathscr{L(B, B)}$ 
the set of all bounded linear operators from $\mathscr{B}$ to $\mathscr{B}$
equipped with the operator norm
$\left\| \cdot \right\|_{\mathscr{B} \to \mathscr{B}}$.

\begin{lemma} \label{Lem_spectral_gap}
Under conditions \ref{Condi-density} and \ref{Condi-density-invariant}, we have for all $s \in I_\mu$ 
\begin{align}
P^n_{s} = \kappa(s)^{n} \Pi_{s} + N^{n}_{s},  
\quad  n \geq 1, \label{perturb001_aa}
\end{align}
where the function $s \mapsto \kappa(s): I_{\mu}^{\circ} \to \bb R$ is analytic, 
the mappings $s \mapsto \Pi_{s}:  I_{\mu}^{\circ} \to \mathscr{L(B, B)}$
and $s \mapsto N_{s}: I_{\mu}^{\circ} \to \mathscr{L(B, B)}$ are analytic
in the strong operator sense, 
$\Pi_{s}$ is a rank-one projection with 
$\Pi_{s}(\varphi)(x) = \nu_{s}(\varphi r_s)$ for any $\varphi \in \mathscr{B}$ and $x\in \bb P(V)$,
$\Pi_{s} N_{s} = N_{s} \Pi_{s} = 0$. 
Moreover, for any fixed integer $k \geq 0$, 
there exist constants $c>0$ and $0< a < 1$ such that for any compact set $K \subset I_{\mu}$,  
\begin{align}\label{OperatorNbound_aa} 
\sup_{s \in K}
\Big\| \frac{d^{k}}{ds^{k}} N^{n}_{s} \Big\|_{\mathscr{B}\rightarrow\mathscr{B}} 
\leq c a^n, \quad  n\geq 1.  
\end{align}
\end{lemma}

\begin{proof}
For any fixed $s \in I_\mu$, the decomposition \eqref{perturb001_aa} is just a reformulation 
of the spectral gap property of $P_s$, proved in Proposition \ref{prop:Ps}.  
The analyticity of the functions $\kappa(s)$, $\Pi_s$, $N_s$ as well as the bound \eqref{OperatorNbound_aa} 
will follow from the perturbation theorem \cite[Theorem III.8]{HL01}. 
We extend the definition of $P_s$ to complex-valued arguments $z$ with $\Re z \in I_\mu^\circ$ by setting
$$P_z \varphi(x) = \int_{\bb G}  e^{z \sigma(g, x)} \varphi(g \cdot x) \mu(dg)$$
for any $\varphi \in \scr B$ and $x \in \bb P(V)$. 
To apply the perturbation theorem, we have to prove that the mapping $z \mapsto P_z$ is weakly holomorphic. We repeat the short proof of \cite[Lemma 4.10]{BDMM13} to show that it remains true for $s<0$.

Let $\gamma$ be a closed path in the domain $\{z \in \bb C \, : \, \Re z \in I_\mu^\circ\}$. We show that for any $\varphi \in \scr{B}$ and any finite measure $\eta$ on $\bb P(V)$, it holds $\int_\gamma \int_{\bb P(V)} P_z \varphi \, d\eta dz =0$. In fact, 
\begin{align*}
\int_\gamma \int_{\bb P(V)} P_z \varphi \, d\eta dz 
 = \int_\gamma \int_{\bb P(V)} \int_{\bb G} e^{z \sigma(g,x)} \varphi(g \cdot x) \mu(dg) \ \eta(dx) \ dz 
 = \int_{\bb P(V)} \int_{\bb G} \bigg( \int_\gamma e^{z \sigma(g,x)} dz \bigg) \varphi(g \cdot x) \mu(dg) \ \eta(dx). 
\end{align*}
The innermost integral is zero since $\sigma(g,x)>0$ always 
and hence the mapping $z \mapsto e^{z \sigma(g,x)}$ is holomorphic. 
The change in the order of integration is allowed since condition \ref{Condi-density} provides a general upper bound. 
The property $\int_\gamma \int_{\bb P(V)} P_z \varphi \, d\eta dz =0$ 
together with the fact that $z \mapsto \int_{\bb P(V)} P_z \varphi d\eta$ 
is continuous shows that $P_z$ is weakly holomorphic. 
By Dunford's Theorem (see e.g. \cite[Theorem V.3.1]{Y80}), it is then strongly holomorphic.

Now we can apply \cite[Theorem III.8]{HL01} for a fixed parameter $s \in I_\mu^\circ$ to obtain the analyticity of the mappings in an open ball around $s$. Since we can do so for each $s \in I_\mu^\circ$, we obtain analyticity on the whole $I_\mu^\circ$; while the bound \eqref{OperatorNbound_aa} extends to compact subset by taking the maximal bound of a finite union of open sets.
\end{proof}

For any $s \in  I_{\mu}$ and $\varphi \in \scr B$, 
let 
\begin{align}\label{def-operator-Qs}
Q_s \varphi(x) = \frac{1}{\kappa(s) r_s(x)} P_s (\varphi r_s) (x),  \quad  x \in \bb P(V). 
\end{align} 
We have the following corollary of Proposition \ref{prop:Ps}:

\begin{corollary}\label{cor:Qs_spectral_gap}
For any $s \in I_\mu$, the operator $Q_s$ is compact and has the unique 
and algebraically simple dominant eigenvalue $1$, 
the corresponding eigenfunctions are constant and there is a unique $Q_s$-invariant probability measure given by 
$\pi_s(\varphi):= \nu_s(\varphi r_s)$ for any $\varphi \in \scr B$. 
It has a strictly positive density with respect to the uniform probability measure on $\bb P(V)$.
\end{corollary}

\begin{remark}\label{Rem-absolute-con}
By Corollary \ref{cor:Qs_spectral_gap}, for any $s \in I_{\mu}$, the measures $\nu$, $\nu_s$ and $\pi_s$ are all absolutely continuous with each other. 
\end{remark}

\subsection{A change of measure formula}
Recall that $(g_i)_{i \geq 1}$ denotes a sequence of i.i.d.\
random elements on $\bb G$ with the law $\mu$ and $G_n := g_n \ldots g_1$, $n\geq 1$,
their left product.
With any starting point $X_0=x = \bb R v  \in \bb P(V)$ and $S_0 = 0$, denote for $n \geq 1$, 
\begin{align}\label{Def_Xn}
X_n: = g_n \cdot X_{n-1} = \bb R G_n v  \quad \mbox{and} \quad   S_n: = \sum_{k = 1}^n \sigma(g_k, X_{k-1}),
\end{align}
where the cocycle $\sigma(\cdot, \cdot)$ is defined in \eqref{Def-dot-cocycle}. 
Then the sequence $(X_n, S_n)_{n \geq 0}$ constitutes a Markov random walk on $\bb P(V) \times \bb R$. 
Denote by $\bb P_x$ the probability measure on the canonical space  $(\bb P(V)\times \bb R)^{\bb N}$ 
induced by the Markov chain $(X_n, S_n)_{n \geq 0}$ with starting point $(X_0, S_0) = (x, 0)$. 
Let $\bb E_x$ be the corresponding expectation. 
For any $s \in I_{\mu}$, let $\kappa(s)$ and $r_s$ be the eigenvalue  and the eigenfunction  given by Proposition \ref{prop:Ps}. 
Since $\sigma(\cdot, \cdot)$ is a cocycle, 
one can check that the family of kernels  
\begin{align}\label{Def-kernel-qns}
q_{n}^{s}(x, G_n) : = \frac{ 1 }{\kappa(s)^{n}}  e^{s \sigma(G_n, x)}  \frac{r_{s}(G_n \cdot x)}{r_{s}(x)},  \quad  n\geq 1, 
\end{align}
satisfies the property: for any $x \in \bb P(V)$ and $n, m \geq 1$, 
\begin{align} \label{cocycle01}
q_{n}^{s}(x,  G_n)  q_{m}^{s}(G_n \cdot x, g_{n+m} \ldots g_{n+1} ) = q_{n+m}^{s}(x, G_{n+m}).
\end{align}
By \eqref{Def-Ps} and Proposition \ref{prop:Ps}, we have that
\begin{align}\label{def-Qsn}
\mathbb{Q}_{s,n}^x(dg_1, \ldots, dg_n) : = q_{n}^{s}(x,g_{n}\dots g_{1})\mu(dg_1)\dots\mu(dg_n),  \quad  n\geq 1, 
\end{align}
is a sequence of probability measures, which, by \eqref{cocycle01}, forms a projective system on $\bb G^{\bb N}$. 
Therefore, by the Kolmogorov extension theorem,
there is a unique probability measure  $\mathbb{Q}_s^x$ on $\bb G^{\bb N}$
with marginals $\mathbb{Q}_{s,n}^x$. 
Denote by $\mathbb{E}_{\mathbb{Q}_s^x}$ the corresponding expectation.
All over the paper we use the convention that under the measure $\bb Q_s^x$,
the Markov chain $(X_n)_{n \geq 0}$ defined by \eqref{Def_Xn} starts with the point $x \in \bb P(V)$.

By the definition of $\mathbb{Q}_s^x$, for any bounded measurable function $h: (\bb P(V) \times \mathbb R)^{n} \to \bb R$, 
the following change of measure formula holds: under condition \ref{Condi-density}, for any $s \in I_{\mu}$,
\begin{align} \label{Formu_ChangeMea}
 \frac{ 1 }{ \kappa(s)^{n} r_{s}(x) }
 \mathbb{E}_x \Big[  r_{s}(X_n)  e^{s S_n} 
 h\big( X_1, S_1, \dots, X_n, S_n \big) \Big]   
   = \mathbb{E}_{\mathbb{Q}_{s}^{x}} 
 \Big[ h \big( X_1, S_1, \dots, X_n, S_n  \big) \Big].  
\end{align}
Set, for $s \in I_{\mu}$, 
\begin{align}\label{def-Lambda-s-q}
\Lambda(s) = \log \kappa(s)  \quad \mbox{and}  \quad
q = \Lambda'(s) = \frac{\kappa'(s)}{\kappa(s)}. 
\end{align}
Following \cite{GL16},  one can verify that under conditions \ref{Condi-density} and \ref{Condi-density-invariant}, 
the strong law of large numbers holds: 
$\lim_{n\to\infty} \frac{S_n}{n} = q$, $\mathbb{Q}_{s}^{x}$-almost surely, for any $s \in I_{\mu}$.


\subsection{Spectral gap property of perturbed operators}\label{Sec-SP-perturbed}

For $s \in I_{\mu}$ and $t \in \bb R$, 
define a family of perturbed operators $Q_{s, i t}$ as follows: 
with $q = \Lambda'(s)$, for any $\varphi \in \mathscr{B}$, 
\begin{align} \label{operator Rsz}
Q_{s, i t} \varphi(x) 
 = \mathbb{E}_{\mathbb{Q}_{s}^{x}} \left[ e^{ i t ( S_1 - q ) } \varphi( X_1 ) \right], 
   \quad  x \in \bb P(V). 
\end{align}
It follows from the cocycle property \eqref{cocycle01} that 
$Q^n_{s, i t}\varphi(x) = \mathbb{E}_{\mathbb{Q}_{s}^{x}} \left[e^{ i t ( S_n - nq) } \varphi( X_n ) \right]$ 
for $x \in \bb P(V).$ 
The following result shows that  $Q_{s, i t}$ has spectral gap properties.

\begin{lemma} \label{Lem_Perturbation}
Let $s \in I_{\mu}^\circ$. 
Under conditions \ref{Condi-density} and \ref{Condi-density-invariant}, 
there exists $\delta > 0$ such that for any $t \in (-\delta, \delta)$,
\begin{align}
Q^n_{s, i t} = \lambda^{n}_{s, i t} \Pi_{s, i t} + N^{n}_{s, i t},  
\quad  n \geq 1, \label{perturb001}
\end{align}
where the mappings $t \mapsto \Pi_{s, i t}: (-\delta, \delta) \to \mathscr{L(B, B)}$
and $z \mapsto N_{s, i t}: (-\delta, \delta) \to \mathscr{L(B, B)}$ are analytic
in the strong operator sense, 
$\Pi_{s, i t}$ is a rank-one projection with 
$\Pi_{s,0}(\varphi)(x) = \pi_{s}(\varphi)$ for any $\varphi \in \mathscr{B}$ and $x\in \bb P(V)$,
$\Pi_{s, i t} N_{s, i t} = N_{s, i t} \Pi_{s, i t} = 0$ 
and
\begin{align}\label{relationlamkappa001}
\lambda_{s, i t} = e^{- i t q} \frac{ \kappa(s + i t) }{ \kappa(s)}.  
\end{align}
Moreover, for any fixed integer $k \geq 0$, 
there exist constants $c>0$ and $0< a < 1$ such that
\begin{align}
\sup_{|t| < \delta}
\Big\| \frac{d^{k}}{dz^{k}} N^{n}_{s, i t} \Big\|_{\mathscr{B}\rightarrow\mathscr{B}} 
\leq c a^n, \quad  n\geq 1.  \label{OperatorNbound}
\end{align}
\end{lemma}

\begin{proof}
Using the fact that $Q_s$ has a spectral gap by Corollary \ref{cor:Qs_spectral_gap}, 
the proof can be performed in the same way as that in \cite[Corollary 6.3]{BM16}
and \cite[Proposition 3.3]{XGL20a}, upon replacing $\scr{B}_\epsilon$ there by $\scr{B}$. 
We therefore omit the details.
\end{proof}

Let $\Gamma_\mu$ be the closed semigroup of $\bb G$ generated by the support of the measure $\mu$. A matrix $g$ with an algebraically simple dominant eigenvalue (that exceeds all others in absolute value) is called {\em proximal}. 
The measure $\mu$ is said to satisfy condition (i-p) if
\begin{enumerate}
\item 
There is no finite union $\scr{W} = \bigcup_{i=1}^n W_i$ of subspaces $\{0\} \subsetneq W_i \subsetneq \bb R^d$ 
which is $\Gamma_\mu$ -invariant, i.e. $g \scr{W} \subset \scr{W}$ for all $g \in \Gamma_\mu$ (strong irreducibility).
\item 
$\Gamma_\mu$ contains a proximal matrix (proximality).	
\end{enumerate} 

Further, we say that $\mu$ is {\em arithmetic}, if there is $t>0$ together with $\theta \in [0, 2\pi)$ 
and a continuous function 
$\vartheta: \bb P(V) \to \bb R$ such that for all $g \in \Gamma_\mu$ 
and all $x$ in the support of the invariant measure $\nu=\nu_0$, it holds that 
$$ \exp(i 2 \pi  t  \sigma(g, x) - i \theta + i (\vartheta(g \cdot x) - \vartheta(x))) = 1.$$
In other words, $\sigma(g, x)$ is contained in $t \bb Z$ 
up to a shift that may depend on $g$ and $x$ through the function $\vartheta$.

If no such $t>0$ exists, then $\mu$ is said to be {\em non-arithmetic}.

\begin{lemma}\label{Lem-ip}
Under assumption \ref{Condi-density}, the measure $\mu$ satisfies condition (i-p). Moreover, $\mu$ is non-arithmetic.
\end{lemma}

\begin{proof}
	By Lemma \ref{lem:continuous-density}, the density $\dot{\mu}^{(*2)}$ is continuous, hence $\Gamma_\mu$ contains an open set. Then the result follows by \cite[Proposition IV.2.3 and Remark IV.2.4]{BL85}. Note that the proof of Proposition IV.2.3 proceeds by showing the existence of a proximal element. 
	
	The validity of condition (i-p) for $\mu$ then implies that $\mu$ is non-arithmetic, see \cite[Prop. 4.6]{GU05} or the discussion in \cite[Section 2.5]{BM16}.
\end{proof}

As a consequence of $\mu$ being non-arithmetic, we obtain that the function $\log \kappa$ is {\em strictly} convex.

\begin{lemma}\label{lem:strict_convex}
	Assume \ref{Condi-density} and \ref{Condi-density-invariant}. Then the  function $s \mapsto \Lambda(s) := \log \kappa(s)$ is strictly convex on $I_\mu^\circ$.
\end{lemma}

\begin{proof}
Using the consequences of the perturbation result (Lemma \ref{Lem_Perturbation}), 
it is shown in \cite[Corollary 7.3]{BM16} that $\Lambda''(s)=\kappa''(s)/\kappa(s) - (\kappa'(s)/\kappa(s))^2$
 is equal to a quantity denoted by $\sigma_s^2$ (which can be interpreted as the variance of $S_n^x$ 
 under $\bb Q_s^x$); and \cite[Lemma 7.2]{BM16} shows that $\sigma_s^2=0$ 
 can only occur if $\mu$ would be arithmetic. 
 Hence, by Lemma \ref{Lem-ip}, $\Lambda''(s)=\sigma_s^2 >0$ and thus $\Lambda$ is strictly convex.
\end{proof}

We will also make use of the following property.

\begin{lemma} \label{Lem_StrongNonLattice}
Let $s \in I_{\mu}^\circ$. 
Assume conditions \ref{Condi-density} and \ref{Condi-density-invariant}. 
Then for any compact set $K \subset \bb R \setminus \{0\}$, 
there exist constants $c, c_K >0$ such that 
for any $n\geq 1$ and $\varphi\in \mathscr{B}$, 
\begin{align*}
\sup_{t \in K}   \| Q^{n}_{s, i t} \varphi \|_{\scr B}
\leq  c e^{- n c_K }  \|\varphi\|_{\scr B}.
\end{align*}
\end{lemma}

\begin{proof}
This is proved in \cite[Lemma 6.4]{BM16}, 
using that the spectral radius of $Q_{s,it}$ is strictly smaller than one for each $t \neq 0$, 
a property which is shown in \cite[Theorem 5.1]{BM16} 
as a consequence of $\mu$ being non-arithmetic. 
The quoted results are for positive $s$ and $\varphi$ being H\"older continuous, but the proofs carry over to our setting. 
\end{proof}

Similarly to \eqref{operator Rsz}, 
we define a family of dual perturbed operators $Q_{s, i t}^*$: for $s \in I_{\mu}$, $q = \Lambda'(s)$,
$t \in \bb R$ and $\varphi \in \mathscr{B}$, 
\begin{align*} 
Q_{s, i t}^* \varphi(x) 
 = \mathbb{E}_{ \bb Q_{s}^{x, *} } \left[ e^{ i t ( S_1 - q ) } \varphi( X_1 ) \right], 
   \quad  x \in \bb P(V). 
\end{align*}
It follows from  \eqref{cocycle01} that 
$(Q^{*}_{s, i t})^n \varphi(x) = \mathbb{E}_{ \bb Q_{s}^{x, *} } \left[e^{ i t ( S_n - nq) } \varphi( X_n ) \right]$
 for $x \in \bb P(V)$. 
The following results show that  $Q_{s, i t}^*$ has spectral gap properties,
which are similar to those for the operator $Q_{s, i t}$, see Lemmas \ref{Lem_Perturbation} and \ref{Lem_StrongNonLattice}.

\begin{lemma} \label{Lem_Perturbation_bis}
Let $s \in I_{\mu}$. 
Under conditions \ref{Condi-density} and \ref{Condi-density-invariant}, 
there exists $\delta > 0$ such that for any $t \in (-\delta, \delta)$,
\begin{align*}
(Q^*_{s, i t})^n = \lambda^{n}_{s, i t} \Pi^*_{s, i t} + (N^*_{s, i t})^n,  \quad  n \geq 1,  
\end{align*}
where the mappings $t \mapsto \Pi_{s, i t}^* : (-\delta, \delta) \to \mathscr{L(B,B)}$
and $z \mapsto N_{s, i t}^* : (-\delta, \delta) \to \mathscr{L(B, B)}$ are analytic
in the strong operator sense, 
$\Pi_{s, i t}^*$ is a rank-one projection with 
$\Pi_{s,0}^* (\varphi)(x) = \pi_{s}(\varphi)$ for any $\varphi \in \mathscr{B}$ and $x\in \bb P(V)$,
$\Pi_{s, i t}^* N_{s, i t}^* = N_{s, i t}^* \Pi_{s, i t}^* = 0$ 
and $\lambda_{s, i t} = e^{- i t q} \frac{ \kappa(s + i t) }{ \kappa(s)}$. 
Moreover, for any fixed integer $k \geq 0$, 
there exist constants $c>0$ and $0< a < 1$ such that
$\sup_{|t| < \delta}
\| \frac{d^{k}}{dt^{k}} (N^*_{s, i t})^n \|_{\mathscr{B}\rightarrow\mathscr{B}} 
\leq c a^n$ for any $n \geq 1$. 
\end{lemma}

\begin{lemma}\label{dd-Lem_StrongNonLattice}
Let $s \in I_{\mu}$. 
Assume conditions \ref{Condi-density} and \ref{Condi-density-invariant}. 
Then, for any compact set $K \subset \bb R \setminus \{0\}$, 
there exist constants $c, c_K >0$ such that 
$\sup_{t \in K} \|(Q^{*}_{s, i t})^n \varphi \|_{\mathscr{B}} 
\leq c e^{- n c_K}  \|\varphi\|_{\mathscr{B}}$
for any $n\geq 1$ and $\varphi\in \mathscr{B}$. 
\end{lemma}

The proof of Lemmas \ref{Lem_Perturbation_bis} and \ref{dd-Lem_StrongNonLattice} 
can be done using the techniques
from \cite{HL01, BM16} in the same way as Lemmas \ref{Lem_Perturbation} and \ref{Lem_StrongNonLattice}.

\subsection{The many-to-one formula}
In this section we recall the many-to-one formula which has been established in \cite[Lemma 4.2]{Men16}
for the study of fixed points of multivariate smoothing transforms.
Recall that $X_u^x$ and $S_u^x$ are defined by \eqref{Def_Xnxu_Snxu}. 
Recall also that the function $\mathfrak m$ is defined by \eqref{Def-m-s} 
and that for a node $u \in \bb T$, $u|k$ is the restriction of $u$ to its first $k$ components, $1 \leq k \leq |u|$.
\begin{lemma}[The many-to-one formula]\label{Lem-many-to-one}
Assume condition \ref{Condi-density}. 
Then, for any  $s \in I_{\mu}$ and $x \in \bb P(V)$,  
$n \geq 1$ and any bounded measurable function $h: (\bb P(V) \times \bb R)^n \to \bb R$, 
\begin{align}\label{Formula_many_to_one}
  \bb E  \bigg[ \sum_{|u| = n}  
h\left( X^x_{u|1}, S^x_{u|1}, \ldots, X^x_u, S^x_u \right) \bigg] 
 = r_{s}(x)  \mathfrak m(s)^n  
 \bb E_{\bb Q_{s}^x} \bigg[ \frac{1}{r_{s} (X_n)} e^{-s S_n} h \big(X_1, S_1, \ldots, X_n, S_n \big) \bigg].  
\end{align}
\end{lemma}

This formula allows us to reduce the study of the branching product of random matrices to that of the ordinary 
product of random matrices but under the changed probability measure $\bb Q_{s}^x$.

\section{Law of large numbers for the extremal position} \label{sec-LLN}

In this section, we provide the proofs for Theorem \ref{Thm_LLN-001} and Lemma \ref{lem:transform}.
We start with the following strong law of large numbers
for the counting measure of the branching random walk on the group $\bb G$, inspired by \cite{Big79, BGL20}. More precisely, 
for any Borel set $A \subseteq \bb P(V)$, let $Z_n^x(A):=\sum_{|u|=n} \mathds{1}_{A}(X_u^x)$ denote the number of particles in the $n$-th generation with direction in $A$. Note that the total number of particles in the $n$-th generation, $Z_n:=\sum_{|u|=n} 1$ equals $Z_n^x(\bb P(V))$ for any $x \in \bb P(V)$. Writing $m:=\mathfrak{m}(0)=\bb E N$, consider further the normalisation 
$$ W_n^x(A):= \frac{Z_n^x(A)}{m^n}, \qquad W_n:= \frac{Z_n}{m^n}.$$
Thus, $W_n$ is the martingale in the Galton-Watson process, which is known (see e.g. \cite{LPP95})  to converge to a limit $W$ that is strictly positive on the set of survival $\scr S$, $\bb P_x$-a.s. The condition for the convergence is implied by our assumption \ref{Condi_N}.

\begin{proposition}\label{Prop-SLLN-direction}
Assume conditions \ref{Condi_N}, \ref{Condi-density} and \ref{Condi-density-invariant}. 
Then for all $x \in \bb P(V)$ and Borel set $A \subseteq \bb P(V)$ with $\nu(\partial A)=0$, 
$$ \lim_{n \to \infty} W_n^x(A) = \nu(A)W \quad \text{almost surely.}$$
In particular, if $\nu(A) > 0$, then $\liminf_{n \to \infty} Z_n^x(A)>0$ almost surely on the set of survival $\scr S$. 
\end{proposition}

\begin{proof} 
We introduce a parameter $k= [n^{\eta}]$,  where $\eta \in (0,1)$ is a fixed constant. 	
We first decompose the difference $W_n^x(A)- \nu(A)W$ as follows:
\begin{align*}
W_n^x(A) - \nu(A)W = \big(W_n^x(A) - \nu(A)W_{k}\big) + \nu(A)(W_{k} -W). 
\end{align*}
Observe that $\lim_{n \to \infty}\nu(A)(W_{k}-W)=0$ almost surely due to the convergence of the martingale $W_k$
and the property that $k \to \infty$ as $n \to \infty$. 
To show that the first term converges to $0$, we decompose further at generation $k$:
\begin{align}
& W_n^x(A) - \nu(A)W_k  \notag\\
& =  \frac{1}{m^k} 
\sum_{|u|=k} \left( \big[ W_{n-k}^{X_u^x}(A) \big]_u - \bb E \Big[ \big[ W_{n-k}^{X_u^x}(A) \big]_u \ \Big| \ \mathscr{F}_k \Big] \right) 
+ \frac{1}{m^k} \sum_{|u|=k} \left(\bb E \Big[ \big[ W_{n-k}^{X_u^x}(A) \big]_u \ \Big| \ \mathscr{F}_k \Big] - \nu(A)  \right)  \notag\\
& = : C_n + D_n 
\end{align}
Considering $D_n$ first, it follows from the many-to one lemma \ref{Lem-many-to-one} that, with $y = X_u^x$, 
$$ \bb E \Big[ \big[ W_{n-k}^{y}(A) \big]_u \ \Big| \ \mathscr{F}_k \Big]  
= \frac{1}{m^{n-k}}\bb E \Big[ \sum_{|v|={n-k}} \mathds{1}_{A}(X_v^{y}) \Big] 
= r_0\big(y\big) \bb E_{y} \bigg[ \mathds{1}_{A}(X_{n-k})\bigg] = P_0^{n-k} \mathds{1}_A(y).$$
Note here that the eigenfunction $r_0$  corresponding to $P_0$ is just the constant function $1$.  
Upon fixing $\epsilon>0$, let $\varphi \in \scr{B}$ with $\norm{\varphi}_{\scr B}\leq 1$ and $\norm{\varphi-\mathds{1}_A}_{\scr B}< \epsilon/3$. Then by Proposition \ref{prop:Ps}, there is $n_0$ such that $\norm{P_0^m \varphi - \nu(\varphi)}_{\scr{B}} < \epsilon/3$ whenever $m>n_0$ and thus
\begin{align*}
& \abs{P_0^{n-k} \mathds{1}_A(X_u^x) - \nu(A)}  \notag\\
&\leq \abs{P_0^{n-k} \mathds{1}_A(X_u^x) - P_0^{n-k} \varphi(X_u^x)} + \abs{P_0^{n-k} \varphi(X_u^x) - \int \varphi(x) \nu(dx)} + \int \abs{\varphi(x)-\mathds{1}_A(x)} \nu(dx)  \notag\\
& \leq \norm{\varphi-\mathds{1}_A}_{\scr B} + \norm{P_0^{n-k} \varphi - \nu(\varphi)}_{\scr{B}} + \norm{\varphi-\mathds{1}_A}_{\scr B}  \leq \epsilon/3+\epsilon/3+\epsilon/3 = \epsilon
\end{align*}
whenever $n-k> n_0$. 
Since this estimate is independent of the value of $X_u^x$, we conclude that whenever $n-k>n_0$, 
$$ \abs{D_n} \leq \frac{1}{m^k} \sum_{|u|=k} \left|\bb P_0^{n-k} \mathds{1}_A(X_u^x) - \nu(A)  \right| \leq \frac{1}{m^k} \sum_{|u|=k} \epsilon  = \epsilon W_k. $$
Since $\epsilon$ can be chosen arbitrarily small (and $W_k \to W$), we have that $|D_n| \to 0$ as $n \to \infty$, almost surely.
Turning to $C_n$, let us abbreviate $\xi_u := \big[ W_{n-k}^{X_u^x}(A) \big]_u - \bb E \Big[ \big[ W_{n-k}^{X_u^x}(A) \big]_u \ \Big| \ \mathscr{F}_k \Big] $ and observe that $(\xi_u)_{|u|=k}$ are independent conditioned on $\mathscr{F}_k$, and $\bb E [\xi_u \, |\, \mathscr{F}_k]=0$. Further, we have the bound
\begin{equation}\label{eq:boundxiv}
	 |\xi_u| \leq [W_{n-k}]_u + \bb E W_{n-k} = [W_{n-k}]_u +1. 
\end{equation}
Recall that $W_n$ is a mean one martingale that is uniformly integrable under condition \ref{Condi_N}.
 From this, one could use the theory of triangular arrays of independent random variables 
 (see e.g.  \cite{Gut92}) to deduce the convergence in probability to $0$.

To deduce a.s. convergence, one can proceed in the same way as in Steps 1 and 2 of the proof of Lemma 3.1 in \cite{BGL20}. 
The idea of that argument is to mimic Etemadi's  proof  of the strong law of large numbers (\cite{Ete81}), i.e., to first truncate the random variables and then use a Borel-Cantelli argument, where a uniform bound on the law of the $\xi_u$ is used. Namely, by \eqref{eq:boundxiv}, it holds that $\bb P(|\xi_u|>t) \leq \bb P(W^*>t)$ for all $t \geq 0$, where $W_* := \sup_{n \in \bb N} W_n$.
By using our assumption \ref{Condi_N}, it is proved in \cite[Theorem 1.2]{LL15} that $\bb E (1 + W^*) \log^a(1 + W^*) < \infty$ for every $a \geq 1$. 
Note that in \cite{BGL20}, a weaker moment condition on $N$ is employed (see \cite[condition C1]{BGL20}), which makes it necessary there to choose $\eta$ in a particular range. 
\end{proof}

\begin{proof}[Proof of Theorem \ref{Thm_LLN-001}]

We only show how to prove \eqref{eq_LLN-001}, the proof of \eqref{eq_LLN-001mini} being similar. 
Let
\begin{align}\label{def:gamma}
\gamma =  \inf \left\{ \mathfrak M' (s): s \in I_{\mu}^+, \    s\mathfrak M'(s)  >  \mathfrak M(s) \right\}.
\end{align}
Note that the function  $s \mapsto \mathfrak M(s) - s \mathfrak M'(s)$ 
is strictly decreasing on $I_\mu^+$, as its derivative equals $-s \mathfrak M''(s) = -s (\log \kappa(s))''$ and $\log \kappa$ is strictly convex 
(cf.\  Lemma \ref{lem:strict_convex}). 
By assumption, 
there is $s^* \in (I_\mu^+)^\circ$ such that $s^*\Mf'(s^*)= \Mf(s^*)$, hence $\mathfrak M'(s^*)=\gamma$ and 
we conclude that further $\gamma =\sup \{ \mathfrak M'(s): s \in I_{\mu}^+, \  s\mathfrak M'(s)  < \mathfrak M(s) \}$. 

In the sequel, we first prove that for any $\epsilon >0$ and large $n$, no particle is above the threshold $n (\gamma + \epsilon)$; 
then, by a similar argument, we show that there is a positive number of particles below the threshold $n (\gamma - \epsilon)$.

Using the many-to-one formula \eqref{Formula_many_to_one}, 
we get that for any $s \in (I_{\mu}^+)^{\circ}$, with $q = \mathfrak M'(s)$, 
\begin{align*} 
\bb E \Big[ \sum_{|u| = n}  
    \mathds 1 _{\{ S^x_u > n q  \}}   \Big]
    \leq  e^{-snq}  \bb E \Big[  \sum_{ |u| = n }  
 e^{s S^x_u }  \Big]   
   =  e^{-snq}  r_{s}(x)  \mathfrak m(s)^n  
 \bb E_{\bb Q_{s}^x} \Big[ \frac{1}{r_{s} (X_n)}  \Big]   
  \leq  c e^{-snq}  \mathfrak m(s)^n, 
\end{align*} 
where in the last inequality 
we used the fact that the eigenfunction $r_s$ is 
bounded from below and above by strictly positive constants (cf.\ Lemma \ref{Lem_spectral_gap}). 
This implies that for any Borel set $A \subseteq \bb P(V)$, 
\begin{align*} 
\bb P \Big( \sum_{|u| = n}   \mathds 1_{ \{ X^x_u  \in A, \  S^x_u > n q  \}}  > 0  \Big)
  =   \bb P \Big( \sum_{|u| = n}   \mathds 1_{ \{ X^x_u  \in A, \  S^x_u  >  n q  \}}  \geq 1  \Big)   
 \leq  \bb E \Big[ \sum_{|u| = n}   \mathds 1_{ \{ X^x_u  \in A,  \  S^x_u  >  n q  \}}   \Big]     
   \leq  c e^{-snq}  \mathfrak m(s)^n. 
\end{align*}
If $e^{-sq}  \mathfrak m(s) < 1$ (equivalently $s\mathfrak M'(s)  >  \mathfrak M(s)$), 
then by Borel-Cantelli's lemma we get 
$\sum_{|u| = n}   \mathds 1_{ \{ X^x_u  \in A, \  S^x_u> n q  \}} = 0$ for all but finitely many $n$. 
Hence, for any $\ee >0$, it holds that 
$\sum_{|u| = n}   \mathds 1_{ \{ X^x_u  \in A, \  S^x_u > n (\gamma + \ee)  \}} = 0$ for all but finitely many $n$. 
Since $\ee >0$ can be arbitrary small, this implies that 
\begin{align}\label{LLN-Max-Upper}
\limsup_{n \to \infty} \frac{M_n^x(A)}{n} \leq  \gamma. 
\end{align}

On the other hand, under conditions \ref{Condi_N},  \ref{Condi-density}, 
\ref{Condi-density-invariant} and \ref{Condition_Extra}, 
a large deviation result has been established in \cite[Theorem 2.6]{BGL20}. 
The result in \cite{BGL20} is obtained under the stronger assumption that the matrices $(G_u)_{|u|=1}$ 
are independent identically distributed and independent of $N$, but the inspection of the
proof shows that it still holds under the assumptions of our paper, 
the additional tools being the definition of $\mu$ in \eqref{Def-mu} 
and the many-to-one formula in Lemma \ref{Lem-many-to-one}. 
According to these results,   
for any $s \in (I_{\mu}^+)^{\circ}$ such that $e^{-sq}  \mathfrak m(s) > 1$ 
(equivalently $s\mathfrak M'(s)  <  \mathfrak M(s)$), 
 and any Borel set $A \subseteq \bb P(V)$ satisfying $\nu(A) > 0$ and $\nu(\partial A) = 0$
 (noting that $\nu_s(A) > 0$ and $\nu_s(\partial A) = 0$ is equivalent to saying that $\nu(A) > 0$ and $\nu(\partial A) = 0$ 
 under \ref{Condi-density} and \ref{Condi-density-invariant}), 
we have $\bb P$-a.s.
\begin{align*} 
\lim_{n\to\infty} \sqrt{2\pi n} \sigma_{s} e^{n \Lambda^*(q)} 
\frac{\sum_{|u| = n} \mathds 1_{ \{ X_u^x  \in A, \  S_u^x \geq n q   \}} }{(\bb E N )^n} = \frac{1}{s} W_s^x \frac{r_s(x)}{\nu_s(r_s)} \nu_s(A) ,
\end{align*} 
where $\sigma_{s} = \sqrt{\Lambda''(s)} = \sqrt{\mathfrak M''(s)} >0$, $q=\Lambda' (s)= \mathfrak M'(s)$, 
$\Lambda^*(q)= s q-\Lambda(s)$ and $W_s^x$ is a positive random variable on the survival event $\scr S$. 
This can be rewritten as: $\bb P$-a.s. on the survival event $\scr S$
\begin{align*}  
\lim_{n\to\infty} \sqrt{2\pi n} \sigma_{s} e^{n ( \Lambda^*(q) - \log (\bb E N)) } 
\sum_{|u| = n} \mathds 1_{ \{ X_u^x  \in A, \  S_u^x \geq n q   \}}  = \frac{1}{s} W_s^x \frac{r_s(x)}{\nu_s(r_s)} \nu_s(A)  >0.
\end{align*}
We choose $q>\gamma$, 
which is equivalent to 
\begin{align*} 
\log (\bb E N) - \Lambda^*(q) = \log (\bb E N) - sq + \Lambda(s) = \mathfrak M(s) - sq > 0. 
\end{align*}
Then the number of nodes $u$ at generation $n$ such that $\sigma( G_u, x)$ is above the level $nq$ and that $G_u \cdot x \in A$  explodes as $n \to \infty$ $\bb P$-a.s. on the survival event $\scr S$.   
It follows that 
\begin{align}\label{LLN-Max-Lower}
\liminf_{n \to \infty} \frac{M_n^x(A)}{n} \geq  \gamma
\end{align}
on the survival event $\scr S$.

Combining \eqref{LLN-Max-Upper} and \eqref{LLN-Max-Lower} concludes the proof of \eqref{eq_LLN-001}. 
\end{proof}

\begin{remark}\label{rem:linear speed extremal position}
If $s^*$ does not exist, we can still obtain $ \lim_{n \to \infty} \frac{M_n^x(A)}{n}  = \eta$, 
$\bb P$-almost surely on $\scr{S}$, 
with the limit being defined by $\eta =\inf_{s \in I_{\mu}^+} s^{-1}\mathfrak M(s)$.
\end{remark}

\begin{proof}[Proof of Remark \ref{rem:linear speed extremal position}]
If $s^*$ does not exist,  then it holds for all $s \in I_\mu^+$ that $s \Mf'(s) >\Mf(s)$. Observe that the derivative of $h(s):=s^{-1}  \mathfrak M(s)$ for $s>0$ equals  
$h'(s) = s^{-2}\big( s\mathfrak M'(s) -\mathfrak M(s) \big)$
and thus $h(s)$ is decreasing on $I_\mu^+$. 
Consequently, $\eta=\gamma$ with $\gamma$ as defined in \eqref{def:gamma}. 
Thus we can still deduce the lower bound \eqref{LLN-Max-Lower}. 
To obtain the upper bound \eqref{LLN-Max-Upper}, one can proceed as follows.
For $s \in I_\mu^+$ it holds $\kappa(s)=\lim_{n \to \infty} \big(\bb E \norm{G_n}^s \big)^{\frac1n}$, 
see e.g. \cite[Proposition 2.1]{BM16}. 
Let $Z_n := \sup_{|u|=n} \log \norm{G_u}$. 
Then for any $s \in I_\mu^+$, any $n \geq 1$, by an application of Jensen's inequality,
\begin{align*}
\frac1n \bb E \big[ s Z_n \big] 
\leq \frac1n \log \bb E \big[ e^{sZ_n}\big] 
\leq \frac1n \log \bb E \big[ \sum_{|u|=n} \norm{G_u}^s \big] 
= \frac1n \log (\bb E N \kappa(s)) = \frac1n \log \mathfrak{m}(s). 
\end{align*}
Observe that due to submultiplicativity of the norm, and independence of generations, it holds
$Z_{n+k} \leq Z_n + \widetilde{Z}_k,$
where $\widetilde{Z}_k$ has the same law as $Z_k$ and is independent of $Z_n$.
This allows to use an extension of the subadditive ergodic theorem provided in \cite{Lig85} to infer that 
$\lim_{n \to \infty} \frac{Z_n}{n} = \alpha$ almost surely and in $L^1$. By the above calculation, $\alpha \leq \gamma$. 
Finally, $$ \limsup_{n \to \infty} \frac{M_n^x(A)}{n} \leq \limsup_{n\to\infty} \frac{Z_n}{n} = \alpha \leq \gamma,$$
which ends the proof. 
\end{proof}

We conclude this section with the proof of Lemma \ref{lem:transform}, which shows how to transform a branching random walk into one that is in the boundary case. 
Recall that we supply  quantities  corresponding to the original process $(\widetilde{G}_u)_{u \in \bb T}$ 
before the transformation with a tilde, 
and the transformed process 
 is denoted by 
$G_u := \exp\big(- \widetilde{\Mf}'(s^*)\big) \widetilde{G}_u$ for $u \in \bb T$.

\begin{proof}[Proof of Lemma \ref{lem:transform}]
By assumption, there is $s^*$ such that $s^* \widetilde{\Mf}'(s^*) = \widetilde{\Mf}(s^*)$. 
Without loss of generality, we can assume that $\widetilde{\Mf}(s^*) \neq 0$, 
since otherwise we would already be in the boundary case. 
Then necessarily also $s^*$ and $\widetilde{\Mf}'(s^*)$ are nonzero.
Observe that $$S_u^x = \sigma(G_u, x) = \sigma (\widetilde{G}_u, x) -  \widetilde{\Mf}'(s^*) = \widetilde{S}_u^x -  \widetilde{\Mf}'(s^*),$$
while $X_u^x = \widetilde{X}_u^x$.
Therefore, it holds that, for any $s \in I_\mu$, $\varphi \in \scr B$ and $x \in \bb P(V)$, 
\begin{align*}
P_s \varphi(x)  
& = \frac1{\bb E(N)} \bb E \bigg( \sum_{|u|=1} e^{ s\big(\widetilde{S}_u^x -  \widetilde{\Mf}'(s^*)\big) } \varphi(\widetilde{X}_u^x)  \bigg)   \notag\\
& =  \frac{ \exp\left( -  \widetilde{\Mf}'(s^*) \right) }{\bb E(N)} \bb E \bigg( \sum_{|u|=1} e^{ s \widetilde{S}_u^x  } \varphi(\widetilde{X}_u^x)  \bigg)
=  \exp\left( -  \widetilde{\Mf}'(s^*) \right) \widetilde{P}_s \varphi(x). 
\end{align*}
In other words, $P_s$ is just a scalar multiple of $\widetilde{P}_s$. 
In particular, it has the same eigenfunctions and eigenmeasures, and
$$\kappa(s)  = \exp\left( -  \widetilde{\Mf}'(s^*) \right) \widetilde{\kappa}(s).$$
Recalling $\Mf(s) =\log \mathfrak{m}(s)  = \Lambda(s)+ \log \bb E N = \log \kappa(s) + \log \bb E N$ 
(see \eqref{Def-m-s}), it follows that
$$ \Mf(s) = \log \kappa(s) + \log \bb E N 
 = - s  \widetilde{\Mf}'(s^*) + \log \widetilde{\kappa}(s) + \log \bb E N = - s  \widetilde{\Mf}'(s^*) + \widetilde{\Mf}(s),$$	
and further $\Mf'(s) = - \widetilde{\Mf}'(s^*) +  \widetilde{\Mf}'(s)$. 
Hence, taking $s = s^*$, 
we get $\Mf(s^*) = - s^* \widetilde{\Mf}'(s^*) + \widetilde{\Mf}(s^*)= - \widetilde{\Mf}(s^*) + \widetilde{\Mf}(s^*) =0$ 
and $\Mf'(s^*)=0$.
This means that $\mathfrak m (s^*)  = 1$ and $\mathfrak m'(s^*) = 0$, 
so that $s^*=\alpha$ for the transformed process.
\end{proof}

\section{Duality and conditioned integral limit theorems}

In this section we make use of assumptions \ref{Condi-density} and \ref{Condi-density-invariant} to establish 
duality identities and state a series of conditioned limit theorems for products of random matrices under a change of measure.
Note that all these results are stated for products of random matrices, independent of the context of branching random walks.

\subsection{Duality} \label{Sec-dual-Markov}

Recall that $\scr B$ is the Banach space of real-valued continuous functions 
on $\bb P(V)$ endowed with the supremum norm $\| \cdot \|_{\scr B}$. 
Recall also that $Q_s$ is defined by \eqref{def-operator-Qs}. 
Under assumptions \ref{Condi-density} and \ref{Condi-density-invariant},   
on the projective space $\bb P(V)$ there exists a unique invariant probability measure $\pi_s$ of the Markov operator $Q_s$. 
This is proved in \cite[Section 4]{BM16} for $s \in I_{\mu}^+$, and in Corollary \ref{cor:Qs_spectral_gap} for $s \in I_{\mu}^-$. 
Moreover, by the same Corollary \ref{cor:Qs_spectral_gap}, 
the measure $\pi_s$ is absolutely continuous with respect to 
the uniform probability measure on $\bb P(V)$ (denoted by $dx$),
i.e.\ $\pi_s(dx) = \overset\cdot \pi_s(x) dx$, 
where the density function $\overset\cdot \pi_s$ is strictly positive on the projective space $\bb P(V)$. 
 Recall that $dx$ is invariant under the action of the orthogonal group $O(d)$. 
For any $s \in I_{\mu}$ and $\varphi \in \scr B$, define the dual operator $Q_s^*$ as follows: 
\begin{align} \label{Def_Qs_star}
Q_s^*  \varphi (x) 
= \int_{\mathbb G}   \varphi(g \cdot x)   
\frac{r_s(x) e^{-(s+d)\sigma(g, x)}  }{\kappa(s) r_s(g \cdot x)}  
    \frac{\overset\cdot \pi_s(g \cdot x)}{\overset\cdot \pi_s(x)}  | \textup{det}(g)|   \check \mu(dg),  \quad  x \in \bb P(V), 
\end{align}
where $\check \mu$ is the image of the measure $\mu$ by the map $g \mapsto g^{-1}$. 
It can be verified that, under \ref{Condi-density} and \ref{Condi-density-invariant}, the operator $Q_s^*$ is well defined.
The following result shows that $Q_s^*$ is indeed the dual Markov operator of $Q_s$.

\begin{lemma}
Assume conditions \ref{Condi-density} and \ref{Condi-density-invariant}. 
Then, for any $s \in I_{\mu}$ and any $\varphi, \psi \in \scr B$, 
we have 
\begin{align} \label{defofthedualoper003}
\int_{\bb P(V)} \varphi(x) Q_s \psi(x) \pi_s(dx) 
 = \int_{\bb P(V)}  \psi(x) Q_s^*  \varphi (x) \pi_s(dx). 
\end{align}
\end{lemma}

\begin{proof}
Since $\pi_s(dx) = \overset\cdot \pi_s(x) dx$,  by the definition of $\mathbb Q_s^x$, 
we have that for any bounded measurable function $F: \bb P(V)\times  \bb G \times \bb P(V)  \to \bb R$, 
\begin{align}\label{basic-dial-001}
I: = \int_{\bb P(V)}  \mathbb E_{\mathbb Q_s^x} F(x,g_1,X_1) \pi_s(dx)   
= \int_{\mathbb G} \int_{\bb P(V)}  
\frac{r_s(g \cdot x)e^{s\sigma(g, x)}}{\kappa(s) r_s(x)}  F(x, g, g \cdot x) 
  \overset\cdot \pi_s(x) dx   \mu(dg).
\end{align}
Notice that for any bounded measurable function $\varphi: \bb P(V) \to \bb R$ and any $g \in \bb G$, we have
\begin{align}\label{Change-of-variable-manifold}
\int_{\bb P(V)} \varphi(x) dx = \int_{\bb P(V)} \varphi(g \cdot x)  |\textup{det}(g)|  e^{-d \sigma(g, x)} dx. 
\end{align}
Applying this formula to the integral in \eqref{basic-dial-001}, we obtain that for any $g \in \bb G$, 
\begin{align*}
& \int_{\bb P(V)}  \frac{r_s(g \cdot x)e^{s\sigma(g, x)}}{ r_s(x)}  F(x, g, g \cdot x)  \overset\cdot \pi_s(x) dx   \notag\\
& =   \int_{\bb P(V)}  \frac{r_s(x)  e^{s\sigma(g, g^{-1} \cdot x)}}{ r_s(g^{-1} \cdot x)}  
 F(g^{-1} \cdot x, g, x)  \overset\cdot \pi_s(g^{-1} \cdot x)  | \textup{det}(g^{-1})|  e^{-d \sigma(g^{-1}, x)} dx.  
\end{align*}
Therefore, 
\begin{align*}
I = \int_{\mathbb G} \int_{\bb P(V)}   
   \frac{r_s(x)e^{s\sigma(g, g^{-1} \cdot x)}}{\kappa(s) r_s(g^{-1} \cdot x)}  F(g^{-1} \cdot x, g, x) 
  \frac{\overset\cdot \pi_s(g^{-1} \cdot x)}{\overset\cdot \pi_s(x)}  | \textup{det}(g^{-1})|  e^{-d \sigma(g^{-1}, x)}   \pi_s(dx)   \mu(dg). 
\end{align*}
Since $\sigma(g, g^{-1} \cdot x) = - \sigma(g^{-1},x)$
and $\check \mu$ is the image of $\mu$ by $g \mapsto g^{-1}$, 
by a change of variable and Fubini's theorem, we obtain
\begin{align}\label{basic-dial-002}
I =  \int_{\bb P(V)} \int_{\mathbb G}   \frac{r_s(x)e^{-(s+d)\sigma(g,x)}}{\kappa(s) r_s(g \cdot x)}  F(g \cdot x,g^{-1},x) 
  \frac{\overset\cdot \pi_s(g \cdot x)}{\overset\cdot \pi_s(x)}  | \textup{det}(g)|   \check  \mu(dg)  \pi_s(dx).
\end{align}
In particular, taking $F(x,g,x')=\vphi(x) \psi( x')$ for $x,x'\in \bb P(V)$, 
where $\vphi$ and $\psi$  are bounded measurable functions from $\bb P(V)$ to $\bb R$, 
using the definition of $Q_s^*$ we get that
\begin{align*} 
 \int_{\bb P(V)}  \vphi(x)  Q_s \psi(x)  \pi_s(dx)  
&  =\int_{\bb P(V)}  \vphi(x)   \mathbb E_{\mathbb Q_s^x} \psi(X_1) \pi_s(dx)   \nonumber \\
&=  \int_{\bb P(V)} \psi(x) \int_{\mathbb G}   \frac{r_s(x) e^{-(s+d)\sigma(g,x)}}{\kappa(s) r_s(g \cdot x)}  
  \varphi(g \cdot x)  \frac{\overset\cdot \pi_s(g \cdot x)}{\overset\cdot \pi_s(x)}  | \textup{det}(g)|   \check  \mu(dg)  \pi_s(dx) \notag\\
& = \int_{\bb P(V)}  \psi(x) Q_s^*  \varphi (x) \pi_s(dx), 
\end{align*}
which ends the proof of the lemma.  
\end{proof}

Similarly to \eqref{Def-kernel-qns} and \eqref{cocycle01}, 
by the fact that $\sigma(\cdot, \cdot)$ is a cocycle, 
one can verify that for any $s \in I_{\mu}$, the family of kernels
\begin{align}\label{def-dual-kernel}
q_n^{s,*}(x, G_n) =  \frac{1}{ \kappa(s)^n }   e^{- (s+d) \sigma(G_n, x)}    \frac{r_s(x) }{ r_s(G_n \cdot x)}  
\frac{\overset\cdot \pi_s(G_n \cdot x)}{\overset\cdot \pi_s(x)}  |\textup{det}(G_n)|,   \quad  n \geq 1, 
\end{align}
satisfies the following property: for any $x \in \bb P(V)$, 
\begin{align} \label{cocycle002}
q_{n}^{s,*}(x,  G_n)  q_{m}^{s,*}(G_n \cdot x, g_{n+m} \ldots g_{n+1} ) = q_{n+m}^{s,*}(x, G_{n+m}). 
\end{align}
Recalling that $\check \mu$ is the image of $\mu$ by the map $g \mapsto g^{-1}$, 
similarly to \eqref{def-Qsn}, one can verify that 
\begin{align}\label{def-Qsn-star}
\mathbb{Q}_{s,n}^{x,*}(dg_1, \ldots, dg_n) : = q_n^{s,*}(x, G_n) \check \mu(dg_1) \ldots \check \mu(dg_n),  \quad  n \geq 1,
\end{align}
is a sequence of probability measures, 
which, by \eqref{cocycle002}, forms a projective system on $\Omega^* = \bb G^{\bb N}$. 
Hence, by the Kolmogorov extension theorem,
there exists a unique probability measure  $\mathbb{Q}_s^{x,*}$ on $\bb G^{\bb N}$
with marginals $\mathbb{Q}_{s,n}^{x,*}$. 
Denote by $\mathbb{E}_{\mathbb{Q}_s^{x,*}}$ the corresponding expectation.

For any $s \in I_{\mu}$, consider the probability space $(\Omega^*, \mathscr B(\Omega^*), \mathbb{Q}_s^{x,*})$. 
Let $g_1^*,g_2^*,\ldots$ be coordinate maps of $\Omega^*$, i.e. $g_k(\omega)=\omega_k$, 
where $\omega \in \Omega^*$ and 
$\omega_k$ is the $k$-th coordinate of $\omega$.  
The dual Markov chain $(X_n^*)_{n\geq 0}$ with starting point $X_0 \in \bb P(V)$ 
is defined on the space $(\Omega^*, \mathbb{Q}_s^{x,*})$ by setting 
\begin{align}\label{Def-dual-MC}
X_n^* = (g_n^*\ldots g_1^*) \cdot X_0,  \quad  n \geq 1. 
\end{align} 
For any starting point $X_0 = x \in \bb P(V)$ and any measurable set $A \subseteq \bb P(V)$, 
the transition probability of $(X_n^*)_{n\geq 0}$ is defined by $\bb Q_s^*(x,A) := Q_s^* \mathds 1_{A}(x)$, 
where $Q_s^*$ is defined by \eqref{Def_Qs_star}. 
Hence, for any bounded measurable function $\varphi$ on $\bb P(V)$, 
\begin{align*} 
\bb E_{\mathbb{Q}_s^{x,*}} \varphi (X_n^*) = (Q_s^*)^n  \varphi (x),  \quad  x \in \bb P(V), 
\end{align*}
where 
\begin{align*} 
(Q_s^*)^n  \varphi (x) 
=   \int_{\mathbb G}   \varphi(G_n \cdot x)  \frac{ r_s(x) e^{-(s+d) \sigma(G_n, x)} }{\kappa(s)^n  r_s(G_n \cdot x)} 
\frac{\overset\cdot \pi_s(G_n \cdot x)}{\overset\cdot \pi_s(x)}  |\textup{det}(G_n)|  \check \mu(dg_1) \ldots \check \mu(dg_n).
\end{align*}

\begin{lemma}\label{lemmaReversed001}
Assume conditions \ref{Condi-density} and \ref{Condi-density-invariant}.  
Then, for any $s \in I_{\mu}$, $n \geq 1$
and any measurable function $F: \bb P(V)\times  (\bb G \times \bb P(V))^{n}  \to \bb R_+$, 
we have  
\begin{align*} 
 \int_{\bb P(V)}  \mathbb E_{\mathbb Q_s^x} F(x,g_1,X_1,\ldots,g_n,X_n) \pi_s(dx)  
=  \int_{\bb P(V)} \mathbb E_{\mathbb{Q}_s^{z,*}} F(X_n^*,(g_n^*)^{-1},\ldots,X_1^*, (g_1^*)^{-1},z)  \pi_s(dz).
\end{align*}
\end{lemma}
\begin{proof}
The assertion of the lemma for $n=1$ follows from 
the identities \eqref{basic-dial-001} and \eqref{basic-dial-002}, and the definition of the measure $\mathbb{Q}_s^{z,*}$.  
The case $n=2$ is obtained from the case $n=1$ by the following calculations:
\begin{align*} 
 J: & =\int_{\bb P(V)}  \mathbb E_{\mathbb Q_s^x} F(x, g_1, X_1, g_2, X_2) \pi_s(dx)   \notag\\
&=  \int_{\mathbb G} \int_{\bb P(V)}  
\int_{\bb G}  F(x, g_1, X_1, g_2, g_2 \cdot X_1)    
\frac{r_s( g_2 \cdot X_1) e^{s\sigma( g_2,  X_1)} }{\kappa(s) r_s(X_1)}  \bb Q_{s,1}^x(dg_1)   \pi_s(dx)  \mu(dg_2),
\end{align*}
where $\bb Q_{s,1}^x$ is a probability measure defined by \eqref{def-Qsn}.
From \eqref{basic-dial-001} and \eqref{basic-dial-002}, and the definition of $\mathbb Q_s^{z,*}$,  
we see that for any measurable function $H: \bb P(V)\times  \bb G \times \bb P(V)  \to \bb R_+$, 
\begin{align} \label{basic-dial-001_bis}
\int_{\bb P(V)}  \mathbb E_{\mathbb Q_s^x} H(x,g_1, g_1 \cdot x) \pi_s(dx)   
=  \int_{\bb P(V)}  \mathbb E_{\mathbb Q_s^{z,*}} H(g_1^* \cdot z, (g_1^*)^{-1}, z) \pi_s(dz).  
\end{align}
Using \eqref{basic-dial-001_bis}, we get that for any fixed $g_2 \in \bb G$, 
\begin{align*}
& \int_{\bb P(V)}  \int_{\bb G}
F(x, g_1, X_1, g_2, g_2 \cdot X_1)    \frac{r_s( g_2 \cdot X_1) e^{s\sigma( g_2,  X_1)} }{\kappa(s) r_s(X_1)} 
\bb Q_{s,1}^x(dg_1)  \pi_s(dx)  \notag\\
& =   \int_{\bb P(V)}  \int_{\bb G}
F( g_1^* \cdot z, (g_1^*)^{-1},  z, g_2, g_2 \cdot z)    \frac{r_s( g_2 \cdot z) e^{s\sigma( g_2,  z)} }{\kappa(s) r_s(z)} 
\bb Q_{s,1}^{z,*}(dg_1)  \pi_s(dz), 
\end{align*}
where $\bb Q_{s,1}^{x,*}$ is a probability measure defined by \eqref{def-Qsn-star}. 
By a change of variable $z = g_2^{-1} \cdot z'$ and applying \eqref{Change-of-variable-manifold}, we obtain
\begin{align*} 
J & = \int_{\mathbb G}  
\int_{\bb P(V)}   \int_{\bb G}
 F( g_1^* \cdot z, (g_1^*)^{-1},  z, g_2, g_2 \cdot z)    \frac{r_s( g_2 \cdot z) e^{s\sigma( g_2,  z)} }{\kappa(s) r_s(z)} 
 \bb Q_{s,1}^{z,*}(dg_1)   \overset \cdot \pi_s(z) dz \mu(dg_2)  \notag\\
& = \int_{\mathbb G}  \int_{\bb P(V)}  \int_{\bb G} 
F( g_1^* \cdot (g_2^{-1} \cdot z), (g_1^*)^{-1},  g_2^{-1} \cdot z, g_2, z)  
  \frac{r_s( z) e^{s\sigma( g_2,  g_2^{-1} \cdot z)} }{\kappa(s) r_s(g_2^{-1} \cdot z)}    \notag\\
& \qquad\qquad\qquad\qquad\qquad \times      
   |\textup{det}(g_2^{-1})|  e^{-d \sigma(g_2^{-1}, z)}  \bb Q_{s,1}^{g_2^{-1} \cdot z,*}(dg_1)  
   \overset \cdot \pi_s(g_2^{-1} \cdot z)  dz \mu(dg_2).  
 \end{align*}
Using the fact that $\sigma(g, g^{-1} \cdot z) = - \sigma(g^{-1}, z)$ and passing to the measure $\check \mu$, we get
 \begin{align*}
J  & = \int_{\mathbb G}  
\int_{\bb P(V)}  \int_{\bb G}   F( g_1^*  g_2 \cdot z, (g_1^*)^{-1},  g_2 \cdot z, g_2^{-1}, z)    
  \frac{r_s(z) e^{ - (s+d)\sigma(g_2,  z)} }{\kappa(s) r_s(g_2 \cdot z)}  
   \bb Q_{s,1}^{g_2 \cdot z,*}(dg_1) 
     \overset \cdot \pi_s( g_2 \cdot z)  |\textup{det}(g_2)| dz  \check\mu(dg_2)   \notag\\
& = \int_{\mathbb G}  
\int_{\bb P(V)}   
\int_{\bb G}  F( g_1^* g_2 \cdot z, (g_1^*)^{-1},  g_2 \cdot z, g_2^{-1}, z)    
   \frac{r_s(z) e^{ - (s+d)\sigma(g_2,  z)} }{\kappa(s) r_s(g_2 \cdot z)} 
 \bb Q_{s,1}^{g_2 \cdot z,*}(dg_1) 
   \frac{\overset \cdot \pi_s( g_2 \cdot z)}{ \overset \cdot \pi_s(z) }  |\textup{det}(g_2)|  \pi_s(dz)  \check\mu(dg_2)   \notag\\
&= \int_{\bb P(V)} \int_{\mathbb G}  \int_{\bb G} 
F(g_1^* g_2 \cdot z,  (g_1^*)^{-1} , g_2 \cdot z, g_2^{-1} ,z)   \bb Q_{s,1}^{g_2 \cdot z,*}(dg_1) 
\bb Q_{s,1}^{z,*}(dg_2) \pi_s(dz),
\end{align*}
where in the last equality we used \eqref{def-dual-kernel} and \eqref{def-Qsn-star}. 
Using \eqref{cocycle002} and the fact that the law of $(g_1^*, g_2^*)$
coincides with that of $(g_2^*, g_1^*)$ under the measure $\mathbb{Q}_s^{z,*}$, 
we obtain
\begin{align*} 
J&=\int_{\bb P(V)} \mathbb E_{\mathbb{Q}_s^{z,*}} 
F(X_2^*, (g_2^*)^{-1}, X_1^*,  (g_1^*)^{-1} ,z)  \pi_s(dz), 
\end{align*}
which concludes the proof of the case $n =2$.  The case of $n \geq 3$ is proved similarly.
\end{proof}

From Lemma \ref{lemmaReversed001} we immediately get the following assertion:  

\begin{lemma}\label{lemRev000}
Assume conditions \ref{Condi-density} and \ref{Condi-density-invariant}. 
Then, for any $s \in I_{\mu}$ and any measurable function $F: ( \bb P(V))^{n+1}  \to \bb R_+$, 
\begin{align*} 
\int_{\bb P(V)}  \mathbb E_{\mathbb Q_s^x} F(x,X_1,\ldots,X_n) \pi_s(dx)  
=  \int_{\bb P(V)} \mathbb E_{\mathbb{Q}_s^{z,*}} F(X_n^*,\ldots,X_1^*,z)  \pi_s(dz).
\end{align*}
\end{lemma}

Recall that the Markov random walk $(X_n, S_n)$ is defined by \eqref{Def_Xn}. 
We consider the first time when the random walk $(y-S_n)_{n\geq 1}$
starting at the point $y \in \mathbb{R}_+$ becomes negative:
\begin{align}\label{Def-tau-y}
\tau_y = \inf \{ k \geq 1: y - S_k < 0\}. 
\end{align}
Now we introduce the dual random walk $(S_n^*)_{n\geq 0}$ by setting $S_0^* =0$ and for any $n\geq 1$,
\begin{align}\label{def-dual-randomwalk-Sn}
S_n^* =  \sum_{k=1}^n \sigma(g_k^*, X_{k-1}^*), 
\end{align}
where $(X_n^*)_{n \geq 0}$ is defined by \eqref{Def-dual-MC}.  
The first time when the random walk $(y-S_n^*)_{n\geq 1}$
starting at the point $y \in \mathbb{R}_+$ becomes negative is defined by
\begin{align*}
\tau_y^{*} = \inf \{ k \geq 1: y - S_k^* < 0\}. 
\end{align*}

In its turn Lemma \ref{lemRev000} implies a duality property of the Markov chains 
$(X_n, S_n)_{n\geq 0}$ and $(X_n^*, S_n^*)_{n\geq 0}$ conditioned to stay positive.

\begin{lemma}\label{Lem_Duality}
Assume conditions \ref{Condi-density} and \ref{Condi-density-invariant}. 
Then, for any $n \geq 1$,  $s \in  I_{\mu}$
and any measurable functions $\varphi, \psi: \bb P(V) \times \bb R \to \bb R_+$, 
we have  
\begin{align}\label{duality-formular-abc}
&\int_{\bb P(V)} \int_{\mathbb{R}} \varphi\left(x, y\right) 
\mathbb{E}_{ \mathbb{Q}_s^{x}} \Big[ \psi \left( X_n,  y - S_{n} \right); \tau_{y} > n - 1 \Big] dy \pi_s(dx) \notag\\
&\qquad\qquad = \int_{\bb P(V)}  \int_{\mathbb{R}} 
   \psi \left(z, t\right) 
   \mathbb{E}_{ \mathbb{Q}_s^{z,*}} \Big[ \varphi\left( X_n^*,  t - S_{n}^* \right); \tau_{t}^* > n - 1 \Big] dt \pi_s(dz). 
\end{align}
In particular, for any $n \geq 1$, $s \in  I_{\mu}$
and any measurable functions $\varphi, \psi: \bb P(V) \times \bb R_+ \to \bb R_+$, 
we have 
\begin{align}\label{duality-formular-particu-abc}
&\int_{\bb P(V)} \int_{\mathbb{R}_{+}} \varphi\left(x, y\right) 
\mathbb{E}_{ \mathbb{Q}_s^{x}} \Big[ \psi \left( X_n,  y - S_{n} \right); \tau_{y} > n \Big] dy \pi_s(dx)  \notag\\
&\qquad\qquad = \int_{\bb P(V)}  \int_{\mathbb{R}_{+}} 
   \psi \left(z, t\right) 
   \mathbb{E}_{ \mathbb{Q}_s^{z,*}} \Big[ \varphi\left( X_n^*,  t - S_{n}^* \right); \tau_{t}^* > n \Big] dt \pi_s(dz). 
\end{align}
\end{lemma}

\begin{proof}
Consider the function $\Psi$ defined as follows: for $x_0, x_n \in \bb P(V)$ and $y_0, y_1, \ldots, y_n \in \bb R$, 
\begin{align*} 
\Psi(x_0,y_0,y_1,\ldots, y_{n-1} ,y_n,x_n) := \vphi(x_0,y_0) \psi(x_n,y_n) \mathds 1_{ \{ y_1\geq 0,\ldots, y_{n-1} \geq 0 \} }. 
\end{align*}
By the definition of $\Psi$, it follows that 
\begin{align*} 
J: &=\int_{\bb P(V)} \int_{\mathbb{R}} \varphi\left(x, y\right) 
\mathbb{E}_{ \mathbb{Q}_s^{x}} \Big[ \psi \left( X_n,  y - S_{n} \right); \tau_{y} > n - 1 \Big]   dy \pi_s(dx) \\
&=\int_{\bb P(V)}  \int_{\mathbb{R}}  \mathbb E_{\mathbb Q_s^x} 
\Psi \Big( x, y, y - \sigma(g_1,x), 
  \ldots, y - \sigma(g_1,x) - \ldots - \sigma(g_n,X_{n-1}), X_n  \Big)  dy \pi_s(dx).
\end{align*}
Applying Lemma \ref{lemmaReversed001}, we obtain
\begin{align*} 
J =\int_{\bb P(V)} \int_{\mathbb{R}}  \mathbb E_{\mathbb{Q}_s^{z,*}} 
\Psi \bigg( X_n^*, y, y - \sigma((g_n^*)^{-1}, X_n^*), y - \sum_{k=n}^{n-1} \sigma((g_k^*)^{-1},X_k^*), 
\ldots, y - \sum_{k=n}^1 \sigma((g_k^*)^{-1}, X_k^*), z  \bigg) dy \pi_s(dz).
\end{align*}
By a change of variable 
$y = t + \sum_{k=n}^1 \sigma((g_k^*)^{-1},X_k^*)$, 
we see that
\begin{align*} 
J  = \int_{\bb P(V)} \int_{\mathbb{R}}  \mathbb E_{\mathbb{Q}_s^{z,*}} 
\Psi \bigg( X_n^*, t + \sum_{k=1}^{n} \sigma((g_k^*)^{-1}, X_k^*) , t + \sum_{k=1}^{n-1} \sigma((g_k^*)^{-1},X_k^*), 
\ldots, t + \sigma((g_1^*)^{-1}, X_{1}^* ), t , z  \bigg) dt \pi_s(dz).
\end{align*}
Since $\sigma((g_k^*)^{-1},X_k^*) = -\sigma(g_k^*,X_{k-1}^*)$
for any $1 \leq k \leq n$ with $X_0^* = z$ under the measure $\mathbb{Q}_s^{z,*}$,  we get
\begin{align*} 
J & =\int_{\bb P(V)} \int_{\mathbb{R}}  \mathbb E_{\mathbb{Q}_s^{z,*}} 
\Psi \bigg( X_n^*, t - \sum_{k=1}^{n} \sigma(g_k^*,X_{k-1}^*) ,t - \sum_{k=1}^{n-1} \sigma(g_k^*,X_{k-1}^*), 
\ldots, t -  \sigma(g_1^*, z ), t , z  \bigg) dt \pi_s(dz) \\
& = \int_{\bb P(V)}  \int_{\mathbb{R}} 
   \psi \left(z, t\right) \mathbb{E}_{ \mathbb{Q}_s^{z,*}} \Big[ \varphi\left( X_n^*,  t - S_{n}^* \right); \tau_{t}^* > n - 1 \Big] dt \pi_s(dz), 
\end{align*}
which finishes the proof of \eqref{duality-formular-abc}.  
Identity \eqref{duality-formular-particu-abc} follows from \eqref{duality-formular-abc}
by taking $\varphi\left(x, y\right)  = \varphi_1\left(x, y\right) \mathds 1_{\{y \geq 0\}}$
and $\psi\left(x, y\right)  = \psi_1\left(x, y\right) \mathds 1_{\{y \geq 0\}}$.  
\end{proof}


\subsection{Conditioned integral limit theorems}  
In this section we state several conditioned integral limit theorems,
which will play important roles in Section \ref{Sec-CLLT} to establish the conditioned local limit theorems
for products of random matrices.  
In the following result we give the existence of the harmonic function $V_{s}$ under the changed measure $\bb Q_s^x$,
and states some of its properties.

\begin{proposition}\label{Prop-harmonic}
Assume condition \ref{Condi-density} and $\kappa'(s) = 0$ for some $s \in  I_{\mu}$. 
\begin{enumerate}
\item
For any $x \in \bb P(V)$ and $y \geq 0$, the following limit exists: 
\begin{align*}
\lim_{n \to \infty}
\mathbb{E}_{\bb Q_{s}^x} \left( y - S_n; \,   \tau_y  >  n \right)  =: V_{s}(x, y). 
\end{align*}

\item
For any $x \in \bb P(V)$, 
the function $V_{s}(x, \cdot)$ is increasing on $\bb R_+$ and
there exist constants $c_1, c_2 >0$ such that for all $x \in \bb P(V)$ and $y \geq 0$, 
\begin{align*}
0 \vee (y-c_1) < V_{s}(x, y) \leq c_2(1 + y). 
\end{align*}
Moreover, $\inf_{x \in \bb P(V), y \geq 0} V_s(x, y) >0$ and $\lim_{y \to \infty} \frac{V_{s}(x,y)}{y} = 1$. 

\item
The function $V_{s}$ is harmonic, i.e., for any $x \in \bb P(V)$ and $y \geq 0$, 
\begin{align*}
\mathbb{E}_{\bb Q_{s}^x} \left( V_{s}(X_1, y - S_1); \,   \tau_y  >  1 \right) 
 = V_{s}(x, y). 
\end{align*}
\end{enumerate}

\end{proposition}

The assertion $\inf_{x \in \bb P(V), y \geq 0} V_s(x, y) >0$ is not explicitly stated  in \cite{GLP17}, 
but it can be found in the proof of Proposition 5.12 of \cite{GLP17}.

The following result gives a uniform upper bound for $\bb Q_{s}^x (\tau_y > n)$. 

\begin{theorem}\label{Th_Asymptotic_ExitTime}
Assume condition \ref{Condi-density} and $\kappa'(s) = 0$ for some $s \in  I_{\mu}$. 
Then
\begin{align*}
\limsup_{n \to \infty} n^{1/2} 
\sup_{x \in \bb P(V)} \sup_{y \geq 0} \frac{1}{y+1} \bb Q_{s}^x (\tau_y > n) < \infty. 
\end{align*}
\end{theorem}

In the following we formulate a conditional integral limit theorem for the random walk $(y - S_n)$
under the changed measure $\bb Q_{s}^x$. 
Denote $\sigma_s = \sqrt{\Lambda''(s)}$ and let $\Phi^+(t) = (1 - e^{-t^2/2}) \mathds 1_{\{ t \geq 0\}} $ 
be the Rayleigh distribution function on $\bb R$. 

\begin{theorem}\label{Theor-IntegrLimTh} 
Assume condition \ref{Condi-density} and $\kappa'(s) = 0$ for some $s \in  I_{\mu}$. 
Let $(\alpha_n)_{n \geq 1}$ be any sequence of positive numbers satisfying $\lim_{n \to \infty} \alpha_n = 0$. 
Then, for any $\ee >0$, 
there exists a constant $c_{\ee} > 0$ such that for any $n\geq 1$, $x \in \bb P(V)$  
and $y \in [0, \alpha_n \sqrt{n}]$, 
\begin{align*} 
\sup_{t\in \bb R}  \left| \bb Q_{s}^x \left( \frac{y - S_{n}}{\sigma_{s} \sqrt{n}} \leq t, \,  \tau_y>n\right)
   -  \frac{2V_{s}(x, y)}{\sigma_{s} \sqrt{2\pi n}} \Phi^+(t)\right|  
\leq c_{\ee}  \frac{1 + y}{\sqrt{n}} (\alpha_n + n^{-\ee}). 
\end{align*}
\end{theorem}

The proof of Proposition \ref{Prop-harmonic} and Theorems \ref{Th_Asymptotic_ExitTime}, \ref{Theor-IntegrLimTh}
 can be performed in the same way as the corresponding assertions in \cite[Theorems 2.2, 2.4 and 2.5]{GLL18}.
The key point is that the pair $(X_n, S_n)$ is a Markov chain under the changed measure $\bb Q_{s}^x$
and the spectral gap properties hold for the corresponding perturbed operator which allows us to obtain a martingale approximation for the Markov walk $S_n$. 
The rate of convergence in Theorem \ref{Theor-IntegrLimTh}  
can be obtained by using the techniques similar to that in \cite{GX21}.

Next we formulate similar results for the dual Markov random walk $(X_n^*, S_n^*)$. 
They are easy consequences of Theorems \ref{Th_Asymptotic_ExitTime} and \ref{Theor-IntegrLimTh} due to 
the fact that the dual Markov random walk has spectral gap properties (see Section \ref{subsect:spectral gap} below). 
For any $x \in \bb P(V)$ and $y \geq 0$, we define 
$V_{s}^*(x, y) := \lim_{n \to \infty} \mathbb{E}_{\bb Q_{s}^{x, *}} ( y - S_n^*; \,   \tau_y^*  >  n )$.
Then $V_s^*$ satisfies similar properties as $V_s$ stated in Proposition \ref{Prop-harmonic}.

\begin{theorem}\label{Th_Asymptotic_ExitTime_dual}
Assume conditions \ref{Condi-density}, \ref{Condi-density-invariant},
and $\kappa'(s) = 0$ for some $s \in  I_{\mu}$. 
Then
\begin{align*}
\limsup_{n \to \infty} n^{1/2} 
\sup_{x \in \bb P(V)} \sup_{y \geq 0} \frac{1}{y + 1} (\bb Q_{s}^{x,*}) (\tau_{y}^* > n) < \infty. 
\end{align*}
\end{theorem}

\begin{theorem}\label{Theor-IntegrLimTh-reversed} 
Assume conditions \ref{Condi-density}, \ref{Condi-density-invariant},
and $\kappa'(s) = 0$ for some $s \in  I_{\mu}$. 
Let $(\alpha_n)_{n \geq 1}$ be any sequence of positive numbers satisfying $\lim_{n \to \infty} \alpha_n = 0$. 
Then, for any $\ee >0$, there exists a  constant $c_{\ee} > 0$ such that for any $n\geq 1$, $x \in \bb P(V)$  
and $y \in [0, \alpha_n \sqrt{n}]$, 
\begin{align*} 
\sup_{t\in \bb R}  \left| \bb Q_{s}^{x, *} \left( \frac{y - S_{n}^* }{\sigma_{s} \sqrt{n}} \leq t, \,  \tau_y^* > n \right)
   -  \frac{2V^*_{s}(x, y)}{\sigma_{s} \sqrt{2\pi n}} \Phi^+(t)\right|  
\leq c_{\ee}  \frac{1 + y}{\sqrt{n}} (\alpha_n + n^{-\ee}). 
\end{align*}
\end{theorem}

\section{Martingale approximation and local limit theorems}

\subsection{Martingale approximation}

We shall use the strategy of Gordin \cite{Gor69} 
to construct a martingale approximation for the Markov walk $S_n$ under the changed measure $\bb Q_s^x$.
Assume that $q = \Lambda'(s) = 0$ (or equivalently $\kappa'(s) = 0$) for some $s \in I_{\mu}$. 
Denote 
\begin{align*}
\bar{\sigma}(x): = \bb E_{\bb Q_s^x} S_1 = \int_{\bb G} \sigma(g,x) \bb Q_{s,1}^x(dg),  \quad x \in \bb P(V), 
\end{align*}
where $\bb Q_{s,1}^x$ is defined by \eqref{def-Qsn-star}. 
One can verify that $\bar{\sigma} \in \mathscr B$ and the equation 
\begin{align}\label{Cohomo_Equa}
\bar{\sigma}(x) = \theta(x) - Q_s \theta(x),  \quad x \in \bb P(V),  
\end{align}
has a (unique) solution given by 
\begin{align}\label{Solution_Poisson}
\theta(x) = \bar{\sigma}(x) + \sum_{n=1}^{\infty} Q_s^n \bar{\sigma} (x),  \quad x \in \bb P(V),  
\end{align}
where $Q_s$ is defined by \eqref{def-operator-Qs}. 
To show this, note that, using the spectral gap properties of $Q_s$ (by Lemma \ref{Lem_Perturbation} with $t = 0$),
\begin{align}
Q_s^n \bar{\sigma} (x) = \pi_s(\bar{\sigma}) + N_{s,0}^n \bar{\sigma}(x). 
\end{align}
In addition,
\begin{align}\label{Formular_q}
\Lambda'(s) = \int_{\bb P(V)} \int_{\bb G} \sigma(g,x) \bb Q_{s,1}^x(dg) \pi_s(dx).  
\end{align}
Since $\pi_s(\bar{\sigma}) = 0$ (by \eqref{Formular_q} and the assumption $q = 0$)    
and $|N_{s,0}^n \bar{\sigma}(x)| \leq Ce^{-cn}$, we get $|Q_s^n \bar{\sigma} (x)| \leq Ce^{-cn}$. 
Therefore, the function $\theta$ in \eqref{Solution_Poisson} is well defined
and satisfies the equation \eqref{Cohomo_Equa}.

Let $\mathscr{F}_{0}$ be the trivial $\sigma $-algebra 
and $\mathscr{F}_{n} = \sigma \{ g_{k}: 1 \leq k \leq n \} $ for $n\geq 1.$ 
For any $g \in \bb G$ and $x \in \bb P(V)$, let
\begin{align}\label{Marting_app_01}
\sigma_0(g, x) = \sigma(g, x) - \theta(x) + \theta(g \cdot x). 
\end{align}
Then $\sigma_0$ satisfies the cocycle property $\sigma_0(g_2 g_1, x) = \sigma_0(g_2, g_1 \cdot x) + \sigma_0(g_1, x)$
for any $g_2, g_1 \in \bb G$ and $x \in \bb P(V)$. 
Define 
\begin{equation}
M_{0} = 0 \quad  \text{and} \quad 
M_{n} = \sum_{k=1}^{n} \sigma_0(g_k, X_{k-1}),  \quad n\geq 1.  \label{mart001}
\end{equation}
By the Markov property, we have $\mathbb{E}_{\bb Q_s^x} (M_k |\mathscr{F}_{k-1} ) = M_{k-1}$ 
and hence $\left( M_{n},\mathscr{F}_{n}\right) _{n\geq 0}$ is a $0$ mean $\mathbb{Q}_s^x$-martingale.
The following lemma shows that the difference $S_{n}-M_{n}$ is bounded.

\begin{lemma} \label{Martingale approx} 
Assume conditions \ref{Condi-density}, \ref{Condi-density-invariant}, 
and $\kappa'(s) = 0$ for some $s \in  I_{\mu}$.  
Then, there exists a constant $c >0$ such that for any $x\in \bb P(V)$, 
\begin{align}
& \sup_{n\geq 0} \left| S_{n}-M_{n} \right| \leq c,  \quad  \mathbb Q_s^x\mbox{-a.s.}  \label{martingale01}\\
& \sup_{n\geq 0} \left| S_{n}^* - M_{n}^* \right| \leq c,  \quad  \mathbb Q_s^{x,*}\mbox{-a.s.}  \label{martingale02}
\end{align}
\end{lemma}

\begin{proof}
By \eqref{Marting_app_01}, we have
\begin{align*}
\sum_{k=1}^{n} \sigma_0(g_k, X_{k-1}) 
= \sum_{k=1}^{n} \sigma(g_k, X_{k-1}) - \theta(X_0) + \theta(g_n \!\cdot\! X_{n-1}),  
\end{align*}
where $X_0 = x$. 
Taking into account \eqref{mart001}, we get $M_n = S_n - \theta(X_0) + \theta(g_n \!\cdot\! X_{n-1}).$ 
Since the function $\theta \in \mathscr B$ is bounded, we get \eqref{martingale01}. 
The proof of \eqref{martingale02} is similar. 
\end{proof}

The following result is a consequence of Burkholder's inequality. 

\begin{lemma}\label{Lp boud martingales} 
Assume conditions \ref{Condi-density}, \ref{Condi-density-invariant}, 
and $\kappa'(s) = 0$ for some $s \in  I_{\mu}$.  
Then, for any $p>2$, we have 
\begin{align*}
& \sup_{n \geq 1} {1\over { n^{p/2} }}\sup_{x \in \bb P(V)}   \mathbb{E}_{\bb Q_s^x} (|M_n|^{p}) < +\infty,  \notag\\
& \sup_{n \geq 1} {1\over { n^{p/2} }}\sup_{x \in \bb P(V)}   \mathbb{E}_{\bb Q_s^{x,*} } (|M_n^*|^{p}) < +\infty. 
\end{align*} 
\end{lemma}

\begin{proof}
Denote $\xi_k = \sigma_0(g_k, X_{k-1})$ for $1 \leq k \leq n$. By Burkholder's inequality and H\"{o}lder's inequality, we get
\begin{align*}
\mathbb{E}_{\bb Q_s^x} (|M_n|^p) 
 \leq  c_p \mathbb{E}_{\bb Q_s^x} \left| \sum_{k=1}^{n} \xi_k^2 \right|^{p/2}  
 \leq  c_p  n^{p/2-1} \mathbb{E}_{\bb Q_s^x} \sum_{k=1}^{n} \left\vert \xi _{k}\right\vert ^{p}
  \leq  c_p  n^{p/2}  \sup_{1\leq k\leq n} \mathbb{E}_{\bb Q_s^x} (\left| \xi_{k} \right|^{p}). 
\end{align*}
Since there exists a constant $c>0$ such that for all $x \in \bb P(V)$, 
\begin{align*}
\sup_{1\leq k\leq n} \mathbb{E}_{\bb Q_s^x} (\left| \xi_{k} \right|^{p})
\leq \mathbb{E}_{\bb Q_s^x} (\log^p N(g_1))
\leq c \bb E \left[ \|g_1\|^s \log^p N(g_1) \right], 
\end{align*} 
the first inequality follows.
The second inequality can be proved in the same way. 
\end{proof}

\subsection{Non-asymptotic local limit theorems}

In the following we establish effective local limit theorems 
for products of random matrices under a change of measure $\bb Q_s^x$.
Our results are non-asymptotic, i.e.\ they are written in the form of precise 
upper and lower bounds which hold for any fixed $n\geq 1$.
Besides, we  consider a general target function $h$ on the pair $(X_n, S_n)$ 
and this plays a crucial role for establishing conditioned local limit theorems in Section \ref{Sect-CLLT}. 
The main difficulty is to give the explicit dependence of the remainder terms on the target function $h$. 
The following lemma is taken from \cite[Lemma 5.3]{GQX21}. 

\begin{lemma}[\cite{GQX21}]\label{Lem_Measurability}
Let $h$ be a real-valued function on $\bb P(V) \times \bb R$ such that
\begin{enumerate}
\item[(1)]   For any $t \in \bb R$, the function $x \mapsto h(x, t)$ is continuous on $\bb P(V)$.  
\item[(2)]  For any $x \in \bb P(V)$, the function $t \mapsto h(x, t)$ is measurable on $\bb R$.
 \end{enumerate}
Then, the function $(x, t) \mapsto h(x,t)$ is measurable on $\bb P(V) \times \bb R$ and 
the function $t \mapsto \| h (\cdot, t) \|_{ \mathscr B }$ is measurable on $\bb R$. 
Moreover, if the integral $\int_{\bb R} \| h (\cdot, t) \|_{ \mathscr B } dt$ is finite, 
we define the partial Fourier transform $\widehat h$ of $h$ by setting for any $x \in \bb P(V)$ and $u \in \bb R$,
\begin{align}\label{def-Fourier-h}
\widehat h(x, u) = \int_{\bb R} e^{-itu} h (x, t) dt.
\end{align}
This is a continuous function on $\bb P(V) \times \bb R$. 
In addition, for every $u \in \bb R$, 
the function $x \mapsto \widehat h(x, u)$ is continuous
and $\| \widehat h (\cdot, u) \|_{ \mathscr B } \leq \int_{\bb R} \| h (\cdot, t) \|_{ \mathscr B } dt$. 
\end{lemma}

We denote by $\mathscr H$ the set of real-valued functions on $\bb P(V) \times \bb R$ 
such that conditions (1) and (2) of Lemma \ref{Lem_Measurability} hold and 
the integral $\int_{\bb R} \| h (\cdot, t) \|_{\mathscr B} dt$ is finite. 
For any compact set $K\subset \bb R $, 
denote by $\mathscr H_{K}$ the set of functions $h \in \mathscr H$
such that the Fourier transform $\widehat h(x,\cdot)$ has a support contained in $K$  for any $x \in \bb P(V)$. 
Below, for $h \in \mathscr H$,  we use the notation
\begin{align*} 
\| h \|_{\mathscr H} =  \int_{\bb R} \| h(\cdot, u) \|_{\mathscr B}   du,
\qquad
\| h \|_{\pi_s \otimes \Leb} =  \int_{ \bb P(V) \times \bb R}  | h(x, u) |  \pi_s(dx) du. 
\end{align*}
Let $\phi(y) = \frac{1}{\sqrt{2 \pi}} e^{- y^2 /2}$, $y \in \bb R$ be the standard normal density function. 

\begin{theorem}\label{Theor-Loc-Improved}
Assume condition \ref{Condi-density} and $\kappa'(s) = 0$ for some $s \in  I_{\mu}$. 
Let $K\subset \bb R$ be a compact set of $\bb R$. 
Then there exists a constant $c_K >0$ such that for any $h\in \mathscr H_{K}$, 
$n\geq 1$, $x \in \bb P(V)$ and $y\in \mathbb{R}$, 
\begin{align}\label{LLT_bb}
\left| \sigma_s \sqrt{n} \, \EQ h \left(X_n, y - S_n \right)  - \int_{\bb P(V) \times \bb R} 
 h\left( x', y'\right) \phi \left( \frac{y-y'}{\sigma_s \sqrt{n}}\right)  \pi_s(dx')  dy'  \right|  
 \leq  \frac{c_K}{\sqrt{n}}  \|h\|_{\scr H}.   
\end{align}
\end{theorem}

\begin{proof}
For the sake of ease in exposition, 
we assume that $\sigma_s = 1.$ 
By the Fourier inversion formula, with the notation \eqref{def-Fourier-h} it holds that  
\begin{align*}
h( x, y) = \frac{1}{2\pi } \int_{\mathbb{R}}  e^{i ty}
\widehat{h}(x, t) dt,  \quad  x \in \bb P(V),  \   y\in \mathbb{R}.
\end{align*} 
By Fubini's theorem, this implies that for any $x \in \bb P(V)$ and $y\in \mathbb{R}$, 
\begin{align*}
\EQ h(X_n, y - S_{n} )  
= \frac{1}{2 \pi} \int_{\bb R} e^{i ty} \EQ \left[ e^{- i t S_n } \widehat{h} (X_n, t) \right] dt.
\end{align*}
By a change of variable $t=\frac{u}{\sqrt{n}}$, we obtain
\begin{align}\label{locI-start}
  \sqrt{n} \, \EQ h(X_n, y - S_{n} )   
&  = \frac{1}{2\pi }\int_{ \mathbb{R} } e^{ i u \frac{y}{\sqrt{n}} } 
\EQ \left[ e^{ - i u \frac{S_n}{\sqrt{n}} } \widehat{h} \left( X_n, \frac{u}{\sqrt{n}} \right) \right] du  \notag \\
& = \frac{1}{2\pi }\int_{\mathbb{R}}  e^{i u \frac{y}{\sqrt{n}}}   \widehat{\phi}(u)
\left[  \int_{\bb P(V)}  \widehat{h}\left(x', \frac{u}{\sqrt{n}}\right)  \pi_s(dx') \right] du  + I(x,y),
\end{align}
where $\widehat{\phi}(u) = e^{- u^2/2}$ and
\begin{equation*}
I(x,y) = \frac{1}{2\pi }\int_{\mathbb{R}} e^{i u \frac{y}{\sqrt{n}}} J(x,u) du,   \quad  x \in \bb P(V),  \   y\in \mathbb{R}, 
\end{equation*}
with 
\begin{align*}
J(x,u) = \EQ \left[ e^{ - i u \frac{S_n}{\sqrt{n}} } \widehat{h} \left( X_n, \frac{u}{\sqrt{n}} \right) \right]
 -  \widehat{\phi}(u) \int_{\bb P(V)}
\widehat{h}\left(x', \frac{u}{\sqrt{n}}\right)  \pi_s(dx'). 
\end{align*}
For the first term in \eqref{locI-start},  since $\sqrt{n} \widehat{\phi} (\sqrt{n} \, \cdot)$ is the Fourier transform of $\phi(\frac{\cdot}{\sqrt{n}})$. 
using the change of variable $u' = \frac{u}{\sqrt{n}}$ and the Fourier inversion formula,  we get
\begin{align}
 \frac{1}{2\pi }\int_{\mathbb{R}}  e^{i u \frac{y}{\sqrt{n}}} 
\widehat{h}\left(x', \frac{u}{\sqrt{n}}\right) \widehat{\phi}(u) du 
& = \frac{1}{2\pi }\int_{\mathbb{R}}  e^{i u' y} 
\widehat{h}\left(x', u'\right) \sqrt{n}  \widehat{\phi}( \sqrt{n} u') du'  \notag\\
& = \int_{\mathbb{R}} h\left( x', y'\right) \phi \left( \frac{y-y'}{\sqrt{n}}\right) dy'.  \label{locI-002}
\end{align}
Therefore, to establish \eqref{LLT_bb}, it remains to prove that there exists a constant $c_K >0$ 
such that for any $h\in \mathscr H_{K}$, 
$n\geq 1$, $x \in \bb P(V)$ and $y\in \mathbb{R}$, 
\begin{align} 
|I(x,y)| \leq  \frac{c_K}{\sqrt{n}} \|h\|_{\scr H}.   \label{locI-final}
\end{align}
Now we are going to show \eqref{locI-final}. 
Let $\delta >0$ be a sufficiently small constant. We decompose the integral $I(x,y)$ into three parts: 
$I(x,y)=I_{1}+I_{2}+I_{3},$ where
\begin{align*}
I_{1} & = \frac{1}{2\pi }\int_{\left\vert u\right\vert \leq \delta \sqrt{n}}
 e^{i u \frac{y}{\sqrt{n}}} J(x, u) du,   \notag\\
I_{2} & = -\frac{1}{2\pi } \int_{|u| >\delta \sqrt{n} }   e^{i u \frac{y}{\sqrt{n}}}   \widehat{\phi}(u)
 \left[ \int_{\bb P(V)} \widehat{h}\left(x', \frac{u}{\sqrt{n}} \right)   \pi_s(dx')  \right]  du,   \notag\\
I_{3} & =\frac{1}{2\pi } \int_{\left\vert u\right\vert >\delta \sqrt{n} }
  e^{i u \frac{y}{\sqrt{n}}}  
\EQ \left[ e^{- i u \frac{S_n }{\sqrt{n}} } \widehat{h} \left( X_n, \frac{u}{\sqrt{n}} \right) \right] du.
\end{align*}

\textit{Bound of $I_1.$}
By Lemma \ref{Lem_Perturbation}, we have that for any $\left\vert u\right\vert \leq \delta \sqrt{n}$, 
\begin{align*}
 J(x,u) & = Q^n_{s, \frac{- i u}{\sqrt{n}}} \widehat{h} \left(\cdot, \frac{u}{\sqrt{n}} \right)(x)
    - \widehat{\phi}(u)  \int_{\bb P(V)}
      \widehat{h}\left(x', \frac{u}{\sqrt{n}}\right)  \pi_s(dx')   \notag\\
& = J_1(x, u) + J_2(x, u),
\end{align*}
where 
\begin{align*}
& J_1(x, u) = \lambda^{n}_{s, \frac{-i u}{\sqrt{n}} } \Pi_{s, \frac{-i u}{\sqrt{n}} }  
     \widehat{h} \left(\cdot, \frac{u}{\sqrt{n}} \right)(x) 
   -  \widehat{\phi}(u) \int_{\bb P(V)}
      \widehat{h}\left(x', \frac{u}{\sqrt{n}}\right)  \pi_s(dx'),   \notag\\
& J_2(x, u) =  N^{n}_{s, \frac{-i u}{\sqrt{n}} } \widehat{h} \left(\cdot, \frac{u}{\sqrt{n}} \right)(x). 
\end{align*}
For the first term $J_1(x, u)$, since $\Pi_{s,0} (\varphi) = \pi_s(\varphi)$, we have 
\begin{align*}
J_1(x, u) 
& =  \left( \lambda^{n}_{s, \frac{-i u}{\sqrt{n}} } - \widehat{\phi}(u) \right) 
  \Pi_{s, \frac{-i u}{\sqrt{n}} }  
     \widehat{h} \left(\cdot, \frac{u}{\sqrt{n}} \right)(x)  
 + \widehat{\phi}(u) \left( \Pi_{s, \frac{-i u}{\sqrt{n}} } - \Pi_{s, 0} \right) 
   \widehat{h} \left(\cdot, \frac{u}{\sqrt{n}} \right)(x)   \notag\\
& =: K_1 + K_2. 
\end{align*}
For $K_1$, since $\widehat{\phi}(u) = e^{- u^2/2}$, 
using \eqref{relationlamkappa001} and following the proof of (4.11) in \cite{GLL20}, one can verify that 
there exists a constant $c>0$ such that for all $|u| \leq \delta \sqrt{n}$ and $n \geq 1$, 
\begin{align*}
\left| \lambda^{n}_{s, \frac{-i u}{\sqrt{n}} } - \widehat{\phi}(u) \right|
\leq \frac{c}{\sqrt{n}} e^{- \frac{u^2}{4}}. 
\end{align*}
By Lemma \ref{Lem_Perturbation}, the mapping 
$t \mapsto \Pi_{s, i t}: (-\delta, \delta) \to \mathscr{L(B,B)}$ is analytic, hence
there exists a constant $c>0$ such that
\begin{align*}
\sup_{x \in \bb P(V)} \sup_{|u| \leq \delta \sqrt{n}}
\left| \Pi_{s, \frac{-i u}{\sqrt{n}} } \widehat{h} \left(\cdot, \frac{u}{\sqrt{n}} \right)(x) \right|
 \leq  \sup_{|y'| \leq \delta }  \left\| \Pi_{s,iy'} \right\|_{\mathscr{B}\rightarrow\mathscr{B}}
    \left\| \widehat{h} \left(\cdot, y' \right) \right\|_{ \mathscr{B} }  
 \leq  c   \|h\|_{\scr H}.  
\end{align*}
It follows that $K_1 \leq \frac{c}{\sqrt{n}} e^{- \frac{u^2}{4}} \|h\|_{\scr H}$. 
For $K_2$, by Lemma \ref{Lem_Perturbation}, we get
$K_2 \leq c \frac{|u|}{\sqrt{n}}  e^{- \frac{u^2}{2}} \|h\|_{\scr H} \leq  \frac{c}{\sqrt{n}} e^{- \frac{u^2}{4}}  \|h\|_{\scr H}$ 
and hence there exists a constant $c>0$ such that for any $x \in \bb P(V)$ and $\left\vert u\right\vert \leq \delta \sqrt{n}$, 
\begin{align}
J_1(x, u) \leq \frac{c}{\sqrt{n}} e^{- \frac{u^2}{4}}  \|h\|_{\scr H}.   \label{Bound_J1_aa}
\end{align}
For $J_2(x, u)$, 
using \eqref{OperatorNbound} we get that there exist constants $c, c'>0$ such that
for any $x \in \bb P(V)$ and $\left\vert u\right\vert \leq \delta \sqrt{n}$, 
\begin{align}
J_2(x, u) 
\leq   c e^{-c'n}   \sup_{|y'| \leq \delta }   \left\| \widehat{h} \left(\cdot, y' \right) \right\|_{ \mathscr{B} }  
\leq c e^{-c'n} \|h\|_{\scr H}.    \label{Bound_J2_aa}
\end{align}
Therefore, combining \eqref{Bound_J1_aa} and \eqref{Bound_J2_aa}, 
we obtain the upper bound for $I_1$:
\begin{align}
|I_{1}|  \leq \frac{1}{2\pi }\int_{\left\vert u\right\vert \leq \delta \sqrt{n}} |J(x, u)| du
& \leq  c \left( \frac{1}{\sqrt{n}} +  e^{-c'n} \right)  \|h\|_{\scr H}    
 \leq  \frac{c}{\sqrt{n}} \|h\|_{\scr H}.   \label{I_1 final}
\end{align}

\textit{Bound of $I_{2}.$} 
Since $h\in \mathscr H_{K}$, we have
\begin{align}
\left\vert I_{2}\right\vert 
& \leq \frac{1}{2\pi } \int_{\left\vert u\right\vert >\delta \sqrt{n}}  \widehat{\phi}(u)
 \int_{\bb P(V) } 
\left\vert \widehat{h}\left(x', \frac{u}{\sqrt{n}}\right) \right\vert 
  \pi_s(dx') du  
 \leq  c e^{-\delta^{2}  n/4}   \|h\|_{\scr H}.   \label{locI2-final}
\end{align}

\textit{Bound of $I_{3}.$}
Since $h\in \mathscr H_{K}$, the Fourier transform $\widehat h(x,\cdot)$ has a support contained in $K$  for any $x \in \bb P(V)$. 
Applying Lemma \ref{Lem_StrongNonLattice}, we get that there exist constants $c_K, c'_K > 0$ such that
\begin{align}
\left\vert I_{3}\right\vert 
 \leq c_K e^{-n c'_K} \|h\|_{\scr H}.   \label{locI3-final}   
\end{align}

Collecting the bounds (\ref{I_1 final}), (\ref{locI2-final}) and (\ref{locI3-final}), we obtain \eqref{locI-final} and thus 
conclude the proof of \eqref{LLT_bb}. 
\end{proof}

Now we give an extension of Theorem \ref{Theor-Loc-Improved} for functions $h$ with
non-integrable Fourier transforms. 
For any $\ee >0$ and any nonnegative measurable function $h \in \scr H$, we denote by $h_{\ee}$ a measurable function such that 
for any $x \in \bb P(V)$ and $t \in \bb R$, 
it holds that $h(x, t) \leq h_{\ee}(x, t +v)$ for all $|v| \leq \ee$. 
In this case we simply write $h \leq_{\ee} h_{\ee}$ or $h_{\ee} \geq_{\ee} h$.
Similarly, we denote by $h_{-\ee}$ a measurable function such that $h(x, t) \geq h_{-\ee}(x, t +v)$ 
for any $x \in \bb P(V)$, $t \in \bb R$ and $|v| \leq \ee$,
and we write $h_{-\ee} \leq_{\ee} h$ or $h \geq_{\ee} h_{-\ee}$.

In the proofs we make use of the following smoothing inequality 
(cf.\  \cite[Lemma 5.2]{GLL20} or \cite[Lemma 2.4]{GX21}). 
Denote by $\rho$ the nonnegative density function on $\bb R$, 
which is the Fourier transform of the function $(1 - |t|) \mathds 1_{\{ |t| \leq 1\}}$ for $t \in \bb R$. 
Set $\rho_{\ee}(u) = \frac{1}{\ee} \rho ( \frac{u}{\ee} )$ for $u \in \bb R$ and $\ee >0$.  
\begin{lemma}\label{smoothing-lemma-001}
Let $\ee \in (0, \frac{1}{8})$. Let $h: \bb R \to \bb R_+$ 
 be an integrable function and let $h_{-\ee}$ and $h_{\ee}$ be any measurable functions 
 such that $h_{-\ee} \leq_{\ee} h \leq_{\ee} h_{\ee}$. 
Then for any $u \in \bb R$,
\begin{align*}
h(u) \leq   (1 + 4 \ee)  h_{\ee} * \rho_{\ee^2} (u),  \quad  
h(u)  \geq h_{-\ee} * \rho_{\ee^2} (u) - \int_{|v| > \ee} h_{-\ee} \left( u- v \right) \rho_{\ee^2} (v) d v.
\end{align*}
\end{lemma}

Below, for any $h \in \mathscr H$, we use the notation 
\begin{align*}
h * \rho_{\ee^2} (x,t) = \int_{\bb R} h (x, t -v) \rho_{\ee^2}(v) dv,  
\quad  x \in \bb P(V),  \  t \in \bb R.  
\end{align*}

\begin{theorem} \label{Theorem Loc bounds}
Assume condition \ref{Condi-density} and $\kappa'(s) = 0$ for some $s \in  I_{\mu}$. 
Then, for any $\ee \in (0, \frac{1}{8})$, 
there exist constants $c, c_{\ee} >0$ such that for any nonnegative function $h$ and any function $h_{\ee} \in \mathscr H$, 
$n\geq 1$, $x \in \bb P(V)$ and $y \in \bb R$, 
\begin{align}
 \mathbb{E}_{\mathbb{Q}_s^x} h\left(X_n, y - S_{n} \right)   
 \leq  \frac{1}{\sigma_s \sqrt{n}}  \int_{\bb P(V) \times \bb R}    h_{\ee} \left( x', y'\right) 
  \phi \left( \frac{y-y'}{\sigma_s \sqrt{n}}\right)  \pi_s(dx')  dy'   
  +   \frac{c\ee}{\sqrt{n}} \| h_{\ee} \|_{\pi_s \otimes \Leb}  
     +  \frac{c_{\ee}}{ n }  \| h_{\ee} \|_{\mathscr H}  \label{LLT_Upper_aa}
\end{align}
and 
\begin{align}
 \mathbb{E}_{\mathbb{Q}_s^x} h\left(X_n, y - S_{n}  \right)   
 \geq  \frac{1}{\sigma_s \sqrt{n}}  \int_{\bb P(V) \times \bb R}    h_{-\ee} \left( x', y'\right) 
  \phi \left( \frac{y-y'}{\sigma_s \sqrt{n}}\right) \pi_s(dx')  dy'    
 -   \frac{c\ee}{\sqrt{n}} \| h \|_{\pi_s \otimes \Leb}  
     -  \frac{c_{\ee}}{ n }  \| h \|_{\mathscr H} . \label{LLT_Lower_aa}
\end{align}
\end{theorem}

\begin{proof}
We first prove the upper bound \eqref{LLT_Upper_aa}. 
By Lemma \ref{smoothing-lemma-001},  we have $h \leq  (1 + 4 \ee) h_{\ee} * \rho_{\ee^2}$
and hence
\begin{align}  \label{LLT_Inequ_Smooth}
 \mathbb{E}_{\mathbb{Q}_s^x} h\left(X_n, y - S_{n} \right) 
\leq   (1 + 4 \ee)  \EQ h_{\ee} * \rho_{\ee^2} \left( X_n, y - S_{n}  \right).   
\end{align}
Since the support of the function
$\widehat{h_{\ee} * \rho_{\ee^2}} (x, \cdot) = \widehat{h}_{\ee} (x, \cdot) \widehat{\rho}_{\ee^2}(\cdot)$
is included in $[- \frac{1}{\ee^2}, \frac{1}{\ee^2}]$ for any $x \in \bb P(V)$, 
by Theorem \ref{Theor-Loc-Improved}, there exists  $c_{\ee} >0$ 
such that for all $n\geq 1$, $x \in \bb P(V)$ and $y \in \bb R$, 
\begin{align}\label{LLT_Inequa_aa}
  \EQ h_{\ee} * \rho_{\ee^2} \left(X_n, y - S_n  \right)  
& \leq \frac{1 + 4 \ee}{\sigma_s \sqrt{n}}  \int_{\bb P(V) \times \bb R}   h_{\ee} * \rho_{\ee^2} \left( x', y'\right) 
 \phi \left( \frac{y' - y}{\sigma_s \sqrt{n}}\right) \pi_s(dx')  dy'  +   \frac{c_{\ee}}{ n }   \|  h_{\ee} \|_{\scr H}   \notag\\
& = \frac{1 + 4 \ee}{\sigma_s \sqrt{n}}  \int_{\bb P(V) \times \bb R}   h_{\ee} * \rho_{\ee^2} \left( x, u + y \right) 
 \phi \left( \frac{u}{\sigma_s \sqrt{n}}\right) \pi_s(dx)  du  +   \frac{c_{\ee}}{ n }   \|  h_{\ee} \|_{\scr H}.      
\end{align}
By a change of variable and Fubini's theorem, we have for any $x \in \bb P(V)$, 
\begin{align} \label{LLT_B_Fubini_aa}
\frac{1}{ \sigma_s \sqrt{n} } 
\int_{\mathbb{R}}  h_{\ee} * \rho_{\ee^2} \left(x, u + y \right)  \phi \left(\frac{u}{ \sigma_s \sqrt{n} } \right)    du   
=   \int_{\bb R }  h_{\ee}(x, t + y) \phi_{\sigma_s \sqrt{n}} * \rho_{\ee^2} \left( t \right)  dt, 
\end{align}
where $\phi_{\sigma_s \sqrt{n}}(t) = \frac{1}{ \sigma_s \sqrt{2 \pi n}} e^{- \frac{t^2}{2 \sigma_s^2 n}}$, $t \in \bb R$. 
For brevity, denote $\psi(t)=\sup_{|v|\leq\ee} \phi_{\sigma_s \sqrt{n}} (t +v)$, $t \in \bb R$. 
Using the second inequality in Lemma \ref{smoothing-lemma-001}, we have 
\begin{align*} 
&  \int_{\bb R }   h_{\ee}(x, t + y)  \phi_{\sigma_s \sqrt{n}} * \rho_{\ee^2} \left( t \right)  dt  \notag\\
& \leq    \int_{\mathbb{R}}   h_{\ee}(x, t + y)  \psi \left( t \right)    dt    
   +   \int_{\mathbb{R}}   h_{\ee}(x, t + y)  \left[ \int_{\abs{v} \geq \ee} 
        \phi_{\sigma_s \sqrt{n}} \left( t - v \right) \rho_{\ee^2} (v) d v  \right]   dt   
=: J_1+J_2.  
\end{align*}
For $J_1$, by Taylor's expansion and the fact that the function $\phi'$ is bounded on $\bb R$,
we derive that
\begin{align} \label{LLT_B_J1} 
J_1 \leq   \frac{1}{\sigma_s \sqrt{n}}   \int_{\bb R}  h_{\ee}(x, t + y)  \phi \left( \frac{t}{\sigma_s \sqrt{n}} \right)    dt  
     +  \frac{c \ee}{\sqrt{n}}  \int_{\bb R}   h_{\ee} (x, t)  dt.  
\end{align}
For $J_2$, since $\phi_{\sigma_s \sqrt{n}} \leq \frac{c}{\sqrt{n}}$ and $\int_{|v| \geq \ee} \rho_{\ee^2} (v) d v \leq c \ee$, we get
\begin{align}\label{LLT_B_J2}
J_2 \leq   \frac{c}{\sqrt{n}} \int_{\mathbb{R}}  h_{\ee}(x, t + y)  \left[ \int_{|v| \geq \ee} \rho_{\ee^2} (v) d v \right]    dt  
      \leq  \frac{ c \ee}{\sqrt{n}}   \int_{\bb R}   h_{\ee}(x, t)  dt. 
\end{align}
Putting together \eqref{LLT_Inequ_Smooth}, 
\eqref{LLT_Inequa_aa}, \eqref{LLT_B_Fubini_aa}, \eqref{LLT_B_J1} and \eqref{LLT_B_J2}, 
we get \eqref{LLT_Upper_aa}. 

We next prove the lower bound \eqref{LLT_Lower_aa}. 
Since $h \geq_{\ee} h_{-\ee}$, 
using the second inequality in Lemma \ref{smoothing-lemma-001}, 
we get
\begin{align}  \label{LLT_Inequ_Smooth-002}
 \mathbb{E}_{\mathbb{Q}_s^x} h\left(X_n, y - S_{n} \right) 
 \geq  \mathbb{E}_{\mathbb{Q}_s^x} h_{-\ee} * \rho_{\ee^2} \left(X_n, y - S_{n} \right)   
   -    \int_{|v| \geq \ee}
  \mathbb{E}_{\mathbb{Q}_s^x}  h_{-\ee} \left(X_n, y - S_{n} - v \right) \rho_{\ee^2} (v) dv.  
\end{align}
For the first term,  by Theorem \ref{Theor-Loc-Improved}, 
there exists  $c_{\ee} >0$ 
such that for all $n\geq 1$, $x \in \bb P(V)$ and $y \in \bb R$, 
\begin{align}
 \mathbb{E}_{\mathbb{Q}_s^x} h_{-\ee} * \rho_{\ee^2} \left(X_n, y - S_{n} \right) 
& \geq \frac{1}{\sigma_s \sqrt{n}}  \int_{\bb P(V) \times \bb R}   h_{-\ee} * \rho_{\ee^2} \left( x', y'\right) 
 \phi \left( \frac{y' - y}{\sigma_s \sqrt{n}}\right) \pi_s(dx')  dy'  -  \frac{c_{\ee}}{ n }   \|  h_{-\ee} \|_{\scr H}   \notag\\
& =  \frac{1}{\sigma_s \sqrt{n}}  \int_{\bb P(V) \times \bb R} 
h_{-\ee} * \rho_{\ee^2}  \left(x, u + y \right) 
\phi \left(\frac{u}{\sigma_s \sqrt{n}} \right)    \pi_s(dx) du  
  -   \frac{c_{\ee}}{ n }   \|  h_{-\ee} \|_{\scr H}.    \label{LLT_Inequa_aa-002}
\end{align}
In the same way as in \eqref{LLT_B_Fubini_aa}, we have
\begin{align} \label{LLT_B_Fubini_bb}
\frac{1}{ \sigma_s \sqrt{n} }
\int_{\mathbb{R}}    h_{-\ee} * \rho_{\ee^2} \left(x, u + y \right)    \phi \left(\frac{u}{\sigma_s \sqrt{n}} \right)  du 
=  \int_{\bb R }  h_{-\ee} (x, t + y)  \phi_{\sigma_s \sqrt{n}} * \rho_{\ee^2} (t)  dt.
\end{align}
Using the first inequality in Lemma \ref{smoothing-lemma-001}, 
we have $\phi_{\sigma_s \sqrt{n}} *\rho_{\ee^2} (t) \geq (1 - c \ee) \psi(t)$, for $t\in \bb R$, 
where $\psi(t) = \inf_{|v| \leq \ee} \phi_{\sigma_s \sqrt{n}} (t + v)$. 
Proceeding in the same way as in \eqref{LLT_B_J1} and \eqref{LLT_B_J2}, 
we obtain that 
\begin{align}\label{LLT_B_Lower_Term1}
  \int_{\bb R }  h_{-\ee} (x, t + y)  \phi_{\sigma_s \sqrt{n}} * \rho_{\ee^2} (t)  dt  
 \geq  \frac{1}{\sigma_s \sqrt{n}}  \int_{\bb R}     h_{-\ee} \left(x, t + y \right)  \phi \left(\frac{t}{ \sigma_s \sqrt{n} } \right) dt
  - \frac{c \ee}{\sqrt{n}}  \int_{\bb R}   h_{-\ee} (x, t) dt.
\end{align}
Therefore, combining \eqref{LLT_Inequa_aa-002}, \eqref{LLT_B_Fubini_bb} and \eqref{LLT_B_Lower_Term1}, we get
\begin{align}\label{CLLT-First-term-lower}
 \mathbb{E}_{\mathbb{Q}_s^x} h_{-\ee} * \rho_{\ee^2} \left(X_n, y - S_{n} \right) 
& \geq  \frac{1}{\sigma_s \sqrt{n}}  \int_{\bb P(V) \times \bb R}     h_{-\ee} \left(x, t + y \right)  
     \phi \left(\frac{t}{ \sigma_s \sqrt{n} } \right)  \pi_s(dx) dt     \notag\\
& \quad  - \frac{c \ee}{\sqrt{n}}  \| h_{-\ee} \|_{\pi_s \otimes \Leb} 
   -   \frac{c_{\ee}}{ n }   \|  h_{-\ee} \|_{\scr H}. 
\end{align}
For the second term on the right hand side of \eqref{LLT_Inequ_Smooth-002}, 
using the upper bound \eqref{LLT_Upper_aa} and the fact that $h_{-\ee} \leq_{\ee} h$ and $\phi \leq 1$, 
we get that there exist constants $c, c_{\ee} >0$ such that for any $v \in \bb R$ and $n \geq 1$, 
\begin{align*}
 \mathbb{E}_{\mathbb{Q}_s^x}  h_{-\ee} \left(X_n, y - S_{n} - v \right)   
& \leq    \frac{1}{\sigma_s \sqrt{n}}  \int_{\bb P(V) \times \bb R}    h \left( x, t \right) 
  \phi \left( \frac{y - v -t}{\sigma_s \sqrt{n}}\right)  \pi_s(dx)  dt 
     +   \frac{c\ee}{\sqrt{n}} \| h \|_{\pi_s \otimes \Leb}  
     +  \frac{c_{\ee}}{ n }  \| h \|_{\mathscr H}    \notag\\
& \leq   \frac{c}{\sqrt{n}} \| h \|_{\pi_s \otimes \Leb}   +  \frac{c_{\ee}}{ n }  \| h \|_{\mathscr H}.  
\end{align*}
This, together with the fact that $\int_{|v| \geq \ee} \rho_{\ee^2} (v) d v \leq c \ee$, implies
\begin{align} \label{LLTlowbnd-004}
 \int_{|v| \geq \ee}
  \mathbb{E}_{\mathbb{Q}_s^x}  h_{-\ee} \left(X_n, y - S_{n} - v \right) \rho_{\ee^2} (v) dv 
  \leq    \frac{c \ee}{\sqrt{n}} \| h \|_{\pi_s \otimes \Leb}   +  \frac{c_{\ee}}{ n }  \| h \|_{\mathscr H}.  
\end{align}
Substituting \eqref{CLLT-First-term-lower} and \eqref{LLTlowbnd-004} into  \eqref{LLT_Inequ_Smooth-002} 
and using the fact that $\| h_{-\ee} \|_{\pi_s \otimes \Leb} \leq  \| h \|_{\pi_s \otimes \Leb}$ and
$ \|  h_{-\ee} \|_{\scr H} \leq  \|  h \|_{\scr H}$,  
we obtain the lower bound \eqref{LLT_Lower_aa}. 
\end{proof}

\section{Conditioned local limit theorems}\label{Sec-CLLT}

\subsection{Bounds in the conditioned local limit theorems}

The next lemma shows that the $\| \cdot \|_{\pi_s \otimes \Leb}$ norm of the 
probability $\bb Q_s^x (y - S_n \in [a,b], \, \tau_y > n)$ is of order $n^{-1/2}$. 
This turns out to be one of the key points in the sequel.
The proof is based upon the duality lemma (Lemma \ref{Lem_Duality}) and
the bound for the exit time $\tau_y^*$ for the dual random walk $S_n^*$ (Theorem \ref{Th_Asymptotic_ExitTime_dual}).
\begin{lemma}\label{lemma Qng integr}
Assume conditions \ref{Condi-density}, \ref{Condi-density-invariant},
and $\kappa'(s) = 0$ for some $s \in  I_{\mu}$. 
Then, there exists a constant $c>0$ such that for any $n \geq 1$ and $0 \leq a < b < \infty$, 
\begin{equation*}
 \int_{\bb P(V) \times \bb{R}_+} 
 \bb Q_s^x (y - S_n \in [a,b], \, \tau_y > n) \pi_s(dx) dy 
\leq \frac{c}{\sqrt{n}} (b-a)(b + a+1).
\end{equation*}
\end{lemma}

\begin{proof}
Using the duality lemma (Lemma \ref{Lem_Duality}) and Fubini's theorem, 
we get that for $n \geq 1$, 
\begin{align*}
J := \int_{\bb P(V) \times \bb{R}_+} 
 \bb Q_s^x (y - S_n \in [a,b], \, \tau_y > n)  \pi_s(dx) dy   \
   =  \int_{\bb P(V) \times \bb{R}_+}   \mathds 1_{[a,b]} (z)
 \mathbb{Q}_s^{x,*} \left( \tau _{z}^{\ast }>n\right) \pi_s(dx) dz. 
\end{align*}
By Theorem \ref{Th_Asymptotic_ExitTime_dual}, there exists a constant $c$ such that
$\mathbb{Q}_s^{x,*} \left( \tau _{z}^{\ast }>n\right) \leq c\frac{1+z}{\sqrt{n}}$ 
 for any $x \in \bb P(V)$, $z\geq 0$ and $n \geq 1$. 
Therefore,
\begin{align*}
J  \leq \frac{c}{\sqrt{n}} \int_{\bb P(V) \times \bb{R}_+} 
\mathds 1_{[a,b]} (z) (1+z)  \pi_s(dx) dz 
 =  \frac{c}{\sqrt{n}} (b-a)(b+a+1), 
\end{align*}
which ends the proof of the lemma. 
\end{proof}

\begin{lemma}\label{Lem_Integral_LLT}
Assume condition \ref{Condi-density} and $\kappa'(s) = 0$ for some $s \in  I_{\mu}$. 
Then, for any $\ee \in [0, \frac{1}{2})$, there exists a constant $c_{\ee} >0$ such that for any $n \geq 2$ 
and $- \sqrt{n} \log^{1-\ee} n \leq a < b \leq \sqrt{n} \log^{1-\ee} n$,
\begin{align*}
J: = \int_{\bb R} \sup_{x \in \bb P(V)}  \mathbb{Q}_s^x \left(  y - S_n  \in [a, b]  \right) dy
\leq  c_{\ee}  (b - a + 1) \log^{1-\ee} n. 
\end{align*}
\end{lemma}

\begin{proof}
We first decompose the integral into three parts: $J = J_1 + J_2 + J_3$, where 
\begin{align*}
& J_1 = \int_{|y| \leq 2 \sqrt{n} \log^{1-\ee} n} 
      \sup_{x \in \bb P(V)} \mathbb{Q}_s^x \left(  y - S_n  \in [a, b] \right) dy,  \notag\\
& J_2 = \int_{2 \sqrt{n} \log^{1-\ee} n < |y| \leq n^2} 
      \sup_{x \in \bb P(V)} \mathbb{Q}_s^x \left(  y - S_n  \in [a, b] \right) dy,   \notag\\
& J_3 = \int_{|y| > n^2}  \sup_{x \in \bb P(V)}  \mathbb{Q}_s^x \left(  y - S_n  \in [a, b] \right) dy. 
\end{align*}

\textit{Bound of $J_1$.}
By the local limit theorem \eqref{LLT_Upper_aa}, 
there exists a constant $c >0$ such that for any $a < b$ and $n \geq 1$, 
\begin{align}
J_1 &  \leq  
  4  \log^{1-\ee} n \sup_{x\in \bb P(V)} \sup_{y\in \mathbb{R}} 
    \sqrt{n} \,  \mathbb{Q}_s^x \left(  y - S_n  \in [a, b] \right)  \notag\\
& \leq  4  \log^{1-\ee} n \left[  (b - a + c \ee) 
    +   \frac{c_{\ee}}{\sqrt{n}}  (b - a + c \ee)  \right]  \notag\\
& \leq   c (b - a + 1)  \log^{1-\ee} n.    \label{Pf_dual_Bound_I1}
\end{align}

\textit{Bound of $J_2$.}
When $y \in [2 \sqrt{n} \log^{1-\ee} n, n^2]$ and $b \leq  \sqrt{n} \log^{1-\ee} n$, 
there exist constant $c, c' >0$ such that for any $x \in \bb P(V)$ and $y \in [2 \sqrt{n} \log^{1-\ee} n, n^2]$, 
\begin{align}\label{Inequa-Moderate}
\mathbb{Q}_s^x \left(  y - S_n  \in [a, b] \right)
&  \leq  \mathbb{Q}_s^x \left( S_n  \geq  \sqrt{n} \log^{1-\ee} n \right)
\leq c e^{- c' \log^{2-2\ee} n}, 
\end{align}
where in the last inequality we used the upper tail moderate deviation asymptotic for $S_n$ 
under the changed measure $\mathbb{Q}_s^x$ (cf. \cite{XGL19a}). 
Since $\ee \in [0, \frac{1}{2})$, it follows that
\begin{align}
 \int_{2 \sqrt{n} \log^{1-\ee} n}^{n^2} \sup_{x \in \bb P(V)} \mathbb{Q}_s^x \left(  y - S_n \in [a, b] \right) dy   
 \leq c  e^{- c' \log^{2-2\ee} n}.  \label{Upper_I2_hh01}
\end{align}

When $y \in [-n^2, - 2 \sqrt{n} \log^{1-\ee} n]$ and $a \geq - \sqrt{n} \log^{1-\ee} n$, 
there exist constant $c, c' >0$ such that for any $x \in \bb P(V)$ and $y \in [-n^2, -2 \sqrt{n} \log^{1-\ee} n]$, 
\begin{align*}
\mathbb{Q}_s^x \left(  y - S_n  \in [a, b] \right)
&  \leq  \mathbb{Q}_s^x \left( S_n  \leq - \sqrt{n} \log^{1-\ee} n \right) 
\leq c e^{- c' \log^{2-2\ee} n}, 
\end{align*}
where in the last inequality we used 
the lower tail moderate deviation asymptotic for $S_n$ under the changed measure $\mathbb{Q}_s^x$ (cf. \cite{XGL19a}). 
It follows that
\begin{align}
\int_{-n^2}^{- 2 \sqrt{n} \log^{1-\ee} n}  \sup_{x \in \bb P(V)}  \mathbb{Q}_s^x \left(  y - S_n  \in [a, b] \right) dy
 \leq  c e^{- c' \log^{2-2\ee} n}.  \label{Upper_I2_hh02}
\end{align}
Combining \eqref{Upper_I2_hh01} and \eqref{Upper_I2_hh02},  
there exist constants $c, c' >0$ such that for any 
 $-  \sqrt{n} \log^{1-\ee} n \leq a < b \leq  \sqrt{n} \log^{1-\ee} n$, 
\begin{align}
J_2 \leq c e^{- c' \log^{2-2\ee} n}.  \label{Pf_dual_Bound_I2}
\end{align}

\textit{Bound of $J_3$.}
Since $y > n^2$ and $b \leq \sqrt{n} \log^{1-\ee} n$, 
by the Markov inequality, we have for sufficiently small $\delta >0$, 
\begin{align*}
\mathbb{Q}_s^x \left(  y - S_n \in [a, b] \right)
 \leq  \mathbb{Q}_s^x \left(  S_n \geq  \frac{y}{2} \right) 
 \leq  e^{- \frac{\delta}{2} y}  \bb E_{\bb Q_s^x}  e^{\delta  S_n}. 
\end{align*}
Using \eqref{Formu_ChangeMea}, we get that for $s \in I_{\mu}^+$, 
\begin{align*}
 \sup_{x \in \bb P(V)} \bb E_{\bb Q_s^x}  e^{\delta S_n}
& =  \sup_{x \in \bb P(V)}  \frac{1}{ \kappa(s)^{n} r_{s}(x) } \mathbb{E}_x \Big[  r_{s}(X_n)  e^{(s + \delta) S_n}  \Big]  \notag\\
& \leq c e^{n \Lambda(s)} \sup_{x \in \bb P(V)} \bb E_x \left( e^{(s+\delta) S_n} \right) 
 \leq  c  e^{n \Lambda(s) }  \left[ \bb E \left( e^{(s+\delta) \log \|g_1\|} \right)  \right]^n
 \leq  c  e^{c'n}. 
\end{align*}
Similarly, for $s \in I_{\mu}^-$,  it also holds that $\sup_{x \in \bb P(V)} \bb E_{\bb Q_s^x}  e^{\delta S_n} \leq c  e^{c'n}$
by using the fact that 
$\sup_{x \in \bb P(V)} \bb E_x ( e^{(s+\delta) S_n} ) \leq [ \bb E ( e^{-(s+\delta) \log \|g_1^{-1}\|} )  ]^n$.   
Hence there exists a constant $\delta >0$ such that for any $s \in  I_{\mu}$,  
\begin{align*}
\int_{n^2}^{\infty}  \sup_{x \in \bb P(V)}  \mathbb{Q}_s^x \left(  y - S_n  \in [a, b] \right) dy
\leq  e^{- \frac{\delta}{2} n^2}  \sup_{x \in \bb P(V)}  \bb E_{\bb Q_s^x}  e^{\delta S_n} 
\leq e^{- \frac{\delta}{4} n^2}.  
\end{align*}
In the same way, we can show that 
\begin{align*}
\int_{-\infty}^{- n^2}  \sup_{x \in \bb P(V)} \mathbb{Q}_s^x \left(  y - S_n  \in [a, b] \right) dy
\leq  e^{- \frac{\delta}{4} n^2}. 
\end{align*}
Therefore, there exists a constant $\delta > 0$ such that for any 
 $- \sqrt{n} \log^{1-\ee} n \leq a < b \leq \sqrt{n} \log^{1-\ee} n$, 
\begin{align}
J_3 \leq 2 e^{- \frac{\delta}{4} n^2}. \label{Pf_dual_Bound_I3}
\end{align}
Putting together \eqref{Pf_dual_Bound_I1}, \eqref{Pf_dual_Bound_I2} and \eqref{Pf_dual_Bound_I3}
concludes the proof of the lemma. 
\end{proof}

Below we state two lemmata based upon each other and providing successively improved bounds of the integral
\begin{align}\label{DefInIntegral}
I_n : = \int_{\bb R} \sup_{x \in \bb P(V)}  \mathbb{Q}_s^x \big(  y - S_n  \in [a, b],  \,  \tau_y > n \big) dy. 
\end{align}
The estimate of $I_n$ is one of the difficult points of the paper.

\begin{lemma}\label{Lem_Integral_LLTnew}
Assume conditions \ref{Condi-density}, \ref{Condi-density-invariant} 
and $\kappa'(s) = 0$ for some $s \in  I_{\mu}$. 
Then, for any $\ee \in [0, \frac{1}{2})$, there exists a constant $c_{\ee} >0$ such that for any $n \geq 2$ 
and $- \sqrt{n} \log^{1-\ee} n \leq a < b \leq \sqrt{n} \log^{1-\ee} n$,
\begin{align*}
& I_n \leq c_{\ee} \frac{ \log^{2 - 2\ee} n }{\sqrt{n}} (b-a+ 1) (b + a +1). 
\end{align*}
\end{lemma}
\begin{proof}
In view of \eqref{DefInIntegral}, we split the integral $I_n$ into two parts: 
\begin{align}
I_n  & = \int_{|y| \leq 2 \sqrt{n} \log^{1-\ee} n } 
  \sup_{x \in \bb P(V)}    \mathbb{Q}_s^x \left(  y - S_n \in [a, b], \tau_y >n \right) dy \notag \\
& \quad + \int_{|y| > 2 \sqrt{n} \log^{1-\ee} n } 
  \sup_{x \in \bb P(V)}    \mathbb{Q}_s^x \left(  y - S_n \in [a, b], \tau_y >n \right) dy  \notag \\
& =: I_{n,1} + I_{n,2}.   \label{Decom-In-Lem2}
\end{align}
For the first term $I_{n,1}$, we use  the Markov property to get that for any
$m = [\frac{n}{2}]$ and $k = n - m$, 
\begin{align} \label{Qsx-bound-00a}
 \mathbb{Q}_s^x \left(  y - S_n \in [a, b] , \tau_y >n\right)  
& =  \int_{\bb P(V) \times \bb R_+} \mathbb{Q}_s^{x'} \left(  y' - S_m  \in [a, b] , \tau_{y'} >m\right)  
 \mathbb{Q}_s^x \left( X_k \in dx', y - S_k \in dy' , \tau_y >k\right) \notag\\
& \leq  \int_{\bb P(V) \times \bb R_+} h (x', y')   
\mathbb{Q}_s^x \left( X_k \in dx', y - S_k \in dy'\right) =: J_n(x, y),
\end{align}
where we denoted 
\begin{equation}\label{Def-h-xy-00a}
h \left(x', y'\right) =
\begin{cases}
\mathbb{Q}_s^{x'} \left(  y' - S_m  \in [a, b] , \tau_{y'} > m \right), \quad & x' \in \bb P(V), \,  y' > 0 \\
0,  \quad  & x' \in \bb P(V),  \,   y' \leq 0. 
\end{cases}
\end{equation}
By the local limit theorem (\eqref{LLT_Upper_aa} of Theorem  \ref{Theorem Loc bounds}), we get 
\begin{align}
\sup_{x \in \bb P(V)}  J_n(x, y)
\leq \frac{1}{\sigma_s \sqrt{k}}  \int_{\bb P(V) \times \bb{R}}   h_{\ee} \left( x', y'\right) 
  \phi \left( \frac{y-y'}{\sigma_s \sqrt{k}}\right)  \pi_s(dx')  dy'    
 +  \frac{c\ee}{\sqrt{k}} \| h_{\ee} \|_{\pi_s \otimes \Leb}  
     +  \frac{c_{\ee}}{ k }  \| h_{\ee} \|_{\mathscr H},    \label{LLT_Upper_aa-xxxx}
\end{align}
where we choose 
\begin{align}\label{Pf_Inequa_dd02}
h_{\ee}(x', y') = 
\begin{cases}
\mathbb{Q}_s^{x'} \left( y' - S_{m} \in [a - \ee, b + \ee], \tau_{y' + \ee} > m \right), \quad & x' \in \bb P(V), \,  y' > - \ee \\
0,  \quad  & x' \in \bb P(V),  \,   y' \leq - \ee. 
\end{cases}
\end{align}
Since $\phi \leq 1$ and $h_{\ee} (x', y') = 0$ when $y' \leq -\ee$, 
we get an upper bound for the first term in the right hand side of \eqref{LLT_Upper_aa-xxxx}: 
\begin{align}\label{Bound_h_2ee_008}
 \int_{\bb P(V) \times \bb R}   h_{\ee} \left( x', y'\right) 
  \phi \left( \frac{y-y'}{\sigma_s \sqrt{n}}\right)  \pi_s(dx')  dy'  
&  \leq  \int_{\bb P(V)} \int_{- \ee}^{\infty}   h_{\ee} \left( x', y'\right)  \pi_s(dx')  dy' 
   =  \int_{\bb P(V) \times \bb R_+}   h_{\ee} \left( x', t - \ee \right)  \pi_s(dx')  dt   \nonumber\\
& =  \int_{\bb P(V) \times \bb R_+}
     \mathbb{Q}_s^{x'} \left( t - S_{m} \in [a, b + 2 \ee], \tau_{t} > m \right) \pi_s(dx')  dt  \nonumber \\
& \leq \frac{c}{\sqrt{m}} (b-a + 1) (b + a +1), 
\end{align}
where in the last inequality we used Lemma \ref{lemma Qng integr}. 

For the second term in the right hand side of \eqref{LLT_Upper_aa-xxxx}, 
we proceed in the same way as for the first one to get that 
\begin{align}\label{Bound_h_2ee_008-second}
\frac{c\ee}{\sqrt{k}} \| h_{\ee} \|_{\pi_s \otimes \Leb}  
& \leq  \frac{c\ee}{\sqrt{km}} (b-a+  \ee) (b + a +1). 
\end{align}

For the third term in the right hand side of \eqref{LLT_Upper_aa-xxxx}, 
by \eqref{Pf_Inequa_dd02} and Lemma \ref{Lem_Integral_LLT}, 
there exists a constant $c >0$ such that for all $n \geq 2$ 
and $- \sqrt{n} \log^{1-\ee} n \leq a < b \leq \sqrt{n} \log^{1-\ee} n$,
\begin{align}\label{Bound_h_2ee_009}
\| h_{\ee} \|_{\mathscr H}
& = \int_{- \ee}^{\infty} \sup_{x' \in \bb P(V)}  
 \mathbb{Q}_s^{x'} \left( y' - S_m \in [a - \ee, b + \ee], \, \tau_{y' + \ee} > m  \right) dy'  \notag\\
& =   \int_{\bb R_+} \sup_{x' \in \bb P(V)}  \mathbb{Q}_s^{x'} \left( t - S_m \in [a, b + 2\ee], \, \tau_{t} > m  \right) dt  \nonumber\\ 
& \leq  c_{\ee} (b - a +  1) \log^{1-\ee} n,
\end{align}
where we used the fact that $m = [\frac{n}{2}]$ and the inequality in Lemma \ref{Lem_Integral_LLT} still holds
when $a$ and $b$ are replaced by their constant multiples. 
Substituting \eqref{Bound_h_2ee_008}, \eqref{Bound_h_2ee_008-second} and \eqref{Bound_h_2ee_009} into \eqref{LLT_Upper_aa-xxxx},
and recalling that $m = [\frac{n}{2}]$ and $k = n - m$, one has 
\begin{align*}
\sup_{x \in \bb P(V)}  J_n(x, y) 
& \leq  \frac{c}{n} (b-a+ 1) (b + a +1)   +  \frac{c \ee}{n} (b-a+ 1) (b + a +1)  
   +  c_{\ee}  \frac{\log^{1-\ee} n}{n}  (b - a + 1)   \\
& \leq   c_{\ee}  \frac{\log^{1-\ee} n}{n}  (b - a + 1) (b + a +1),  
\end{align*}
from which we get
\begin{align}
 I_{n,1} 
 \leq  \int_{|y| \leq 2 \sqrt{n} \log^{1-\ee} n } \sup_{x \in \bb P(V)}  J_n(x, y)  dy   
\leq  c_{\ee}  \frac{ \log^{2 - 2\ee} n }{\sqrt{n}}  (b - a + 1) (b + a +1).   \label{Bound_y_nlogn_01}
\end{align}
It was shown in the proof of Lemma \ref{Lem_Integral_LLT} (cf.\ \eqref{Pf_dual_Bound_I2} and \eqref{Pf_dual_Bound_I3}) that 
there exist constants $c, c' >0$ such that
for any $-  \sqrt{n} \log^{1-\ee} n \leq a < b \leq   \sqrt{n} \log^{1-\ee} n$, 
\begin{align}
I_{n,2} 
\leq  \int_{|y| > 2 \sqrt{n} \log^{1-\ee} n } 
  \sup_{x \in \bb P(V)}    \mathbb{Q}_s^x \left(  y - S_n \in [a, b] \right) dy 
\leq  c e^{- c' \log^{2-2\ee} n}.   \label{Bound_y_nlogn_02}
\end{align}
Putting together \eqref{Bound_y_nlogn_01} and \eqref{Bound_y_nlogn_02}, 
we conclude the proof of the lemma. 
\end{proof}

The convergence rate in Lemma \ref{Lem_Integral_LLTnew} can be improved  by repeating the same proof. 
Recall that $I_n$ is given by \eqref{DefInIntegral}.  

\begin{lemma}\label{Lem_Integral_LLTnew_new2}
Assume conditions \ref{Condi-density}, \ref{Condi-density-invariant} 
and $\kappa'(s) = 0$ for some $s \in  I_{\mu}$. 
Then, for any $\ee \in [0, \frac{1}{2})$, there exists a constant $c_{\ee} >0$ such that for any $n \geq 2$ 
and $- \sqrt{n} \log^{1-\ee} n \leq a < b \leq \sqrt{n} \log^{1-\ee} n$,  
\begin{align*}
 I_n \leq c_{\ee} \frac{ \log^{1- \ee} n }{\sqrt{n}} (b - a + 1) (b + a +1). 
\end{align*}
\end{lemma}

\begin{proof}
We repeat the same proof as in Lemma \ref{Lem_Integral_LLTnew}.
The only difference is that in \eqref{Bound_h_2ee_009},
we apply Lemma \ref{Lem_Integral_LLTnew}
instead of Lemma \ref{Lem_Integral_LLT}  to get that
there exists a constant $c >0$ such that for all $m \geq 2$ 
and $-  \sqrt{m} \log^{1-\ee'} m \leq  a < b \leq \sqrt{m} \log^{1-\ee'} m$ with $\ee' \in (0, \ee)$,
\begin{align}\label{Bound_h_2ee_009-Lem3}
\| h_{\ee} \|_{\mathscr H}
& =   \int_{\bb R} \sup_{x' \in \bb P(V)}  \mathbb{Q}_s^{x'} \left( y' - S_m \in [a - \ee, b + \ee], \, \tau_{y'} > m  \right) dy'  \nonumber\\ 
& \leq  c_{\ee} \frac{\log^{2 - 2\ee'} m}{\sqrt{m}}  (b - a +  1) (b + a + 1). 
\end{align}
Substituting \eqref{Bound_h_2ee_008}, \eqref{Bound_h_2ee_008-second} and \eqref{Bound_h_2ee_009-Lem3} 
into \eqref{LLT_Upper_aa-xxxx},
and taking into account that $\ee' \in (0, \ee)$, $m = [\frac{n}{2}]$ and $k = n - m$, one has 
for any $n \geq 2$ and $- \sqrt{n} \log^{1-\ee} n \leq a < b \leq \sqrt{n} \log^{1-\ee} n$,  
\begin{align*}
\sup_{x \in \bb P(V)}  J_n(x) 
& \leq  \frac{c}{n} (b-a+  1) (b + a +1)   +  \frac{c \ee}{n} (b-a+ 1) (b + a +1)  
   +  c_{\ee}  \frac{\log^{2 -2 \ee'} n}{n^{3/2}}  (b - a + 1)   \notag\\
& \leq     \frac{ c }{n}  (b - a + 1) (b + a +1). 
\end{align*}
Therefore, recalling that $I_{n,1}$ is defined by \eqref{Decom-In-Lem2}, and using the first inequality in \eqref{Bound_y_nlogn_01}, 
we get that for any $n \geq 2$ and $- \sqrt{n} \log^{1-\ee} n \leq a < b \leq \sqrt{n} \log^{1-\ee} n$,  
\begin{align*}
 I_{n,1} 
 \leq  \int_{|y| \leq 2 \sqrt{n} \log^{1-\ee} n } \sup_{x \in \bb P(V)}  J_n(x)   dy   
\leq  c  \frac{ \log^{1 - \ee} n }{\sqrt{n}}  (b - a + 1) (b + a +1).  
\end{align*}
Combining this with \eqref{Decom-In-Lem2} and \eqref{Bound_y_nlogn_02} ends the proof of the lemma. 
\end{proof}

Using Lemma \ref{Lem_Integral_LLTnew_new2}, we can now establish an upper bound for the conditioned local probability 
$\mathbb{Q}_s^x  ( y - S_{n}\in [a,b],  \,  \tau_y > n )$. 
Note that this bound does not depend on the starting point $y \geq 0$ and the convergence rate is $O(\frac{1}{n})$. 

\begin{lemma} \label{lemma 2n}
Assume conditions \ref{Condi-density}, \ref{Condi-density-invariant} 
and $\kappa'(s) = 0$ for some $s \in  I_{\mu}$. 
Then, there exists a constant $c >0$ such that for any $x \in \bb P(V)$, $y \geq 0$, 
 $n \geq 1$ and $0 \leq  a < b \leq  \sqrt{n} \log n$,
\begin{align*}
\mathbb{Q}_s^x  \left( y - S_{n}\in [a,b],  \,  \tau_y > n \right)  
 \leq  \frac{c}{n} (b-a + 1) \left( b + a+ 1 \right).  
\end{align*}
\end{lemma}

\begin{proof} 
Let $k= [n/2]$ and $m=n-k.$ 
It is shown in \eqref{Qsx-bound-00a} that 
\begin{align*} 
 \mathbb{Q}_s^x \left(  y - S_n \in [a, b] , \tau_y >n\right)  
 \leq  \int_{\bb P(V) \times \bb R_+} h (x', y')   
\mathbb{Q}_s^x \left( X_k \in dx', y - S_k \in dy'\right) =: J_n(x,y),
\end{align*}
where $h$ is defined by \eqref{Def-h-xy-00a}. 
It is shown in \eqref{LLT_Upper_aa-xxxx} and \eqref{Pf_Inequa_dd02} that 
\begin{align*} 
\sup_{x \in \bb P(V)}  J_n(x,y)
& \leq \frac{1}{\sigma_s \sqrt{k}}  \int_{\bb P(V) \times \bb R}   h_{\ee} \left( x', y'\right) 
  \phi \left( \frac{y-y'}{\sigma_s \sqrt{k}}\right) dy' \pi_s(dx')   
 +  \frac{c\ee}{\sqrt{k}} \| h_{\ee} \|_{\pi_s \otimes \Leb}  
     +  \frac{c_{\ee}}{ k }  \| h_{\ee} \|_{\mathscr H}   \notag\\
&  =:  J_{n,1}(y) + J_{n,2} + J_{n,3},   
\end{align*}
where 
\begin{align}\label{Pf_Inequa_dd02bis}
h_{\ee}(x', y') 
& = \mathbb{Q}_s^{x'} \left( y' - S_{m} \in [a - \ee, b + \ee], \tau_{y' + \ee} > m \right). 
\end{align}
By \eqref{Bound_h_2ee_008} and \eqref{Bound_h_2ee_008-second}, we have 
\begin{align}\label{Bound-h-Jn1-2}
& J_{n,1}(y) \leq  \frac{c}{n} (b-a + 1) (b + a +1),  \quad  J_{n,2} \leq  \frac{c\ee}{n} (b-a + \ee) (b + a +1). 
\end{align}
For $J_{n,3}$, 
using \eqref{Pf_Inequa_dd02bis} and Lemma \ref{Lem_Integral_LLTnew_new2}, 
we get that for any $\ee \in [0, \frac{1}{2})$, 
there exists a constant $c_{\ee} >0$ such that for all $m \geq 1$ 
and $0 \leq a < b \leq \sqrt{n} \log^{1-\ee} n$,
\begin{align}\label{Bound-h-Jn3}
J_{n,3}
 =  \frac{c_{\ee}}{ k }
 \int_{\bb R} \sup_{x' \in \bb P(V)}  \mathbb{Q}_s^{x'} \left( y' - S_m \in [a - \ee, b + \ee], \, \tau_{y'} > m  \right) dy'  
 \leq  c_{\ee} \frac{\log^{1-\ee} n}{n^{3/2}}  (b - a +  \ee) (b+a +1). 
\end{align}
Putting together \eqref{Bound-h-Jn1-2} and \eqref{Bound-h-Jn3} concludes the proof of the lemma. 
\end{proof}

From Lemma \ref{lemma 2n}, 
we further investigate that the convergence rate for the conditioned local probability 
$\mathbb{Q}_s^x  ( y - S_{n}\in [a,b],  \,  \tau_y > n )$
can be improved to $O(\frac{1}{n^{3/2}})$, whereas the upper bound depends on the starting point $y$.

\begin{lemma} \label{lemma 3n}
Assume conditions \ref{Condi-density}, \ref{Condi-density-invariant} 
and $\kappa'(s) = 0$ for some $s \in  I_{\mu}$. 
Then, there exists a constant $c >0$ such that for any $x \in \bb P(V)$, $y \geq 0$, 
$n \geq 1$ and $0 \leq a < b \leq \sqrt{n} \log n$,
\begin{align*} 
 \mathbb{Q}_s^x \left( y - S_{n}\in [a,b],  \tau _{y}>n\right)   
& \leq  \frac{c}{n^{3/2}} (1+y) (b-a + 1) ( b+a+ 1). 
\end{align*}
\end{lemma}

\begin{proof}
As in \eqref{Qsx-bound-00a},  we use  the Markov property to get that for any
$m = [\frac{n}{2}]$ and $k = n - m$, 
\begin{align} \label{3n-lemma 002}
 \mathbb{Q}_s^x \left(  y - S_n \in [a, b] , \tau_y >n\right) 
 =  \int_{\bb P(V) \times \bb R_+} \mathbb{Q}_s^{x'} \left(  y' - S_m  \in [a, b] , \tau_{y'} >m\right)  
 \mathbb{Q}_s^x \left( X_k \in dx', y - S_k \in dy' , \tau_y >k\right). 
\end{align}
By Lemma \ref{lemma 2n}, there exists a constant $c >0$ such that for any $x' \in \bb P(V)$, $y' \geq 0$, 
$m \geq 1$ and $0 \leq a < b \leq  \sqrt{n} \log n$,
\begin{align}
\mathbb{Q}_s^{x'} \left(  y' - S_m  \in [a, b] , \tau_{y'} >m\right) 
 \leq  \frac{c}{n} (b-a + 1) ( b+a+ 1),  \label{inv003_bb}
\end{align}
where we used the fact that $m = [\frac{n}{2}]$ 
and the inequality in Lemma \ref{lemma 2n} still holds when $b$ is replaced by its constant multiple. 
By Theorem \ref{Th_Asymptotic_ExitTime}, there exist constants $c, c' >0$ such that for any $x \in \bb P(V)$ and $y\geq 0$, 
\begin{align}
\mathbb{Q}_s^x \left( \tau_{y} > k\right) 
  \leq  c \frac{1+ y}{\sqrt{k}} = c' \frac{1+ y}{\sqrt{n}}. \label{3N_007}
\end{align}
Combining \eqref{3n-lemma 002}, \eqref{inv003_bb} and  \eqref{3N_007} concludes the proof of the lemma. 
\end{proof}

The following assertion is a combination of Lemmas \ref{lemma 2n} and  \ref{lemma 3n}. 

\begin{lemma}\label{Lem_CondiLLT_cc}
Assume conditions \ref{Condi-density}, \ref{Condi-density-invariant} 
and $\kappa'(s) = 0$ for some $s \in  I_{\mu}$. 
Then, there exists a constant $c >0$ such that for any $x \in \bb P(V)$, $y \geq 0$, 
$n \geq 1$ and $0 \leq  a < b \leq \sqrt{n} \log n$,
\begin{align*}
 \mathbb{Q}_s^x  \left( y - S_{n}\in [a,b], \,  \tau _{y}>n\right)  
 \leq  c \frac{ (1+y) \wedge n^{1/2} }{n^{3/2}} (b-a + 1) ( b+a+ 1). 
\end{align*}
\end{lemma}

\subsection{Effective conditioned local limit theorems}\label{Sect-CLLT}

A central point to establish Theorems \ref{Thm_In_Pro} and \ref{Thm_In_Pro_min}
is a conditioned local limit theorem for random walks on the general linear group $\bb G$.
For sums of i.i.d.\  real-valued random variables, conditioned local limit theorems 
has been well-known in the literature based on the Wiener-Hopf factorization and the duality argument. 
For Markov chains, the obtention of such exact asymptotics turns out to be much more complicated 
and has been recently done in \cite{GLL20} in the particular case when the Markov chain has a finite state space. 
For random walks on groups, the problem is still open. 
In this section we shall establish such kind of results under the additional assumption that 
the matrix law $\mu$ admits a density with respect to the Haar measure on $\mathbb G$.  

To state the effective conditioned local limit theorem, we need to introduce the Rayleigh density function
$\phi^+(t) =  t e^{-t^2/2}$,  $t \in \bb R_+$.  
Below we assume that the function $h$ is supported on $\bb P(V) \times [0, \infty)$,
$h_{\ee}$ on $\bb P(V) \times [-\ee, \infty)$
and $h_{-\ee}$ on $\bb P(V) \times [\ee, \infty)$. 

\begin{theorem} \label{t-A 001}
Assume conditions \ref{Condi-density}, \ref{Condi-density-invariant} 
and $\kappa'(s) = 0$ for some $s \in  I_{\mu}$. 
Let $(\alpha_n)_{n \geq 1}$ be any sequence of positive numbers satisfying $\lim_{n \to \infty} \alpha_n = 0$. 
Then, there exist constants $c, c_{\ee} >0$ 
such that for any $\ee \in (0,\frac{1}{8})$, $x \in \bb P(V)$, $y \in [0, \alpha_n \sqrt{n}]$, $n \geq 1$, 
$h \in \mathscr H$ and $h_{\ee} \in \mathscr H$ satisfying $h \leq_{\ee} h_{\ee}$, 
\begin{align} \label{eqt-A 002_Upper}
 \bb E_{\bb Q_s^x} \big[ h(X_n, y - S_n); \,  \tau_{y} > n  \big]   
& \leq  \frac{2V_s(x, y)}{ n \sigma_s^2 \sqrt{2\pi } }  
 \int_{\bb P(V) \times \bb R_+}  h_{\ee} (x,t)  \phi^+\left(\frac{t}{\sigma_s \sqrt{n}} \right) \pi_s(dx)  dt     \notag\\
& \quad  +  \left( c \ee^{1/4} + c_{\ee} \alpha_n + c_{\ee} n^{-\ee} \right)  \frac{ 1 + y }{n}    \| h_{\ee} \|_{\pi_s \otimes \Leb}
    +   \frac{c_{\ee} (1 + y)}{ n^{3/2} }  \| h_{\ee} \|_{\mathscr H}  
\end{align}
and any $h, h_{-\ee}, h_{\ee} \in \mathscr H$ satisfying $h_{-\ee} \leq_{\ee}  h \leq_{\ee} h_{\ee}$, 
\begin{align} \label{eqt-A 002}
 \bb E_{\bb Q_s^x} \big[ h(X_n, y - S_n); \,  \tau_{y} > n  \big]     
&  \geq  \frac{2V_s(x, y)}{ n \sigma_s^2 \sqrt{2\pi } }  
 \int_{\bb P(V) \times \bb R_+}  h_{-\ee} (x,t)  \phi^+\left(\frac{t}{\sigma_s \sqrt{n}} \right) \pi_s(dx) dt      \notag\\
& \quad  -  \left( c \ee^{1/12} + c_{\ee} \alpha_n +  c_{\ee} n^{-\ee} \right)  
\frac{  1 + y }{n}    \| h_{\ee} \|_{\pi_s \otimes \Leb}
    -   \frac{c_{\ee} (1 + y)}{ n^{3/2} }  \| h_{\ee} \|_{\mathscr H}   \notag\\
& \quad    -    \frac{c_{\ee} (1 + y) }{ n }  \sup_{x \in \bb P(V)} \sup_{y \in \bb R  } h_{\ee} (x, y). 
\end{align}
\end{theorem}

Let $\phi(x) = \frac{1}{\sqrt{2 \pi} } e^{- \frac{x^2}{2}}$, $x \in \bb R$ be the standard normal density function. 
Let 
\begin{align*}
\phi_{v} (x) = \frac{1}{\sqrt{2 \pi v} } e^{- \frac{x^2}{2 v}},  \quad 
\phi^+_{v}(x)=\frac{x}{v} e^{-\frac{x^2}{2 v}} \mathds 1_{\mathbb R_+} (x), 
 \quad x \in \bb R,  
\end{align*}
be the normal density of variance $v > 0$ and the Rayleigh density with scale parameter $\sqrt{v}$, respectively. 
Clearly we have $\phi = \phi_1$ and $\phi^+ = \phi^+_{1}$. 
The following lemma (cf.\  \cite[Lemma 3.3]{GX21}), which will be used in the proof of Theorem \ref{t-A 001}, 
shows that when $v$ is small the convolution $\phi_{v} * \phi^+_{1-v}$ 
behaves like the Rayleigh density. 
\begin{lemma}[\cite{GX21}] \label{t-Aux lemma}
For any $v \in (0,1/2]$ and $t\in \bb R$, it holds that
\begin{align*} 
 - |t| e^{- \frac{t^2}{2}} \mathds 1_{\{t < 0\}} 
\leq  \phi_{v} * \phi^+_{1- v}(t) - \sqrt{1-v} \phi^+(t)
\leq   \sqrt{v}  e^{ -\frac{t^2}{2v} } +  |t| e^{- \frac{t^2}{2}} \mathds 1_{\{t < 0\}}. 
\end{align*}
\end{lemma}

To establish Theorem \ref{t-A 001},  we also need the following inequality of Haeusler \cite[Lemma 1]{Hae84}, 
which is a generalisation of Fuk's inequality for martingales. 

\begin{lemma}[\cite{Hae84}]\label{Lem_Fuk}
Let $\xi_1, \ldots, \xi_n$ be a martingale difference sequence with respect to the non-decreasing 
$\sigma$-fields $\mathscr F_0, \mathscr F_1, \ldots, \mathscr F_n$.
Then, for any $u, v, w > 0$ and $n \geq 1$,
\begin{align*}
\bb P \left(  \max_{1 \leq k \leq n} \left| \sum_{i=1}^k \xi_i \right| \geq u \right)
\leq  \sum_{i=1}^n \bb P \left( |\xi_i|  > v \right)
  + 2 \bb P \left(  \sum_{i=1}^n \bb E \left(  \xi_i^2  \Big|  \mathscr F_{i-1} \right) > w \right)   
  + 2 e^{ \frac{u}{v} ( 1 - \log \frac{uv}{w} ) }.  
\end{align*}
\end{lemma}

We first give a proof of the upper bound \eqref{eqt-A 002_Upper} of Theorem \ref{t-A 001}. 

\begin{proof}[Proof of \eqref{eqt-A 002_Upper}] 
It suffices to prove \eqref{eqt-A 002_Upper} for large enough $n>n_0(\ee)$, where $n_0(\ee)$ depends on $\ee$,
since otherwise the bound becomes trivial.
For any $\ee\in (0,\frac{1}{8})$, let $\delta = \sqrt{\ee}$, 
$m=\left[ \delta n \right]$ and $k = n-m.$ 
Then, for $n$ sufficiently large, we have 
$\frac{1}{2}\delta \leq \frac{\left[ \delta n\right]}{n}  \leq \frac{m}{k} \leq  \frac{\delta}{1-\delta}$.  
Using the Markov property of the pair $(X_n, S_n)$, 
we get that for any $n \geq 1$, $x \in \bb P(V)$ and $y \in \mathbb R_+$, 
\begin{align}\label{JJJJJ-1111-000}
 I_n(x, y)
:&=  \bb E_{\bb Q_s^x} \left[ h(X_n, y - S_n); \,  \tau_y > n  \right] \notag \\ 
&= \int_{\bb P(V) \times \bb R_+}  \bb E_{\bb Q_s^{x'}} \left[ h(X_m, y' - S_m); \,  \tau_{y'}>n  \right]  
     \bb Q_s^x \left(X_k \in dx', \, y - S_{k} \in dy',  \,  \tau_{y}>k\right)  \notag\\
& \leq  \int_{\bb P(V) \times \bb R_+} \bb E_{\bb Q_s^{x'}} \left[ h(X_m, y' - S_m)  \right] 
   \bb Q_s^x \left(X_k \in dx', \, y - S_{k} \in dy',  \,  \tau_{y}>k\right). 
\end{align}
By the local limit theorem (\eqref{LLT_Upper_aa} of Theorem \ref{Theorem Loc bounds}), 
for any $\ee \in (0, \frac{1}{8})$, there exist constants $c, c_{\ee} >0$ such that
for any $m \geq 1$, $x' \in \bb P(V)$ and $y' \in \bb R_+$, 
\begin{align} \label{JJJJJ-1111-001}
\bb E_{\bb Q_s^{x'}} \left[ h(X_m, y' - S_m)  \right] 
 \leq   H_m(y')
   + \frac{c\ee}{\sqrt{m}} \| h_{\ee} \|_{\pi_s \otimes \Leb}  
     +  \frac{c_{\ee}}{ m}  \| h_{\ee} \|_{\mathscr H}, 
\end{align}
where $h \leq_{\ee} h_{\ee}$ and 
\begin{align}\label{JJJ005}
H_m(y') =  \int_{\bb P(V) \times \bb R}  h_{\ee} \left( x, y\right) 
   \frac{1}{\sigma_s \sqrt{m}} \phi \left( \frac{y'-y}{\sigma_s \sqrt{m}}\right) \pi_s(dx) dy. 
\end{align}
Substituting \eqref{JJJJJ-1111-001} into \eqref{JJJJJ-1111-000},
 and using Theorem \ref{Th_Asymptotic_ExitTime}, we get that for any $x \in \bb P(V)$ and $y\in \bb R_+$,
\begin{align} \label{JJJ004}
   I_n(x, y)  
& \leq  J_n(x, y)   
  + \left( \frac{c\ee}{\sqrt{m}} \| h_{\ee} \|_{\pi_s \otimes \Leb}  
     +  \frac{c_{\ee}}{ m }  \| h_{\ee} \|_{\mathscr H} \right)
   \bb Q_s^x \left(\tau_{y} > k \right)  \notag\\
&  \leq  J_n(x, y)   
  + \left( \frac{c\ee}{\sqrt{mk}} \| h_{\ee} \|_{\pi_s \otimes \Leb}  
     +  \frac{c_{\ee}}{ m\sqrt{k} }  \| h_{\ee} \|_{\mathscr H} \right) (1 + y), 
\end{align}
where 
\begin{align}\label{Def-Jnxy-ab}
J_n(x, y)  =  
  \int_{\mathbb{R}_{+}}  H_m(y')  \bb Q_s^x \left(y - S_{k} \in dy', \, \tau_{y} > k \right). 
\end{align}
By a change of variable, for any $x \in \bb P(V)$ and $y\in \bb R_+$,
\begin{align*}
J_n(x, y)  
=  \int_{\mathbb{R}_{+}}  F_m(t)  
  \bb Q_s^x \left(\frac{y - S_{k}}{ \sigma_s \sqrt{k} } \in dt, \, \tau_{y} > k \right), 
\end{align*}
where $F_m(t) = H_m(\sigma_s \sqrt{k} t)$, $t \in \bb R$. 
Since the function $t \mapsto F_m(t)$ is differentiable on $\bb R$ and vanishes as $t \to \pm \infty$, 
using integration by parts, it follows that for any $x \in \bb P(V)$ and $y \in \bb R_+$,
\begin{align*} 
J_n(x, y)
& =  \int_{\mathbb{R}_{+}}  F'_m(t) \ \bb Q_s^x \left(\frac{y - S_{k}}{\sigma_s \sqrt{k}} > t, \tau_{y}>k\right) dt.
 \end{align*}
 It is easy to verify that, with $\beta_k : = \alpha_n/(1-\sqrt{\ee})$, we have
 $[0, \alpha_n \sqrt{n}] \subseteq [0, \beta_k \sqrt{k}].$ 
Applying the conditioned integral limit theorem 
(Theorem \ref{Theor-IntegrLimTh}), 
and using the fact that $\frac{1}{\sqrt{k}} (\beta_k + k^{-\ee}) \leq \frac{c}{ \sqrt{n} } (\alpha_n + n^{-\ee})$, 
we get that there exists a constant $c_{\ee} >0$ such that for any $x \in \bb P(V)$ and $y \in [0, \alpha_n \sqrt{n}]$,  
\begin{align} \label{ApplCondLT-002}
J_n(x, y)
 \leq  \frac{2V_s(x, y)}{\sigma_s \sqrt{2\pi k} }  \int_{\mathbb{R}_{+}}  F'_m(t) ( 1 - \Phi^+(t) ) dt  
 +  c_{\ee}  \frac{1 + y}{\sqrt{n}} (\alpha_n + n^{-\ee}) \int_{\mathbb{R}_{+}}  | F'_m(t) | dt.
\end{align}
Since $F_m(t) = H_m(\sigma_s \sqrt{k} t)$, by \eqref{JJJ005} and a change of variable, it holds that 
\begin{align}\label{Def-Fmt}
F_m(t) = \int_{\bb P(V) \times \bb R}   h_{\ee} \left( x, \sigma_s \sqrt{k} y\right) 
    \phi \Big( \frac{t-y}{ \sqrt{m/k}} \Big) \pi_s(dx)   \frac{dy}{ \sqrt{m/k} }. 
\end{align}
Again by a change of variable and Fubini's theorem, we get
\begin{align} \label{JJJ-20001}
 \int_{\mathbb{R}_{+}}  | F'_m(t) | dt  
& \leq  \int_{\mathbb{R}_{+}}  
 \left[ \int_{\bb P(V) \times \bb R}
 h_{\ee} \left(x, \sigma_s \sqrt{k} y \right)    
 \left| \phi' \Big( \frac{t - y}{\sqrt{m/k}} \Big) \right| \pi_s(dx)  \frac{dy}{\sqrt{m/k}}  \right] \frac{dt}{\sqrt{m/k}}  \notag\\ 
& =   \int_{\mathbb{R}_{+}}  
 \left[ \int_{\bb P(V) \times \bb R}  h_{\ee} (x, \sigma_s \sqrt{m} y')  \left| \phi'\left( t'-y' \right) \right|  \pi_s(dx) dy'  \right]  dt'  \notag\\ 
& \leq c  \int_{\bb P(V) \times \bb R}  h_{\ee}(x, \sigma_s \sqrt{m} y')  \pi_s(dx) dy'   \notag \\ 
&  =  \frac{c}{\sqrt{m}} \int_{\bb P(V) \times \bb R}  h_{\ee}(x, y)  \pi_s(dx) dy 
 = \frac{c}{\sqrt{m}} \| h_{\ee} \|_{\pi_s \otimes \Leb}. 
\end{align}
For the first term in the right hand side of \eqref{ApplCondLT-002},  
by the definition of $F_m$ (cf.\ \eqref{Def-Fmt}), using integration by parts, a change of variable and Fubini's theorem,
we deduce that, with $\delta_n = \frac{m}{n}$, 
\begin{align}\label{Bound-Fmt-phit00}
& \int_{\mathbb{R}_{+}}  F'_m(t) (1- \Phi^+ (t)) dt
 =  \int_{\mathbb{R}_{+}}  F_m(t) \phi^+ (t) dt    \notag\\
&= \int_{\bb R _{+}}  \left[   \int_{\bb P(V) \times \bb R}  h_{\ee} \left( x, \sigma_s \sqrt{k} y\right) 
    \phi \left( \frac{t-y}{ \sqrt{m/k}}\right) \pi_s(dx)   \frac{dy}{ \sqrt{m/k} }  \right]  \phi^+(t) dt  \notag\\
&= \int_{\bb R _{+}}  \left[  \int_{\bb P(V) \times \bb R}  h_{\ee} \left( x, \sigma_s \sqrt{n} y' \right) 
    \frac{1}{ \sqrt{m/k} } \phi \left( \frac{t'-y'}{ \sqrt{m/n}}\right) \pi_s(dx)   \frac{dy'}{ \sqrt{k/n} }  \right]   
       \phi^+ \left( \frac{t'}{\sqrt{k/n}} \right) \frac{dt'}{\sqrt{k/n}}  \notag\\
& =  \int_{\bb P(V) \times \bb R}  h_{\ee} \left( x, \sigma_s \sqrt{n} y' \right)   
   \phi_{\delta_n}*\phi^+_{1-\delta_n}(y')  \pi_s(dx)   dy'   \notag\\
& = \frac{1}{ \sigma_s \sqrt{n}}
 \int_{\bb P(V) \times \bb R}  h_{\ee} (x,t) \phi_{\delta_n}*\phi^+_{1-\delta_n}\left(\frac{t}{ \sigma_s \sqrt{n}} \right) \pi_s(dx) dt. 
\end{align}
Using Lemma \ref{t-Aux lemma} with $v = \delta_n$ 
and recalling that $\delta_n = \frac{m}{n}$ and $1-\delta_n=\frac{k}{n}$, we have, for any $t \in \bb R$, 
\begin{align*}
\phi_{\delta_n}*\phi^+_{1-\delta_n}\left(\frac{t}{ \sigma_s \sqrt{n}} \right)
\leq  \sqrt{\frac{k}{n}}  \phi^+ \left(\frac{t}{ \sigma_s \sqrt{n}} \right) +  \sqrt{\frac{m}{n}}  
   + \frac{|t|}{ \sigma_s \sqrt{n} }  \mathds 1_{\{t<0\}}. 
\end{align*}
Implementing this bound into \eqref{Bound-Fmt-phit00} and using the fact that $\phi^+(u) = 0$ for $u \leq 0$, we get  
\begin{align}\label{Bound-Fmt-phit}
 \int_{\mathbb{R}_{+}}  F'_m(t) (1- \Phi^+ (t)) dt  
& \leq \frac{\sqrt{k}}{\sigma_s  n}  \int_{\bb P(V) \times \bb R_+}  
   h_{\ee} (x,t)  \phi^+\left(\frac{t}{\sigma_s \sqrt{n}} \right) \pi_s(dx) dt  \notag\\ 
 & \qquad  +  \frac{ \sqrt{m}}{\sigma_s  n}  \| h_{\ee} \|_{\pi_s \otimes \Leb}
   + \frac{c}{n} \int_{\bb P(V) \times \bb R} h_{\ee} (x,t) |t| \mathds 1_{\{t<0\}}  \pi_s(dx) dt   \notag\\
& \leq   \frac{\sqrt{k}}{\sigma_s  n}  \int_{\bb P(V) \times \bb R_+}  
h_{\ee} (x,t)  \phi^+\left(\frac{t}{\sigma_s \sqrt{n}} \right) \pi_s(dx) dt
    +  \frac{c \ee^{1/4}}{n}  \| h_{\ee} \|_{\pi_s \otimes \Leb},
\end{align}
where in the last inequality we used the fact that $m= [ \ee^{1/2} n]$ 
and the function $h_{\ee}$ is supported on $\bb P(V) \times [- \ee, \infty)$. 
Combining \eqref{ApplCondLT-002},  \eqref{JJJ-20001} and \eqref{Bound-Fmt-phit}, 
and using the fact that $V_s(x, y) \leq c(1 + y)$, we derive that
\begin{align}\label{Upper-Jnxy}
J_n(x, y)
& \leq  \frac{2V_s(x, y)}{ n \sigma_s^2 \sqrt{2\pi } }  
 \int_{\bb P(V) \times \bb R_+}  h_{\ee} (x,t)  \phi^+\left(\frac{t}{\sigma_s \sqrt{n}} \right)  \pi_s(dx) dt    \notag\\
& \quad  +  \frac{c \ee^{1/4} (1 + y)}{n} \| h_{\ee} \|_{\pi_s \otimes \Leb}
    +    \frac{c_{\ee} (1 + y)}{n} (\alpha_n + n^{-\ee}) \| h_{\ee} \|_{\pi_s \otimes \Leb}.
\end{align}
Substituting this into \eqref{JJJ004} ends the proof of the upper bound \eqref{eqt-A 002_Upper}.  
\end{proof}

We next establish the lower bound \eqref{eqt-A 002} of Theorem \ref{t-A 001}. 
\begin{proof}[Proof of \eqref{eqt-A 002}]
Let us keep the notation used in the proof of the upper bound \eqref{eqt-A 002_Upper}. 
By the Markov property of the pair $(X_n, S_n)$, we get
\begin{align} \label{staring point for the lower-bound-001} 
 I_n(x, y)
:=  \bb E_{\bb Q_s^x} \left[ h(X_n, y - S_n); \,  \tau_y > n  \right]  
 = I_{n,1}(x, y) - I_{n,2}(x, y),
\end{align}
where 
\begin{align}
I_{n,1}(x, y) 
& =  \int_{\bb P(V) \times \bb{R}_+}  \bb E_{\bb Q_s^{x'}} \left[ h(X_m, y' - S_m)  \right]  
  \bb Q_s^x \left(X_k \in dx', \, y - S_{k} \in dy',  \,  \tau_{y}>k\right),    \label{In1xy-lowerbound}\\
I_{n,2}(x, y)
& =  \int_{\bb P(V) \times \bb{R}_+} \bb E_{\bb Q_s^{x'}} \left[ h(X_m, y' - S_m);  \, \tau_{y'} \leq m   \right]   
   \bb Q_s^x \left(X_k \in dx', \, y - S_{k} \in dy',  \,  \tau_{y}>k\right).  \label{In2xy-lowerbound}
\end{align}

\textit{Lower bound of $I_{n,1}(x, y)$.}
By the local limit theorem (\eqref{LLT_Lower_aa} of Theorem \ref{Theorem Loc bounds}),
there exist constants $c, c_{\ee} >0$ such that
for any $m \geq 1$, $x' \in \bb P(V)$ and $y' \in \bb R_+$, 
\begin{align}\label{Pf_SmallStarting_Firstthm}
\bb E_{\bb Q_s^{x'}} \left[ h(X_m, y' - S_m)  \right] 
 \geq  H_m^-(y')  -   \frac{c\ee}{\sqrt{m}} \| h \|_{\pi_s \otimes \Leb}  
     -  \frac{c_{\ee}}{ m }  \| h \|_{\mathscr H}, 
\end{align}
where
\begin{align*}
H_m^-(y') =  \int_{\bb P(V) \times \bb R}  h_{-\ee} \left( x, y\right) 
   \frac{1}{\sigma_s \sqrt{m}} \phi \left( \frac{y'-y}{\sigma_s \sqrt{m}}\right) \pi_s(dx) dy. 
\end{align*}
Following the proof of the upper bound of $J_n(x, y)$ (cf.\ \eqref{Def-Jnxy-ab} and \eqref{Upper-Jnxy})
and using the lower bound in Lemma \ref{t-Aux lemma} instead of the upper one, 
one has, uniformly in $x \in \bb P(V)$ and $y \in [0, \alpha_n \sqrt{n}]$, 
\begin{align}\label{Lowerbound-Jn-xy}
J_n^-(x, y) : & =  
  \int_{\mathbb{R}_{+}}  H_m^-(y')  \bb Q_s^x \left(y - S_{k} \in dy', \, \tau_{y} > k \right)  \notag\\
 &  \geq  \frac{2V_s(x, y)}{ n \sigma_s^2 \sqrt{2\pi } }  
 \int_{\bb P(V) \times \bb R_+}  h_{-\ee} (x,t)  \phi^+\left(\frac{t}{\sigma_s \sqrt{n}} \right)  \pi_s(dx) dt    \notag\\
& \quad  -  \frac{c \ee^{1/4} (1 + y)}{n} \| h_{-\ee} \|_{\pi_s \otimes \Leb}
        -    \frac{c_{\ee} (1 + y)}{n} (\alpha_n + n^{-\ee}) \| h_{-\ee} \|_{\pi_s \otimes \Leb}.
\end{align}
Substituting \eqref{Pf_SmallStarting_Firstthm} into \eqref{In1xy-lowerbound}, using \eqref{Lowerbound-Jn-xy},
the bound $\bb Q_s^x ( \tau_y > k ) \leq \frac{c(1 + y)}{\sqrt{n}}$ 
and the fact that $\| h_{-\ee} \|_{\pi_s \otimes \Leb} \leq  \| h \|_{\pi_s \otimes \Leb}$ and
$ \|  h_{-\ee} \|_{\scr H} \leq  \|  h \|_{\scr H}$, we get 
\begin{align}\label{Lower_In_lll}
 I_{n,1}(x, y)  
&  \geq  \frac{2V_s(x, y)}{ n \sigma_s^2 \sqrt{2\pi } }  
 \int_{\bb P(V) \times \bb R_+}  h_{-\ee} (x,t)  \phi^+\left(\frac{t}{\sigma_s \sqrt{n}} \right) \pi_s(dx) dt     \notag\\
& \quad  -  \left( c \ee^{1/4} + c_{\ee} \alpha_n +  c_{\ee} n^{-\ee} \right)  \frac{ 1 + y }{n}    \| h \|_{\pi_s \otimes \Leb}
    -   \frac{c_{\ee} (1 + y)}{ n^{3/2} }  \| h \|_{\mathscr H}. 
\end{align}

\textit{Upper bound of $I_{n,2}(x, y)$.}
According to whether the value of $y'$ in the integral is less or greater than $\ee^{1/6} \sqrt{n}$,
we decompose $I_{n,2}(x, y)$ into two terms: 
\begin{align} \label{KKK-111-001}
I_{n,2}(x, y) & = J_{n,1}(x,y) + J_{n,2}(x,y),
\end{align}
where
\begin{align*} 
J_{n,1}(x,y) & = \int_{0}^{ \ee^{1/6}\sqrt{n}} 
\int_{\bb P(V)} \bb E_{\bb Q_s^{x'}} \left[ h(X_m, y' - S_m);  \, \tau_{y'} \leq m \right]   
  \bb Q_s^x \left(X_k \in dx', \, y - S_{k} \in dy',  \,  \tau_{y}>k\right),  \\
J_{n,2}(x,y) &=\int_{\ee^{1/6}\sqrt{n}}^\infty 
\int_{\bb P(V)} \bb E_{\bb Q_s^{x'}} \left[ h(X_m, y' - S_m);  \, \tau_{y'} \leq m  \right]   
  \bb Q_s^x \left(X_k \in dx', \, y - S_{k} \in dy',  \,  \tau_{y}>k\right).
\end{align*}
For $J_{n,1}(x,y)$, 
since $\bb E_{\bb Q_s^{x'}} [ h(X_m, y' - S_m);  \, \tau_{y'} \leq m ] \leq \bb E_{\bb Q_s^{x'}} \left[ h(X_m, y' - S_m) \right]$, 
using Theorem \ref{Theorem Loc bounds} 
and proceeding in the same way as in \eqref{JJJJJ-1111-001} and \eqref{JJJ004}, 
we get 
\begin{align}\label{CaraOrder12BoundK1}
J_{n,1}(x,y) 
\leq   K_{n}(x,y)    +   \frac{c\ee^{3/4} (1 + y)}{ n } \| h_{\ee} \|_{\pi_s \otimes \Leb}  
     +  \frac{c_{\ee} (1 + y)}{ n^{3/2} }  \| h_{\ee} \|_{\mathscr H},   
\end{align}
where 
\begin{align*}
K_{n}(x,y) =  \int_{0}^{\ee^{1/6}\sqrt{ n}}  H_m(y')  \bb Q_s^x \left( y - S_{k} \in dy', \, \tau_{y} > k \right)
\end{align*}
and $H_m$ is defined by \eqref{JJJ005}. 
By integration by parts,  it follows that
\begin{align}\label{CaraOrder12BoundK1aa}
K_{n}(x,y)
& =  \int_{0}^{\infty}  H_m(y') 
   \bb Q_s^x \left( y - S_k \in dy',  \,  y - S_k \leq \ee^{1/6}\sqrt{ n},  \,  \tau_y > k  \right)    \notag\\
& =  \int_{0}^{\infty}    H_m'(y')   \bb Q_s^x  \left( y - S_k \in [0, \ee^{1/6}\sqrt{ n}],  \,  y - S_k > y',  \,  \tau_y > k \right)  dy'   \notag\\
& =  \int_{0}^{ \ee^{1/6}\sqrt{ n }  }  
H_m' (t)  \bb Q_s^x  \left( \frac{y - S_k}{\sigma \sqrt{k}} 
    \in \left[ \frac{t}{\sigma_s \sqrt{k}}, \frac{\ee^{1/6} \sqrt{n}}{\sigma_s \sqrt{k}}  \right], 
  \tau_y > k \right)  dt.  
\end{align}
Similarly to the proof of \eqref{ApplCondLT-002}, applying the conditioned integral limit theorem (Theorem \ref{Theor-IntegrLimTh}), 
we derive that uniformly in $t \in \bb R_+$, $x \in \bb P(V)$ and $y \in [0, \alpha_n \sqrt{n}]$, 
\begin{align*}
 \Bigg|  \bb Q_s^x  \left( \frac{y - S_k}{\sigma_s \sqrt{k}} 
    \in \left[ \frac{t}{\sigma_s \sqrt{k}}, \frac{\ee^{1/6} \sqrt{n}}{\sigma_s \sqrt{k}}  \right], 
  \tau_y > k \right)
  -  \frac{2V_s(x,y)}{ \sigma_s \sqrt{2\pi k}}  
    \left[ \Phi^+ \left( \frac{\ee^{1/6} \sqrt{n}}{\sigma_s \sqrt{k}}  \right)   - \Phi^+ \left( \frac{t}{\sigma_s \sqrt{k}} \right)  \right]   \Bigg|  
    \leq  c_{\ee} \left(  \alpha_n  + n^{-\ee} \right) \frac{1 + y}{\sqrt{n}}.
\end{align*}
By the definition of $H_m$ (cf.\ \eqref{JJJ005}) and the fact that $\phi'$ is bounded, it holds that
\begin{align*}
\int_{0}^{ \ee^{1/6}\sqrt{ n }  }   |H_m' (t)| dt 
\leq  \ee^{1/6}\sqrt{ n }  \sup_{t \in \bb R}  |H_m' (t)|  
\leq \frac{c \ee^{1/6} \sqrt{n} }{m}  \| h_{\ee} \|_{\pi_s \otimes \Leb}
\leq \frac{c_{\ee} }{\sqrt{n}}  \| h_{\ee} \|_{\pi_s \otimes \Leb}. 
\end{align*}
Therefore, 
\begin{align}\label{Bound-Ju8}
K_{n}(x,y)
 \leq   \frac{2V_s(x,y)}{ \sigma_s \sqrt{2\pi k}}      \int_{0}^{ \ee^{1/6}\sqrt{ n }  }  H_m' (t)
\left[ \Phi^+ \left( \frac{\ee^{1/6} \sqrt{n}}{\sigma_s \sqrt{k}}  \right)   - \Phi^+ \left( \frac{t}{\sigma_s \sqrt{k}} \right)  \right] dt
  + c_{\ee} \left(  \alpha_n  + n^{-\ee} \right) \frac{1 + y}{n} \| h_{\ee} \|_{\pi_s \otimes \Leb}. 
\end{align}
Using integration by parts and the fact that $H_m(0) \geq 0$
and $\sup_{t \in \bb R} H_m (t) \leq  \frac{1}{ \sigma_s \sqrt{m} } \| h_{\ee} \|_{\pi_s \otimes \Leb}$ (cf.\ \eqref{JJJ005}),  we get
\begin{align}\label{Bound-Ju9}
& \int_{0}^{ \ee^{1/6}\sqrt{ n }  }  H_m' (t)
\left[ \Phi^+ \left( \frac{\ee^{1/6} \sqrt{n}}{\sigma_s \sqrt{k}}  \right)   - \Phi^+ \left( \frac{t}{\sigma_s \sqrt{k}} \right)  \right] dt  \notag\\
& = \frac{1}{\sigma_s \sqrt{k}} \int_{0}^{ \ee^{1/6}\sqrt{ n }  }  H_m \left(t\right) \phi^+ \left( \frac{t}{\sigma_s \sqrt{k}} \right)  dt
   -  H_m(0) \Phi^+ \left( \frac{\ee^{1/6} \sqrt{n}}{\sigma_s \sqrt{k}}  \right)  
     \notag\\
& \leq  \frac{1}{ \sigma_s^2 \sqrt{m k} } \| h_{\ee} \|_{\pi_s \otimes \Leb}
  \int_{0}^{ \ee^{1/6}\sqrt{ n }  }   \phi^+ \left( \frac{t}{\sigma_s \sqrt{k}} \right)  dt =: A_n. 
\end{align} 
By a change of variable and the fact that $\phi^+ (u) \leq u$ for $u \geq 0$, it follows that 
\begin{align*} 
A_n  =  \frac{1}{ \sigma_s \sqrt{m} } \| h_{\ee} \|_{\pi_s \otimes \Leb}
  \int_{0}^{ \frac{\ee^{1/6}}{\sigma_s}\sqrt{ n/k }  }  \phi^+ \left( u \right)  du   
 \leq  \frac{c \ee^{1/3}}{ \sqrt{m} } \| h_{\ee} \|_{\pi_s \otimes \Leb}
  \leq  \frac{c \ee^{1/12}}{ \sqrt{n} } \| h_{\ee} \|_{\pi_s \otimes \Leb}.  
\end{align*}
This, together with \eqref{Bound-Ju8} and \eqref{Bound-Ju9}, implies that
\begin{align*}
K_{n}(x,y) 
& \leq  c \ee^{1/12} \frac{ V_s(x, y)}{ n }   \| h_{\ee} \|_{\pi_s \otimes \Leb}
  +   c_{\ee} \left(  \alpha_n  + n^{-\ee} \right) \frac{ 1 + y}{n}  \| h_{\ee} \|_{\pi_s \otimes \Leb}  \notag\\
& \leq  \left( c \ee^{1/12} + c_{\ee} \alpha_n +  c_{\ee}  n^{-\ee}  \right)  \frac{ 1 + y}{n}  \| h_{\ee} \|_{\pi_s \otimes \Leb}. 
\end{align*}
Combining this with \eqref{CaraOrder12BoundK1}, we obtain 
\begin{align}\label{K1-final bound}
J_{n,1}(x,y)  
\leq   \left( c \ee^{1/12} + c_{\ee} \alpha_n +  c_{\ee}  n^{-\ee}  \right)  \frac{ 1 + y}{n}  \| h_{\ee} \|_{\pi_s \otimes \Leb}  
  +  \frac{c_{\ee} (1 + y)}{n^{3/2}}  \| h_{\ee} \|_{\mathscr H}.  
\end{align}

Now we proceed to give an upper bound for $J_{n,2}(x,y)$, which can be rewritten as
\begin{align}  
J_{n,2}(x,y) = \int_{\bb P(V) \times \bb R} 
  L(x', y')  \bb Q_s^x \left(X_k \in dx', \, y - S_{k} \in dy',  \,  \tau_{y}>k\right),  \label{K2-b01c-001}
\end{align}
where, for $x' \in \bb P(V)$ and $y' \in \bb R$, 
\begin{align} \label{Bytheduality-001}
L(x', y') : =  \mathds 1_{\{ y' > \ee^{1/6}\sqrt{n}  \}}  \bb E_{\bb Q_s^{x'}} \left[ h(X_m, y' - S_m); \, \tau_{y'} \leq m  \right]. 
\end{align}
Since $h(x', \cdot)$ is integrable on $\bb R$, 
by Fubini's theorem and a change of variable, one can check that the function $y' \mapsto L(x', y')$ is integrable on $\bb R$, 
for any $x' \in \bb P(V)$.  
Denote for any $x' \in \bb P(V)$ and $y' \in \bb R$, 
\begin{align} \label{DefM88}
L_{\ee}(x', y') :=  
\mathds 1_{\{ y' + \ee > \ee^{1/6}\sqrt{n}  \}}
\bb E_{\bb Q_s^{x'}} \left[ h_{\ee}(X_m, y' - S_m); \, \tau_{y' - \ee} \leq m  \right].   
\end{align}
Then we have $L \leq_{\ee} L_{\ee}$.  
Using the upper bound \eqref{eqt-A 002_Upper} of  Theorem \ref{t-A 001} 
and the fact that $\left\|  L_{\ee} \right\|_{\pi_s \otimes \Leb}  \leq  \left\|  h_{\ee} \right\|_{\pi_s \otimes \Leb}$,  we obtain that 
uniformly in $x \in \bb P(V)$ and $y \in [0, \alpha_n \sqrt{n}]$, 
\begin{align}\label{eqt-A 001_Lower}
J_{n,2}(x,y)
& \leq  \frac{2V_s(x, y)}{ k \sigma_s^2 \sqrt{2\pi } }  
 \int_{\bb P(V) \times \bb R_+}  L_{\ee} (x', y')  \phi^+\left(\frac{y'}{\sigma_s \sqrt{k}} \right)  \pi_s(dx')  dy'     \notag\\
& \quad  +  \left( c \ee^{1/4} + c_{\ee} \alpha_n + c_{\ee}  n^{-\ee}   \right) \frac{  (1 + y)}{n}    \| h_{\ee} \|_{\pi_s \otimes \Leb}
    +   \frac{c_{\ee} (1 + y)}{ n^{3/2} }  \| L_{\ee} \|_{\mathscr H}. 
\end{align}
For the first term, in view of \eqref{DefM88}, we use the duality formula (Lemma \ref{Lem_Duality}) 
to derive that 
\begin{align}\label{CLLT-J1-6}
&  \int_{\bb P(V) \times \bb R_+}  L_{\ee} (x, y)  \phi^+\left(\frac{y}{\sigma_s \sqrt{k}} \right) \pi_s(dx) dy     \notag\\
& =   \int_{\bb P(V) \times \bb{R}_+}  
\bb E_{\bb Q_s^{x}} \left[ h_{\ee}(X_m, y - S_m); \, \tau_{y - \ee} \leq m  \right]   
   \mathds 1_{\{ y + \ee > \ee^{1/6}\sqrt{n}  \}}
 \phi^+ \left( \frac{y}{ \sigma_s \sqrt{k}} \right)  \pi_s(dx) dy    \notag\\
& =    \int_{\bb P(V) \times \bb{R}_+}  
\bb E_{\bb Q_s^{x}} \left[ h_{\ee}(X_m, y + \ee - S_m); \, \tau_{y} \leq m  \right]    
  \mathds 1_{\{ y + 2\ee > \ee^{1/6}\sqrt{n}  \}}
 \phi^+ \left( \frac{y + \ee}{ \sigma_s \sqrt{k}} \right) \pi_s(dx) dy    \notag\\
& =    \int_{\bb P(V) \times \bb{R}_+} 
\bb E_{\bb Q_s^{x, *}}   \left[ \phi^+ \left( \frac{z - S_m^* + \ee}{ \sigma_s \sqrt{k}} \right)
  \mathds 1_{\{ z - S_m^* + 2\ee > \ee^{1/6}\sqrt{n}  \}}; \, \tau_{z}^*  \leq m  \right]    h_{\ee}(x, z) \pi_s(dx) dz,   
\end{align}
where $(S_m^*)$ is the dual random walk defined by \eqref{def-dual-randomwalk-Sn}.  
Since $\phi^+$ is bounded by $1$, by the martingale approximation (Lemma \ref{Martingale approx})
and the fact that $n$ is large enough so that $\ee^{1/6}\sqrt{n} \geq c$ for some constant $c>0$, we get 
\begin{align}\label{Proba_Holder}
 \bb E_{\bb Q_s^{x, *}}   \left[ \phi^+ \left( \frac{z - S_m^* + \ee}{ \sigma_s \sqrt{k}} \right)
  \mathds 1_{\{ z - S_m^* + 2\ee > \ee^{1/6}\sqrt{n}  \}}; \, \tau_{z}^*  \leq m  \right]  
 & \leq  \bb Q_s^{x, *} \left( z - S_m^* + 2\ee > \ee^{1/6}\sqrt{n},  \,   \min_{1 \leq j \leq m} (z - S_j^*) < 0  \right)   \notag\\
 & \leq   \bb Q_s^{x, *} \left( z - S_m^* -  \min_{1 \leq j \leq m} (z - S_j^*)   >  \ee^{1/6}\sqrt{n} - 2\ee  \right)   \notag\\
  & \leq   \bb Q_s^{x, *} \left( \max_{1 \leq j \leq m} (S_j^* - S_m^*)   >  \ee^{1/6}\sqrt{n} - 2\ee  \right)   \notag\\
 & \leq   \bb Q_s^{x, *} \left( \max_{1 \leq j \leq m} |M_j^*|  >  \frac{1}{2} \ee^{1/6}\sqrt{n}  \right). 
\end{align}
Using Lemma \ref{Lem_Fuk}
with $n$ replaced by $m= [ \ee^{1/2} n]$,
$u = v = \frac{1}{2} \ee^{1/6}\sqrt{ n}$ and $w = \ee^{5/12} n$, 
Markov's inequality and the fact that $\sup_{1 \leq i \leq m} \bb E_{ \bb Q_s^{x, *} } (\xi_i^2) \leq c$
for some constant $c>0$, we obtain
\begin{align}\label{Proba_002}
\bb Q_s^{x, *} \left( \max_{1 \leq j \leq m} |M_j^*|  > \frac{1}{2} \ee^{1/6}\sqrt{n}  \right)   
& \leq  \sum_{i=1}^m  \bb Q_s^{x, *}  \left( |\xi_i^*|  >  \frac{1}{2} \ee^{1/6}\sqrt{ n}  \right)   
  + 2  \bb Q_s^{x, *} \left(  \sum_{i=1}^m \bb E_{ \bb Q_s^{x, *} } 
  \left(  \xi_i^2  \Big|  \mathscr F_{i-1} \right) >  \ee^{5/12} n  \right) 
        +  c \ee^{1/12}   \notag\\
& \leq  4 \frac{m}{\ee^{1/3} n} \sup_{1 \leq i \leq m} \bb E_{ \bb Q_s^{x, *} } (\xi_i^2)
  + 2 \frac{m}{ \ee^{5/12} n }  \sup_{1 \leq i \leq m} \bb E_{ \bb Q_s^{x, *} } (\xi_i^2) 
   + c \ee^{1/12}   
  \leq   c \ee^{1/12}.  
\end{align}
Combining \eqref{CLLT-J1-6}, \eqref{Proba_Holder} and \eqref{Proba_002}, we get 
\begin{align}\label{Bound_J1}
\int_{\bb P(V) \times \bb R_+}  L_{\ee} (x, y)  \phi^+\left(\frac{y}{\sigma_s \sqrt{k}} \right) \pi_s(dx) dy 
 \leq   c \ee^{1/12}  \| h_{\ee} \|_{\pi_s \otimes \Leb}.
\end{align}
For the last term in \eqref{eqt-A 001_Lower},
by the definition of $L_{\ee}$ (cf.\ \eqref{DefM88}) and $\tau_y$ (cf.\ \eqref{Def-tau-y}), we have 
\begin{align*}
\| L_{\ee} \|_{\mathscr H} 
& =   \int_{\bb R}  \mathds 1_{\{ y' + \ee > \ee^{1/6}\sqrt{n}  \}}  
\sup_{x' \in \bb P(V)}  \bb E_{\bb Q_s^{x'}} \left[ h_{\ee}(X_m, y' - S_m); \, \tau_{y' - \ee} \leq m  \right] dy'  \notag\\
& \leq  \sup_{x \in \bb P(V)} \sup_{y \in \bb R  } h_{\ee} (x, y)
 \int_{ \frac{1}{2} \ee^{1/6}\sqrt{n} }^{\infty} \sup_{x' \in \bb P(V)} 
\bb Q_s^{x'}  \left( \tau_{y' - \ee} \leq m  \right) dy'   \notag\\
& \leq  2 \sup_{x \in \bb P(V)} \sup_{y \in \bb R  } h_{\ee} (x, y)
  \int_{ \frac{1}{4} \ee^{1/6}\sqrt{n} }^{\infty} \sup_{x' \in \bb P(V)} 
\bb Q_s^{x'}  \left( \max_{1 \leq j \leq m} |S_j| \geq  t   \right) dt. 
\end{align*}
By the martingale approximation (Lemma \ref{Martingale approx}),
Doob's martingale maximal inequality and Lemma \ref{Lp boud martingales},
we get that for $t \geq \frac{1}{4} \ee^{1/6}\sqrt{n}$ and $\delta>0$,  
\begin{align*}
& \sup_{x' \in \bb P(V)} 
\bb Q_s^{x'}  \left( \max_{1 \leq j \leq m} |S_j| \geq t   \right)
 \leq   \sup_{x' \in \bb P(V)} 
\bb Q_s^{x'}  \left( \max_{1 \leq j \leq m} |M_j| \geq  \frac{t}{2}   \right)   \notag\\
& \leq    \frac{c}{ t^{2 + \delta} }  \sup_{x' \in \bb P(V)} \bb E_{ \bb Q_s^{x'} }  (M_m^{2+\delta})   
 \leq  \frac{c' m^{(2 + \delta)/2} }{ t^{2 + \delta} } \leq  \frac{c' (\ee^{1/2} n)^{1 + \frac{\delta}{2}} }{ t^{2 + \delta} }, 
\end{align*}
so that 
\begin{align}\label{Bound-third-Cara}
\| L_{\ee} \|_{\mathscr H}  
 \leq c'  (\ee^{1/2} n)^{1 + \frac{\delta}{2}}   \sup_{x \in \bb P(V)} \sup_{y \in \bb R  } h_{\ee} (x, y)
\int_{ \frac{1}{4} \ee^{1/6}\sqrt{n} }^{\infty}   \frac{1}{ t^{2 + \delta} } dt   
 \leq  c \ee^{1/3} \sqrt{n} \sup_{x \in \bb P(V)} \sup_{y \in \bb R  } h_{\ee} (x, y).  
\end{align}
Substituting \eqref{Bound_J1} and \eqref{Bound-third-Cara} into \eqref{eqt-A 001_Lower} gives
\begin{align}\label{BoundMt99}
J_{n,2}(x,y)
 \leq   \left( c \ee^{1/12} + c_{\ee} \alpha_n + c_{\ee}  n^{-\ee} \right)  \frac{ 1 + y }{n}    \| h_{\ee} \|_{\pi_s \otimes \Leb} 
    +   \frac{c_{\ee} (1 + y) }{ n }  \sup_{x \in \bb P(V)} \sup_{y \in \bb R  } h_{\ee} (x, y). 
\end{align}
From \eqref{KKK-111-001}, \eqref{K1-final bound} and \eqref{BoundMt99},
we get the upper bound for $I_{n,2}(x, y)$: 
\begin{align*}
I_{n,2}(x, y) 
 \leq  \left( c \ee^{1/12} + c_{\ee} \alpha_n +  c_{\ee}  n^{-\ee}  \right)  \frac{ 1 + y}{n}  \| h_{\ee} \|_{\pi_s \otimes \Leb}  
  +  \frac{c_{\ee} (1 + y)}{n^{3/2}}  \| h_{\ee} \|_{\mathscr H}  
  +   \frac{c_{\ee} (1 + y) }{ n }  \sup_{x \in \bb P(V)} \sup_{y \in \bb R  } h_{\ee} (x, y). 
\end{align*}
Combining this with \eqref{staring point for the lower-bound-001} and \eqref{Lower_In_lll}, 
and using the fact that $\| h \|_{\pi_s \otimes \Leb} \leq \| h_{\ee} \|_{\pi_s \otimes \Leb}$
and $\| h \|_{\mathscr H} \leq \| h_{\ee} \|_{\mathscr H}$,
we conclude the proof of the lower bound \eqref{eqt-A 002}. 
\end{proof}

From Theorem \ref{t-A 001}, we shall establish the following result, which plays a crucial role in proving Theorem \ref{Thm_In_Pro}.

\begin{proposition}\label{Lem_Condi_inf}
Assume conditions \ref{Condi-density}, \ref{Condi-density-invariant} 
and $\kappa'(s) = 0$ for some $s \in  I_{\mu}$.   
Let $(\alpha_n)_{n \geq 1}$ be any sequence of positive numbers satisfying $\lim_{n \to \infty} \alpha_n = 0$. 
Let $(a_n)_{n \geq 1}$ be any sequence of nonnegative numbers satisfying $\limsup_{n \to \infty} \frac{a_n}{\sqrt{n}} < \infty$. 
Then, for any $\Delta_0 > 0$ and $b \in \bb R_+$, there exists a constant $c >0$ such that for any
$x \in \bb P(V)$, $y \in [0, \alpha_n \sqrt{n}]$, $\Delta_n \in [\Delta_0, b \sqrt{n}]$
 and any measurable set $A \subseteq \bb P(V)$ satisfying $\pi_s(\partial A) = 0$,
\begin{align*}
J(x, y) : = \bb Q_s^x  \left( X_{2n} \in A, \,   
    \min_{n < j \leq 2n} (y -S_j) \geq  a_n,  \,   y -S_{2n} \in  [a_n, a_n + \Delta_n], \,  \tau_y > n  \right)   
  \geq \frac{c (1 + y) }{ n^{3/2}  }   \Delta_n^2 \pi_s(A).  
\end{align*}
\end{proposition}

\begin{proof}
By the Markov property of the pair $(X_n, S_n)$,  
we have for any $n \geq 1$, $x \in \bb P(V)$ and $y \in \bb R$,  
\begin{align*}  
J(x, y) 
 =  \int_{\bb P(V) \times \bb R_+} 
 h(x', y')  \bb Q_s^x  \left( X_n \in dx',  \  y -S_n \in dy',  \   \tau_y > n  \right)  
  =  \bb E_{\bb Q_s^x} \big[ h(X_n, y - S_n); \,  \tau_{y} > n  \big], 
\end{align*}
where for brevity we denote for any $x' \in \bb P(V)$ and $y' \geq 0$, 
\begin{align}\label{Def-hxy-abc}
h(x', y') 
= \bb Q_s^{x'}  \left( X_{n} \in A, \,   \min_{1 \leq j \leq n} (y' -S_j) \geq  a_n,  \,   y' - S_{n} \in  [a_n, a_n + \Delta_n]  \right). 
\end{align}
Applying the conditioned local limit theorem (cf.\ \eqref{eqt-A 002} of Theorem \ref{t-A 001}),
we get
\begin{align}\label{Lower-J1xy}
J(x, y)   \geq  J_1(x,y) - J_2(y) - J_3(y) - J_4(y),
\end{align}
where, with $h_{-\ee} \leq_{\ee}  h \leq_{\ee} h_{\ee}$, 
\begin{align*}
& J_1(x, y) =  \frac{2V_s(x, y)}{ n \sigma_s^2 \sqrt{2\pi } }  
 \int_{\bb P(V) \times \bb R_+}  h_{-\ee} (x', y')  \phi^+\left(\frac{y'}{\sigma_s \sqrt{n}} \right)  \pi_s(dx') dy',  \nonumber\\
& J_2(y) =  \left( c \ee^{1/12} + c_{\ee} \alpha_n +  c_{\ee} n^{-\ee} \right)  \frac{ 1 + y}{n} \| h_{\ee} \|_{\pi_s \otimes \Leb}, \nonumber\\
& J_3(y) =  \frac{c_{\ee} (1 + y)}{ n^{3/2} }  \| h_{\ee} \|_{\mathscr H},  \notag\\
& J_4(y) =  \frac{c_{\ee} (1 + y) }{ n }  
   \sup_{x \in \bb P(V)} \sup_{y \in \bb R  } h_{\ee} (x, y).  
\end{align*}
\textit{Lower bound of $J_1(x,y)$.} 
By the definition of $h$ (cf.\ \eqref{Def-hxy-abc}), we have
\begin{align*}
 h_{-\ee} (x', y')  
&  \geq  \bb Q_s^{x'}  \left( X_{n} \in A, \,    y' - S_{n} \in  [a_n + \ee, a_n + \Delta_n - \ee],  
     \,   \min_{1 \leq j \leq n} (y' -S_j) \geq  a_n + \ee \right)  \nonumber\\
& =  \bb Q_s^{x'}  \left( X_{n} \in A, \,   y' - a_n - \ee - S_{n} \in  [0,  \Delta_n - 2\ee],  
    \,  \tau_{y' - a_n - \ee} > n   \right).  
\end{align*}
Hence, by a change of variable $y' - a_n - \ee = t$ and the duality lemma (cf. Lemma \ref{Lem_Duality}), we get
\begin{align}\label{Pf_Bound_J1_ss}
& \int_{\bb P(V) \times \bb R_+}  h_{-\ee} (x', y')  \phi^+\left(\frac{y'}{\sigma_s \sqrt{n}} \right)  \pi_s(dx') dy'     \notag\\
& \geq  \int_{\bb P(V)} \int_{-a_n - \ee}^{\infty}  \bb Q_s^{x'}  \left( X_{n} \in A, \,   
    t - S_{n} \in  [0,  \Delta_n - 2\ee],   \,  \tau_t > n  \right)   
      \phi^+\left(\frac{t + a_n + \ee}{\sigma_s \sqrt{n}} \right)  dt \pi_s(dx')      \notag\\
& \geq  \int_{\bb P(V) \times \bb R_+}   \bb Q_s^{x'}  \left( X_{n} \in A, \,   
    t - S_{n} \in  [0,  \Delta_n - 2\ee],  \,  \tau_t > n  \right)  
      \phi^+\left(\frac{t + a_n + \ee}{\sigma_s \sqrt{n}} \right) dt  \pi_s(dx')   \notag\\
& =  \int_{A}  \int_0^{ \Delta_n - 2\ee} 
      \mathbb{E}_{ \mathbb{Q}_s^{x, *}} 
  \left[ \phi^+ \left( \frac{z - S_n^* + a_n + \ee}{\sigma_s \sqrt{n}} \right); \,  \tau_z^* > n \right]  dz \pi_s(dx).  
\end{align}
For brevity we denote for $z \geq 0$, 
\begin{align*}
\psi(z) = \phi^{+} \left( \frac{z + a_n + \ee}{\sigma_s \sqrt{n}} \right). 
\end{align*}
Applying again \eqref{eqt-A 002} of Theorem \ref{t-A 001} (for the dual random walk $S_n^*$), we obtain that for any $z \geq 0$, 
\begin{align}\label{Lower-bound-expectq}
 \mathbb{E}_{ \mathbb{Q}_s^{x,*}} 
  \left[ \phi^+ \left( \frac{z - S_n^* + a_n + \ee}{\sigma_s \sqrt{n}} \right); \,  \tau_z^* > n \right]   
 =   \mathbb{E}_{ \mathbb{Q}_s^{x,*}}  \left[ \psi \left( z - S_n^* \right); \,  \tau_z^* > n \right]  
 \geq  K_1(x, z) - K_2(z), 
\end{align}
where 
\begin{align*} 
& K_1(x, z) = \frac{2V_s^*(x, z)}{ n \sigma_s^2 \sqrt{2\pi } }  
  \int_{\mathbb R_+}  \psi_{-\ee} (t)  \phi^+\left(\frac{t}{\sigma_s \sqrt{n}} \right) dt   \notag\\
&  K_2(z) =  \left( c \ee^{1/4} + c_{\ee} \alpha_n +  c_{\ee} n^{-\ee} \right)  \frac{ 1 + z }{n}    \| \psi_{\ee} \|_{\pi_s \otimes \Leb}
     +   \frac{c_{\ee} (1 + z)}{ n^{3/2} }  \| \psi_{\ee} \|_{\mathscr H}   
        +   \frac{c_{\ee} (1 + z) }{ n }  \sup_{y \in \bb R  } \psi_{\ee} (y). 
\end{align*}
For the first term $K_1(x, z)$, since
\begin{align*}
 \int_{\bb R_+}  \psi_{-\ee} (t) \phi^{+} \left( \frac{t}{\sigma_s \sqrt{n}} \right) dt
& =  \int_{\bb R_+}  \phi^{+}_{-\ee} \left( \frac{t + a_n + \ee}{\sigma_s \sqrt{n}} \right) 
     \phi^{+} \left( \frac{t}{\sigma_s \sqrt{n}} \right) dt   \nonumber\\
& = \sigma_s \sqrt{n} \int_{\bb R_+} 
    \phi^+_{-\ee}\left(y + \frac{a_n + \ee}{\sqrt{n}}\right)  \phi^{+} \left(y\right)  dy  \nonumber\\
& \geq  c \sqrt{n}, 
\end{align*}
we get that there exists a constant $c'>0$ such that for any $x \in \bb P(V)$ and $z \in \bb R_+$, 
\begin{align}\label{Lower-bound-K1-xz}
K_1(x, z) \geq  \frac{c'}{\sqrt{n}} V_s^*(x, z).  
\end{align}
For the second term $K_2(z)$, since 
\begin{align*}
\| \psi_{\ee} \|_{\pi_s \otimes \Leb} 
= \| \psi_{\ee} \|_{\mathscr H}   
= \int_{\bb R}  \phi^{+}_{\ee} \left( \frac{z + a_n + \ee}{\sigma_s \sqrt{n}} \right)  dz
\leq c \sqrt{n}
\end{align*}
and $\sup_{y \in \bb R  } \psi_{\ee} (y) \leq 1$, 
we have 
\begin{align}\label{Lower-bound-K2-z}
K_2(z) \leq  \left( c \ee^{1/4} + c_{\ee} \alpha_n +  c_{\ee} n^{-\ee} \right)  \frac{1 + z}{ \sqrt{n} }. 
\end{align}
Substituting \eqref{Lower-bound-expectq}, \eqref{Lower-bound-K1-xz} and \eqref{Lower-bound-K2-z} into \eqref{Pf_Bound_J1_ss},
we get
\begin{align*} 
 \int_{\bb P(V) \times \bb R_+}  h_{-\ee} (x', y')  \phi^+\left(\frac{y'}{\sigma_s \sqrt{n}} \right)  \pi_s(dx')  dy'     
&  \geq  \frac{c}{\sqrt{n}}  \int_{A}  \int_0^{ \Delta_n - 2\ee}  V_s^*(x, z)  dz \pi_s(dx)   \notag\\
 & \geq   \frac{c'}{\sqrt{n}}  \int_{A}  \int_0^{ \Delta_n /2 }  z  dz \pi_s(dx)
   = \frac{c''}{\sqrt{n}} \Delta_n^2 \pi_s(A), 
\end{align*}
where in the second inequality we used the fact that $\Delta_n > 4 \ee$ by taking $\ee>0$ sufficiently small 
and there exists a constant $c_1 >0$ 
such that $\inf_{x \in \bb P(V)}V^*(x, z) > c_1 z$ for any $z \geq 0$ (cf.\ Proposition \ref{Prop-harmonic}). 
Thus, 
we get 
\begin{align}\label{Pf_Bound_J1_final}
J_1(x, y) 
\geq  c  \frac{1 + y}{ n^{3/2} }  \Delta_n^2 \pi_s(A). 
\end{align}
\textit{Upper bound of $J_2(y)$.} 
By the definition of $h$ (cf.\ \eqref{Def-hxy-abc}), we have
\begin{align}\label{Upper-bound-h-ee}
 h_{\ee} (x', y')  
\leq  \bb Q_s^{x'}  \left( X_{n} \in A, \,   
  y' - a_n + \ee - S_{n} \in  [0,  \Delta_n + 2\ee],    \,    \tau_{y' - a_n + \ee} > n \right).  
\end{align}
By a change of variable $y' - a_n + \ee = t$ and the duality lemma (cf. Lemma \ref{Lem_Duality}), we get
\begin{align*}
 \| h_{\ee} \|_{\pi_s \otimes \Leb}
&  =  \int_{\bb P(V) \times \bb R}   h_{\ee} (x', y')  \pi_s(dx') dy'  \nonumber\\
& \leq   \int_{\bb P(V)}   \int_{-a_n}^{\infty} 
   \bb Q_s^{x'}  \left( X_{n} \in A,  \,  t - S_{n} \in  [0,  \Delta_n + 2\ee],  \,  \tau_t >n   \right)  \pi_s(dx') dt   \nonumber\\ 
& =  \int_{\bb R}  \int_{\bb P(V)}   \mathds 1_{\{ t \in [-a_n, \infty) \}}
    \bb Q_s^{x'}  \left( X_{n} \in A,  \,  t - S_{n} \in  [0,  \Delta_n + 2\ee],  \,  \tau_t >n - 1   \right)   \pi_s(dx') dt \nonumber\\
& =    \int_{A}  \int_0^{ \Delta_n + 2\ee} 
   \mathbb{Q}_s^{x, *}  \left( z - S_n^* \in  [-a_n, \infty),   \tau_z^* > n - 1 \right)  dz \pi_s(dx)  \notag\\
& \leq   \int_{A}  \int_0^{ \Delta_n + 2\ee} 
   \mathbb{Q}_s^{x, *}  \left(  \tau_z^* > n - 1 \right)  dz \pi_s(dx)     \notag\\
& \leq  \frac{c}{\sqrt{n}}  \Delta_n^2 \pi_s(A),
\end{align*}
where in the last inequality we used Theorem \ref{Th_Asymptotic_ExitTime_dual} 
and the fact that $\Delta_n \geq \Delta_0 > 0$ and $\ee >0$ is sufficiently small.  
Thus
\begin{align}\label{Pf_Bound_J2_final}
J_2(y) \leq   \left( c \ee^{1/12} + c_{\ee} \alpha_n +  c_{\ee} n^{-\ee} \right)  
  \frac{ 1 + y}{n^{3/2}}  \Delta_n^2 \pi_s(A).  
\end{align}
\textit{Upper bound of $J_3(y)$.} 
Using \eqref{Upper-bound-h-ee},
a change of variable $y' - a_n + \ee = t$ and Lemma \ref{Lem_Integral_LLTnew_new2}, it follows that 
\begin{align*}
\| h_{\ee} \|_{\mathscr H}
& =   \int_{\bb R}   \sup_{x' \in \bb P(V)}  h_{\ee} (x', y')  dy'   \nonumber\\
& \leq  
  \int_{\bb R}  \sup_{x' \in \bb P(V)}   \bb Q_s^{x'}  
\left(   y' - a_n + \ee - S_{n} \in  [0,  \Delta_n + 2\ee], \,  \tau_{ y' - a_n + \ee }  > n \right) dy'   \nonumber\\
& =  \int_{\bb R}  \sup_{x' \in \bb P(V)}   \bb Q_s^{x'}  
\left(   t - S_{n} \in  [0,  \Delta_n + 2\ee], \,  \tau_{ t }  > n \right) dt   \nonumber\\
& \leq  c_{\ee} \frac{ \Delta_n^2 \log^{1- \ee} n }{\sqrt{n}}, 
\end{align*}
where $\ee \in [0, \frac{1}{2})$. 
Therefore,
\begin{align}\label{Pf_Bound_J3_final}
J_3(y) = \frac{c_{\ee} (1 + y)}{ n^{3/2} }  \| h_{\ee} \|_{\mathscr H} 
    \leq  c_{\ee} (1 + y)  \frac{\Delta_n^2 \log^{1- \ee} n}{ n^{5/2} }. 
\end{align}
\textit{Upper bound of $J_4(y)$.}
To give an upper bound for $h_{\ee} (x', y')$, 
we shall consider two cases: when $y' \in [- \ee,  \eta \sqrt{ n \log n} ]$ and $y' \in (\eta \sqrt{ n \log n}, \infty )$
with $\eta>0$ whose value will be chosen to be sufficiently large. 
Using \eqref{Upper-bound-h-ee} and Lemma \ref{Lem_CondiLLT_cc}, we get that there exists a constant $c >0$
such that for any $x' \in \bb P(V)$ and $y' \in [- \ee,  \eta \sqrt{ n \log n} ]$, 
\begin{align}\label{Bound-h-xy-ysmall}
h_{\ee} (x', y')
& \leq  \bb Q_s^{x'}   
 \left(    y' - a_n + \ee - S_{n} \in  [0,  \Delta_n + 2\ee],   \,   \tau_{y' - a_n + \ee}  > n   \right)  \notag\\
& \leq  c_{\ee, \eta} \frac{ \sqrt{n \log n} }{n^{3/2}} \Delta_n^2 
  =  c_{\ee, \eta} \frac{ \sqrt{ \log n} }{n} \Delta_n^2. 
\end{align}
In a similar way as in the proof of \eqref{Inequa-Moderate}, 
by taking $\eta >0$ to be sufficiently large and 
using the lower tail moderate deviation asymptotic for $-S_n$ 
under the changed measure $\mathbb{Q}_s^{x'}$, 
we derive that there exists a constants $c_{\eta} >0$ such that 
for any $x' \in \bb P(V)$ and  
$y' \in (\eta \sqrt{ n \log n}, \infty )$
\begin{align}\label{Bound-h-xy-ylarge}
h_{\ee} (x', y')
& \leq  \bb Q_s^{x'}   
 \left(    - S_{n}  \leq - \frac{1}{2}  \eta \sqrt{ n \log n}  \right)  
 \leq  \frac{c_{\eta}}{n}. 
 \end{align}
Combining \eqref{Bound-h-xy-ysmall} and \eqref{Bound-h-xy-ylarge} gives 
\begin{align}\label{Pf_Bound_J4_final}
J_4(y) \leq  c_{\ee} (1 + y)    \frac{ \sqrt{ \log n} }{n^2} \Delta_n^2.  
\end{align} 
Putting together \eqref{Lower-J1xy}, \eqref{Pf_Bound_J1_final}, \eqref{Pf_Bound_J2_final}, 
\eqref{Pf_Bound_J3_final} and \eqref{Pf_Bound_J4_final} 
concludes the proof of Proposition \ref{Lem_Condi_inf}. 
\end{proof}

\section{Proof of Theorem \ref{Thm_In_Pro}}

The proof of Theorem \ref{Thm_In_Pro} is splitted into two parts: the upper bound (see Lemma \ref{Lem-Minimal-Upp} below)
and the lower bound (see Lemma \ref{Lem-Minimal-Low} below).

\subsection{Proof of the upper bound}
In this section we shall obtain the upper bound in Theorem \ref{Thm_In_Pro}. 
We only give a proof of \eqref{thm1_Minimal_aa_SetA} 
since the proof of \eqref{thm1_Minimal_aa} follows from that of \eqref{thm1_Minimal_aa_SetA} by taking $A = \bb P(V)$.

\begin{lemma}\label{Lem-Minimal-Upp}
Assume conditions \ref{Condi_N},  \ref{Condi-density}, \ref{Condi-density-invariant} and \ref{Condi_ms}. 
Let $x \in \bb P(V)$. 
Then, there exists a constant $c> 0$ such that for any $\ee \in (0, \frac{3}{2 \alpha})$, 
any Borel set $A \subseteq \bb P(V)$ and $n \geq 1$, 
\begin{align}\label{Pf_Minimal_Upp}
 I: =  \bb P \left(\left. \frac{M_n^x(A)}{\log n}  \geq - \frac{3}{2 \alpha} +  \ee  \, \right| \,  \mathscr S \right) 
 \leq c \frac{\log^3 n}{n^{\ee \alpha}}. 
\end{align}
\end{lemma}

\begin{proof}
Let $K >1$. We write $I = I_1 + I_2$, 
where
\begin{align*} 
I_1  & = \bb P \left(  \left.   \frac{M_n^x(A)}{\log n} \geq - \frac{3}{2 \alpha} + \ee, 
  \  \max_{1 \leq i \leq n}  M_i^x(A) \leq  K    \, \right| \,  \mathscr S \right),   \nonumber\\
I_2  & =  \bb P \left(  \left.    \frac{M_n^x(A)}{\log n}  \geq - \frac{3}{2 \alpha} + \ee, 
  \   \max_{1 \leq i \leq n} M_i^x(A) >  K   \, \right| \,  \mathscr S \right),
\end{align*}
with the notation 
$M_i^x(A) = \sup \{ S_{u|i}^x: |u| = n,  \,  X_u^x \in A  \}$ for $x \in \bb P(V)$ and $1 \leq i \leq n$. 
For the first term $I_1$, we have $I_1 \leq \bb E (Z_n^x(A) \, | \,  \mathscr S)$, 
where, for $x \in \bb P(V)$ and $n \geq 1$, 
\begin{align*}
Z_n^x(A) = \sum_{|u| = n} 
\mathds{1}_{ \left\{X^x_u \in A,  \  \frac{S^x_u}{\log n} \geq - \frac{3}{2 \alpha} +  \ee, \  S^x_{u|i} \leq  K,  \  \forall \,  1 \leq i \leq n \right\} }. 
\end{align*}
Using the many-to-one formula \eqref{Formula_many_to_one}, 
the fact that $\mathfrak m (\alpha)  = 1$ (cf.\ condition \ref{Condi_ms})
and $r_{\alpha}$ is bounded and strictly positive on $\bb P(V)$ (cf.\ Proposition \ref{prop:Ps}),  
we get that there exists a constant $c>0$ such that for any $x \in \bb P(V)$ and $n \geq 1$, 
\begin{align*}
 \bb E(Z_n^x(A))   
& = r_{\alpha}(x) \bb E_{\bb Q_{\alpha}^x}  
 \left[ \frac{1}{r_{\alpha} (X_n)} e^{- \alpha S_n} 
   \mathds{1}_{ \{ X_n \in A, \  S_n \geq  (- \frac{3}{2 \alpha} + \ee) \log n, 
        \   S_i \leq K,  \  \forall 1 \leq i \leq n \} } \right]  \nonumber\\
&  \leq  c n^{- (- \frac{3}{2 \alpha} + \ee) \alpha} 
    \bb Q_{\alpha}^x \left( S_n \geq  \Big( - \frac{3}{2 \alpha} + \ee \Big) \log n, 
                    \  S_i \leq  K, \  \forall 1 \leq i \leq n \right) \nonumber\\
&  =  c n^{ \frac{3}{2} - \ee \alpha}    
   \bb Q_{\alpha}^x \left( K - S_n \in \left[0,  K + \left(\frac{3}{2 \alpha} - \ee \right) \log n \right], 
         \, \tau_K > n  \right).   
\end{align*}
Applying Lemma \ref{Lem_CondiLLT_cc} and taking $K = a \log n$ with $a>0$ (whose value will be chosen to be sufficiently large), we obtain
\begin{align*} 
\bb E(Z_n^x(A)) 
\leq  c  n^{ \frac{3}{2} - \ee \alpha}   (1 + K) \frac{[ K + ( \frac{3}{2 \alpha} - \ee) \log n ]^2}{n^{3/2}}   
 \leq c \frac{a^3 \log^3 n}{n^{\ee \alpha}}. 
\end{align*}
Since $\bb P(\mathscr S) > 0$, it follows that 
\begin{align}\label{Mini_I1_001}
I_1 \leq \bb E (Z_n^x(A) \, | \,  \mathscr S) \leq  c \frac{a^3 \log^3 n}{n^{\ee \alpha}}. 
\end{align}
For the second term $I_2$, by Markov's inequality, the many-to-one formula \eqref{Formula_many_to_one}
and the fact that $\mathfrak m (\alpha)  = 1$, we have  
\begin{align}\label{Mini_I2_001}
I_2 & \leq  c \bb P \left( \max_{1 \leq i \leq n} M_i^x(A) >  K \right)  
 \leq  c \bb P \left(  \max_{1 \leq i \leq n} \max_{|u|=i}  S^x_{u} >  K   \right)  \nonumber\\
& \leq  c e^{- \alpha K} \bb E \left( e^{ \alpha \max_{1 \leq i \leq n} \max_{|u|=i}  S^x_{u}} \right)   
 \leq  c e^{- \alpha K} \bb E \left( \max_{1 \leq i \leq n} \sum_{|u|=i} e^{\alpha  S^x_{u}} \right)   \nonumber\\ 
& \leq  c e^{- \alpha K}  \sum_{i=1}^n \bb E \left(  \sum_{|u|=i} e^{\alpha S^x_{u}} \right)  
=  c e^{- \alpha K} r_{\alpha}(x)  \sum_{i=1}^n \bb E_{\bb Q_{\alpha}^x } \left( \frac{1}{ r_{\alpha}(X_i) } \right)
 \leq  \frac{c}{n^{a \alpha - 1}}  \leq \frac{c}{n^{a \alpha/2}} ,
\end{align}
by taking $K = a \log n$ and $a >0$ sufficiently large. 
From \eqref{Mini_I1_001} and \eqref{Mini_I2_001}, 
we get \eqref{Pf_Minimal_Upp}. 
\end{proof}

\subsection{Proof of the lower bound}
We will use the following auxiliary assertion
for the minimal position $m_n^x(A)$ defined in \eqref{MinPosition001A}. 

\begin{lemma} \label{lem-liminf of min-001}
Assume conditions \ref{Condi_N},  \ref{Condi-density}, \ref{Condi-density-invariant} and \ref{Condi_ms}. 
Let $x \in \bb P(V)$. 
Then, for any Borel set $A \subseteq \bb P(V)$ satisfying $\nu(A) > 0$ and $\nu(\partial A) = 0$, 
there exists a constant $c_0 >0$ such that, conditionally on the system's survival, 
\begin{align}\label{Upp_cons}
\liminf_{n \to \infty} \frac{1}{n} m_n^x(A) 
\geq  - c_0  \quad  \bb P\mbox{-a.s.}
\end{align}
\end{lemma}

\begin{proof}
By Proposition \ref{Prop-SLLN-direction}, 
for each Borel set $A \subseteq \bb P(V)$ satisfying $\nu(A) > 0$ and $\nu(\partial A) = 0$, 
we have, on the survival set $\scr S$, $\liminf_{n \to \infty} \sum_{|u| = n} \mathds{1}_{ \{ X^x_u \in A \} }  >0$, $\bb P$-a.s. 
Consider $Y_n: = \sum_{|u| = n} \mathds{1}_{ \{ X^x_u \in A,  \   S^x_u < - c_0 n \} }$, 
where $c_0 >0$ will be chosen sufficiently large. 
It suffices to show that, as $n \to \infty$, it holds $Y_n \to 0$, $\bb P$-a.s.
Clearly, $Y_n \leq e^{- c_0  \delta n} \sum_{|u| = n} e^{- \delta S^x_u}$ for any $\delta >0$. 
By Markov's inequality, the many-to-one formula \eqref{Formula_many_to_one}
and the fact that $\mathfrak m (\alpha)  = 1$ (cf.\ condition \ref{Condi_ms}), we get
\begin{align*}
\bb P (Y_n > 0) = \bb P (Y_n \geq 1)  \leq \bb E (Y_n) 
 \leq e^{- c_0 \delta n} \bb E \Big( \sum_{|u| = n} e^{-\delta S^x_u}  \Big)  
 =  e^{- c_0 \delta n} r_{\alpha}(x) \bb E_{\bb Q_{\alpha}^x}  
 \left[ \frac{1}{r_{\alpha} (X_n)} e^{- (\alpha + \delta) S_n}  \right] = : J. 
\end{align*}
Using \eqref{Formu_ChangeMea} and taking  $\delta > 0$ sufficiently small and then $c_0 > 0$ sufficiently large, 
we obtain that there exists a constant $c > 0$ such that 
\begin{align*}
J = e^{- c_0  \delta n} \frac{1}{\kappa(\alpha)^n} \bb E_{x} \left( e^{ -\delta S_n } \right)
\leq e^{- c_0  \delta n} \frac{1}{\kappa(\alpha)^n}  \left( \bb E e^{\delta \log \|g^{-1}\|} \right)^n
\leq e^{-cn}. 
\end{align*}
By the Borel-Cantelli lemma, it follows that $\lim_{n \to \infty} Y_n = 0$, $\bb P$-a.s., and hence \eqref{Upp_cons} holds. 
\end{proof}

We next proceed to give a lower bound for the lower tail of $M_{n}^x(A)$.

\begin{lemma}\label{Lem-intermediate001}
Assume conditions \ref{Condi_N},  \ref{Condi-density}, \ref{Condi-density-invariant} and \ref{Condi_ms}. 
Then, for any $\Delta >0$, 
there exists a constant $c> 0$ such that for  any $n \geq 2$ and any Borel set $A \subseteq \bb P(V)$ 
satisfying $\nu(A) > 0$ and $\nu(\partial A) = 0$, 
\begin{align} \label{important lower bound-001}
\bb P \left( M_{n}^x(A)  \geq -\frac{3}{2 \alpha} \log n - \Delta \right) 
\geq  \frac{1}{c \log^3 n}. 
\end{align}
\end{lemma}

\begin{proof}
Let $\Delta >0$ be a fixed constant and set 
\begin{align*}
 U_n^x(A) : = \sum_{|u| = 2n}  \mathds{1}_{ \{ X^x_u \in A \} }
\mathds{1}_{ \left\{ S^x_u \in  \left[ -\frac{3}{2 \alpha} \log n - \Delta, 
  \  - \frac{3}{2 \alpha} \log n  \right]  \right\} }  
  \mathds{1}_{ \left\{ \max_{1 \leq i \leq n}  S^x_{u|i} \leq 0,  \  
  \max_{n < j \leq 2n} S^x_{u|j} \leq  - \frac{3}{2 \alpha} \log n  \right\} }. 
\end{align*}
By the definition of $U_n^x(A)$ and $M_{2n}^x(A)$,  it holds that  
\begin{align}\label{Pf_LowerBound_Zn_aa}
 \bb P \left( M_{2n}^x(A)  \geq -\frac{3}{2 \alpha} \log n - \Delta \right)  
\geq  \bb P (U_n^x(A) > 0). 
\end{align}
For simplicity, denote for any integer $n \geq 1$, 
\begin{equation}\label{Def_I_k}
I_k = I_k(n) =
\begin{cases}
(-\infty, 0]  & 1 \leq k \leq n\\
\big(-\infty, \,  - \frac{3}{2 \alpha} \log n \big]   &   n < k < 2n \\
\big[ -\frac{3}{2 \alpha} \log n - \Delta,  \,  - \frac{3}{2 \alpha} \log n  \big]   &   k = 2n.
\end{cases}
\end{equation}
Then $U_n^x(A)$ can be rewritten as 
\begin{align*}
U_n^x(A) = \sum_{|u| = 2n} \mathds{1}_{ \{ X^x_u \in A \} }
     \mathds{1}_{ \left\{ S^x_{u|k} \in I_k, \ \forall 1 \leq k \leq 2n  \right\} }.
\end{align*}
Since $\bb E(U_n^x(A)) \geq 0$, by the Cauchy-Schwarz inequality, it holds that 
\begin{align*}
\bb P (U_n^x(A) > 0) \geq \frac{(\bb E U_n^x(A))^2}{\bb E \left[(U_n^x(A))^2 \right]}. 
\end{align*}
Hence, in order to give a lower bound for $\bb P (U_n^x(A) > 0)$, 
it suffices to provide a lower bound for $\bb E (U_n^x(A))$ and an upper bound for $\bb E[(U_n^x(A))^2]$. 
To deal with $\bb E (U_n^x(A))$, by the many-to-one formula \eqref{Formula_many_to_one}
and the fact that $\mathfrak m (\alpha)  = 1$ (cf.\ condition \ref{Condi_ms}), we get
\begin{align*}
\bb E (U_n^x(A))  = r_{\alpha}(x) \bb E_{\bb Q_{\alpha}^x}  
 \left[ \frac{1}{r_{\alpha} (X_{2n})}  e^{-\alpha S_{2n}}  \mathds{1}_{ \{ X_{2n} \in A \} }
   \mathds{1}_{ \left\{ S_k \in  I_k, \  \forall 1 \leq k \leq 2n \right\} } \right]. 
\end{align*}
Since $r_{\alpha}$ is strictly positive and bounded on $\bb P(V)$ (cf.\ Proposition \ref{prop:Ps})
and $S_{2n} \leq - \frac{3}{2 \alpha} \log n$,
there exist constants $c, c', c'' > 0$ such that for any $x \in \bb P(V)$, 
\begin{align}\label{Lower-UnA}
\bb E (U_n^x(A)) & \geq c \bb E_{\bb Q_{\alpha}^x}  
 \left[ e^{- \alpha  S_{2n}}  \mathds{1}_{ \{ X_{2n} \in A \} }
   \mathds{1}_{ \{ S_k \in  I_k, \  \forall 1 \leq k \leq 2n \} } \right]  \nonumber\\
& \geq c n^{3/2} \bb Q_{\alpha}^x (X_{2n} \in A, \,   S_k \in  I_k,  \  \forall 1 \leq k \leq 2n)  \notag\\
& \geq  c \Delta^2 \pi_s(A) \geq c' \Delta^2 \nu(A) > c'' \Delta^2, 
\end{align}
where in the last line we used Proposition \ref{Lem_Condi_inf} and Remark \ref{Rem-absolute-con}.

Next we are going to give an upper bound for $\bb E[(U_n^x(A))^2]$.
By the definition of $U_n^x(A)$, we have 
\begin{align*}
  (U_n^x(A))^2 
 =  \sum_{|u| = 2n} \sum_{|w| = 2n} 
  \mathds{1}_{ \left\{ X^x_u \in A, \,  X^x_{w} \in A,  \, S^x_{u|k} \in I_k, \,  S^x_{w|k}  \in I_k, \, \forall 1 \leq k \leq 2n  \right\} }. 
\end{align*}
We can rearrange this sum by summing over the generation $j$ 
of the last common ancestor $z$ (with $|z|=j$) of $u$ and $w$, to obtain
\begin{align*}
\bb E[(U_n^x(A))^2] = \bb E (U_n^x(A)) + Y_n(A).
\end{align*}
The first summand corresponds to the case $u=w$, while the second summand is given by
\begin{align*}
 Y_n(A)  = & 2 \bb E \bigg[   \sum_{j = 1}^{2n - 1} \sum_{|z| = j} 
 \mathds{1}_{ \left\{ S^x_{z|i} \in I_i, \ \forall 1 \leq i \leq j  \right\} }    \nonumber\\
&\qquad \times \sum_{1 \leq l < m \leq N_z}\sum_{|u|=2n-j-1} \sum_{|w|=2n-j-1}
 \mathds{1}_{ \left\{ X^x_{zlu} \in A, \,  X^x_{zmw} \in A,  \,  S^x_{zlu|k} \in I_k, \  S^x_{zmw|k} \in I_k, \ \forall j < k \leq 2n  \right\} }
 \bigg]. 
\end{align*}
Here $N_z$ denotes the number of childrens of the particle $z$, 
$l$ and $m$ denote different children of the last common ancestor $z$, the sums over $u$ and $w$ correspond to the paths leading from $zl$ and $zm$ to their children in generation 2n. 
Note that $N_z$ is independent of $\scr{F}_{j}$. 
Taking conditional expectation with respect to $\scr{F}_{j+1}$, we get 
\begin{align}\label{YnA-equ} 
 Y_n(A)  = 2 \bb E \bigg[   \sum_{j = 1}^{2n - 1} \sum_{|z| = j} 
 \mathds{1}_{ \left\{ S^x_{z|i} \in I_i, \ \forall 1 \leq i \leq j  \right\} }   \sum_{1 \leq l < m \leq N_z} 
 \mathds{1}_{ \left\{ S^x_{zl} \in I_{j+1}, \  S^x_{zm} \in I_{j+1}  \right\} }   
  K_{j,n} \left( X^x_{zl}, S^x_{zl} \right)  
    K_{j,n} \left( X^x_{zm}, S^x_{zm} \right)   \bigg],
\end{align}
where for $x' \in \bb P(V)$ and $y \in I_{j+1}$, 
\begin{align*} 
K_{j,n}(x', y) 
 =  \bb E  \bigg[ \sum_{|u | = 2n-j-1} 
    \mathds{1}_{ \left\{ X^{x'}_u \in A, \   y + S^{x'}_{u|l}  \in I_{l+j+1}, \  \forall 1 \leq l \leq 2n-j-1 \right\} } \bigg]. 
\end{align*}
By the many-to-one formula \eqref{Formula_many_to_one}
and the fact that $\mathfrak m (\alpha)  = 1$ (cf.\ condition \ref{Condi_ms}), we get
\begin{align*}
& K_{j,n}(x', y)    
 = r_{\alpha}(x')  \bb E_{\bb Q_{\alpha}^{x'}}  
 \bigg[ (r_{\alpha}^{-1} \mathds 1_{A}) (X_{2n-j-1}) e^{- \alpha S_{2n-j-1}} 
   \mathds{1}_{ \left\{ y + S_l \in I_{l+j+1}, \  \forall 1 \leq l \leq 2n-j -1 \right\} } \bigg].
\end{align*}
Since $y + S_{2n-j-1} \in I_{2n}$, 
in view of \eqref{Def_I_k}, we have $S_{2n-j-1} \geq - \frac{3}{2 \alpha} \log n - \Delta - y$.
Since the eigenfunction $r_{\alpha}$ is strictly positive and bounded on $\bb P(V)$ (cf.\ Proposition \ref{prop:Ps}), 
 there exists a constant $c >0$ such that
\begin{align*}
 K_{j,n}(x', y)  & \leq c n^{3/2} e^{\alpha (\Delta + y)}  
  \bb Q_{\alpha}^{x'}  \left( y + S_l \in I_{l+j+1}, \  \forall 1 \leq l \leq 2n-j -1 \right)  \nonumber\\
& \leq  c n^{3/2}  e^{\alpha  y}  L_{j,n}(x', y),  
\end{align*}
where for $x' \in \bb P(V)$ and $y \in I_{j+1}$, 
\begin{align*}
L_{j,n}(x', y)   
  =  \bb Q_{\alpha}^{x'}  \bigg( 
    - y - S_{2n-j-1}  \in \left[ \frac{3}{2 \alpha} \log n,   \frac{3}{2 \alpha} \log n  + \Delta  \right],  \, \tau_{-y} > 2n-j-1 \bigg). 
\end{align*}
When $1 \leq j \leq n$, applying Lemma \ref{lemma 3n} with $n' = 2n-j-1$, 
$a = \frac{3}{2 \alpha} \log n$ and $b = \frac{3}{2 \alpha} \log n + \Delta$,
 we obtain that there exists a constant $c >0$ such that for any $x' \in \bb P(V)$ and $y \in I_{j+1}$, 
\begin{align*}
L_{j,n}(x', y)  
& \leq  \frac{c}{(n')^{3/2}} (1-y) (b-a + 1) \left( b+a+ 1 \right)  
 \leq  \frac{c \log n}{n^{3/2}} (1-y). 
\end{align*}
When $n < j \leq 2n - 1$, we use the local limit theorem \eqref{LLT_Upper_aa} to get
that there exists a constant $c >0$ such that for any $x' \in \bb P(V)$ and $y \in I_{j+1}$, 
\begin{align*}
L_{j,n}(x', y) 
& \leq \bb Q_{\alpha}^{x'}  \left( 
   - y - S_{2n-j-1}  \in \left[ \frac{3}{2 \alpha} \log n,   \frac{3}{2 \alpha} \log n  + \Delta  \right] 
      \right)  
       \leq \frac{c  }{(2n-j)^{1/2}}. 
\end{align*}
Therefore, we obtain that for $S^x_{zl} \in I_{j+1}$, 
\begin{align}\label{Bound-Kjn-u}
 K_{j,n} \left( X^x_{zl}, S^x_{zl} \right)  
 \leq 
\begin{cases}
c  \log n  \  e^{\alpha S^x_{zl}} \left| 1 - S^x_{zl} \right|,  &   1 \leq j \leq n   \\
c   \frac{ n^{3/2} }{(2n-j)^{1/2}}  e^{\alpha S^x_{zl}},  &   n < j \leq 2n - 1. 
\end{cases}
\end{align}  
Similarly, for $K_{j,n} ( X^x_{zm}, S^x_{zm} )$ we also have
\begin{align}\label{Bound-Kjn-w}
  K_{j,n} \left( X^x_{zm}, S^x_{zm} \right)  
 \leq 
\begin{cases}
c \log n \  e^{\alpha S^x_{zm} } \left| 1 - S^x_{zm} \right|,  &   1 \leq j \leq n   \\
c   \frac{ n^{3/2} }{(2n-j)^{1/2}}  e^{\alpha S^x_{zm} },  &   n < j \leq 2n - 1. 
\end{cases}
\end{align}
Therefore, substituting \eqref{Bound-Kjn-u} and \eqref{Bound-Kjn-w} into \eqref{YnA-equ}, 
we obtain 
\begin{align*}
 Y_n(A)  & \leq  c  \log^{2} n  \  2 \bb E \Bigg\{  \sum_{j = 1}^{n} \sum_{|z| = j} 
 \mathds{1}_{ \left\{ S^x_{z|i}  \in I_i, \ \forall 1 \leq i \leq j  \right\} } 
  \sum_{1 \leq l < m \leq N_z}  \mathds{1}_{ \left\{ S^x_{zl} \in I_{j+1}, \  S^x_{zm} \in I_{j+1}  \right\} } 
    e^{ \alpha S^x_{zl} + \alpha S^x_{zm} } 
   \left| 1 - S^x_{zl} \right|   \left| 1 - S^x_{zm} \right|  \Bigg\}  \nonumber\\
& \quad +  2 c n^3  \,  \bb E \Bigg\{ \sum_{j = n+1}^{2n-1} \frac{1}{2n-j}
    \sum_{|z| = j} 
     \mathds{1}_{ \left\{ S^x_{z|i} \in I_i, \  \forall 1 \leq i \leq j  \right\} }   
   \sum_{1 \leq l < m \leq N_z}   e^{ \alpha S^x_{zl} + \alpha S^x_{zm} }  \Bigg\}. 
\end{align*}
Note that 
\begin{align*}
 2 \sum_{1 \leq l < m \leq N_z}  
\mathds{1}_{ \left\{ S^x_{zl} \in I_{j+1}, \  S^x_{zm} \in I_{j+1}  \right\} } 
 e^{ \alpha S^x_{zl} + \alpha S^x_{zm} } 
   \left| 1 - S^x_{zl} \right|   \left| 1 - S^x_{zm} \right|   
 \leq  \bigg(  \sum_{1 \leq l  \leq N_z} \mathds{1}_{ \left\{ S^x_{zl} \in I_{j+1}  \right\} }  e^{ \alpha S^x_{zl} } 
   \left| 1 - S^x_{zl} \right|  \bigg)^2 
\end{align*}
and 
\begin{align*}
2 \sum_{1 \leq l < m \leq N_z}  \mathds{1}_{ \left\{ S^x_{zl} \in I_{j+1}, \  S^x_{zm} \in I_{j+1}  \right\} }   e^{ \alpha S^x_{zl} + \alpha S^x_{zm} } 
\leq \bigg( \sum_{1 \leq l  \leq N_z}  \mathds{1}_{ \left\{ S^x_{zl} \in I_{j+1}  \right\} }  e^{ \alpha S^x_{zl} }  \bigg)^2.   
\end{align*}
Taking conditional expectation, 
we get for any $1 \leq j \leq n$, $|z|=j$,
\begin{align*} 
 \bb E  \left[ \bigg(  \sum_{1 \leq l  \leq N_z}  \mathds{1}_{ \left\{ S^x_{zl} \in I_{j+1}  \right\} }  e^{ \alpha S^x_{zl}  } 
   ( 1 + |S^x_{zl} | )    \bigg)^2  \bigg|  \scr{F}_{j}  \right]   
   & \leq e^{ 2 \alpha S^x_z  }   \bb E  \left[ \bigg(  \sum_{1 \leq l  \leq N_z} e^{ \alpha (S^x_{zl} - S^x_z )  } 
   ( 1 + |S^x_z | + | S^x_{zl} - S^x_z |   )  \bigg)^2  \bigg|  \scr{F}_{j}  \right]    \notag\\
   & = e^{ 2 \alpha S^x_z  } M\big(S^x_z, X^x_z\big), \notag
\end{align*}
where, for $s \in \bb R$ and $y \in \bb P(V)$, 
\begin{align*}
 M(s,y)  =  \bb E \big( \sum_{|u|=1}   e^{\alpha S^y_u} (1+|s|+|S^y_u|) \big)^2 
  \leq 2 (1+|s|)^2\bb E  \big(\sum_{|u|=1}   e^{\alpha S^y_u} \big)^2 +2 \bb E \big(\sum_{|u|=1}   e^{\alpha S^y_u}|S^y_u| \big)^2.  
\end{align*} 
By condition \ref{Condi_ms},  we have $M(s, y) \leq c (1+|s|)^2$  and consequently
$$  \bb E  \left[ \bigg(  \sum_{1 \leq l  \leq N_z}  \mathds{1}_{ \left\{ S^x_{zl} \in I_{j+1}  \right\} }  e^{ \alpha S^x_{zl}  }   ( 1 + |S^x_{zl} | )    \bigg)^2  \bigg|  \scr{F}_{j}  \right]   \leq  c  e^{ 2 \alpha S^x_z  } \left( 1 + |S^x_z| \right)^2. $$ 
Similarly, it holds
\begin{align*}
 \bb E  \left[ \bigg(  \sum_{1 \leq l  \leq N_z}  \mathds{1}_{ \left\{ S^x_{zl} \in I_{j+1}  \right\} }  e^{ \alpha S^x_{zl}  }  
                  \bigg)^2  \bigg|  \scr{F}_{j}  \right]   
 = e^{ 2 \alpha S^x_z  }   \bb E  \left[ \bigg(  \sum_{1 \leq l  \leq N_z}  \mathds{1}_{ \left\{ S^x_{zl} \in I_{j+1}  \right\} }  e^{ \alpha (S^x_{zl} - S^x_z )  }   \bigg)^2  \bigg|  \scr{F}_{j}  \right]    
 \leq  c  e^{ 2 \alpha S^x_z  }.   
\end{align*}
It follows that 
\begin{align*}
Y_n(A) & \leq c  \log^2 n  \  \sum_{j = 1}^{n} \bb E \Bigg\{   \sum_{|z| = j} 
 \mathds{1}_{ \left\{ S^x_{z|i} \in I_i, \ \forall 1 \leq i \leq j  \right\} }
  e^{2 \alpha S^x_z }  \big( 1 + |S^x_z| \big)^2  \Bigg\}  \nonumber\\
& \quad  +  c n^3    \sum_{j = n+1}^{2n-1} \frac{1}{2n-j} \bb E \Bigg\{ 
    \sum_{|z| = j} \mathds{1}_{ \left\{ S^x_{z|i} \in I_i, \ \forall 1 \leq i \leq j  \right\} }   
 e^{2 \alpha S^x_z }  \Bigg\}. 
\end{align*}
Using the many-to-one formula \eqref{Formula_many_to_one}
and the fact that the eigenfunction $r_{\alpha}$ is strictly positive and bounded, we get 
\begin{align*}
Y_n(A) \leq Y_{n,1}(A) + Y_{n,2}(A),
\end{align*}
where 
\begin{align*}
& Y_{n,1}(A)  
   = c  \log^2 n  \,   \sum_{j = 1}^{n}  \bb E_{\bb Q_{\alpha}^x}  
   \left[  e^{\alpha S_j }  (1 - S_j)^2
    \mathds{1}_{ \left\{ S_i \leq 0, \  \forall 1 \leq i \leq j \right\} }  \right],   \notag\\
& Y_{n,2}(A) = c n^3    \sum_{j = n+1}^{2n-1} \frac{1}{2n-j} 
   \bb E_{\bb Q_{\alpha}^x}  
     \left[  e^{\alpha S_j }
      \mathds{1}_{ \left\{ S_i \in  I_i, \  \forall 1 \leq i \leq j \right\} }  \right]. 
\end{align*}
For $Y_{n,1}(A)$, since 
\begin{align*}
 e^{\alpha S_j }  (1 - S_j)^2
    \mathds{1}_{ \left\{ S_i \leq 0, \  \forall 1 \leq i \leq j \right\} }  
& =  e^{\alpha S_j }  (1 - S_j)^2
    \mathds{1}_{ \left\{ S_j \geq  -\frac{4}{\alpha} \log j \right\} } 
    \mathds{1}_{ \left\{ S_i \leq 0, \  \forall 1 \leq i \leq j \right\} }  \nonumber\\
& \quad    + e^{\alpha S_j }  (1 - S_j)^2 
    \mathds{1}_{ \left\{ S_j  < - \frac{4}{\alpha} \log j \right\} } 
    \mathds{1}_{ \left\{ S_i \leq 0, \  \forall 1 \leq i \leq j \right\} }   \nonumber\\
& \leq  c   \mathds{1}_{ \left\{ S_j \in [ -\frac{4}{\alpha} \log j, \,  0] \right\} } 
    \mathds{1}_{ \left\{ S_i \leq 0, \  \forall 1 \leq i \leq j \right\} }  
    + \frac{c}{j^2}, 
\end{align*}
applying Lemma \ref{Lem_CondiLLT_cc} we obtain
\begin{align*}
Y_{n,1}(A)  
& \leq  c \log^2 n \,   \sum_{j = 1}^{n}  \bb E_{\bb Q_{\alpha}^x}  
   \left[ \mathds{1}_{ \left\{ S_j \in  [- \frac{4}{\alpha} \log j, \, 0] \right\} } 
    \mathds{1}_{ \left\{ S_i \leq 0, \  \forall 1 \leq i \leq j \right\} }   \right]   +  c \log^2 n  \nonumber\\
& \leq   c \log^2 n \,  \sum_{j = 1}^{n}  \frac{\log^2 j}{(j+1)^{3/2}}  +  c \log^2 n   
 \leq  c \log^2 n. 
\end{align*}
For $Y_{n,2}(A)$, since $S_j \leq - \frac{3}{2 \alpha} \log n$ for any $j \in [n+1, 2n-1]$,
using Lemma \ref{Lem_CondiLLT_cc} we get that for any $j \in [n+1, 2n-1]$, 
\begin{align*}
&  \bb E_{\bb Q_{\alpha}^x}  
     \left[  e^{\alpha S_j }
      \mathds{1}_{ \left\{ S_i \in  I_i, \  \forall 1 \leq i \leq j \right\} }  \right] 
   \leq  \bb E_{\bb Q_{\alpha}^x}  
     \left[  e^{\alpha S_j }
      \mathds{1}_{ \left\{ S_i \leq 0, \  \forall 1 \leq i \leq j \right\} }  \right]   \nonumber\\
& =  \bb E_{\bb Q_{\alpha}^x}  
     \left[  e^{\alpha S_j }
      \mathds{1}_{ \left\{ S_i \leq 0, \  \forall 1 \leq i \leq j \right\} }
      \mathds{1}_{ \left\{ S_j \in \left[ - \frac{3}{\alpha} \log n, - \frac{3}{2 \alpha} \log n \right] \right\} }  \right]
      +  \bb E_{\bb Q_{\alpha}^x}  
     \left[  e^{\alpha S_j }
      \mathds{1}_{ \left\{ S_i \leq 0, \  \forall 1 \leq i \leq j \right\} }
      \mathds{1}_{ \left\{ S_j < - \frac{3}{ \alpha} \log n  \right\} }  \right]  \nonumber\\
& \leq  \frac{1}{n^{3/2}}  \bb Q_{\alpha}^x  \left( \max_{1 \leq i \leq j} S_i \leq 0,   \  
       S_j  \in \left[ -\frac{3}{\alpha} \log n, -\frac{3}{2 \alpha} \log n \right]  \right) + \frac{1}{n^3}  \nonumber\\
&  \leq  c \frac{1}{n^{3/2}} \frac{\log^2 n}{j^{3/2}} + \frac{1}{n^3},
\end{align*}
so that 
\begin{align*}
Y_{n,2}(A)  \leq  c n^3   \sum_{j = n+1}^{2n-1} \frac{1}{2n-j}  
   \left( \frac{1}{n^{3/2}} \frac{ \log^2 n }{j^{3/2}} + \frac{1}{n^3} \right)  
     =   c n^{3/2} \log^2 n  \sum_{j = n+1}^{2n-1} \frac{1}{2n-j} \frac{1}{j^{3/2}} 
   + c  \sum_{j=1}^{n-1} \frac{1}{j}  
 \leq c \log^3 n.  
\end{align*}
Hence 
\begin{align}\label{Upper-YnA}
Y_n(A)  \leq  Y_{n,1}(A) + Y_{n,2}(A) 
 &  \leq  c \log^3 n. 
\end{align}
Since $\bb E ((U_n^x(A))^2) = \bb E (U_n^x(A)) + Y_n(A)$ and $\bb E (U_n^x(A)) \geq c \pi_s(A) >0$ (cf.\ \eqref{Lower-UnA}),
using \eqref{Upper-YnA} we get
\begin{align*}
\bb E [(U_n^x(A))^2]
 \leq \bb E (U_n^x(A)) + c \log^3 n 
 \leq  c  \log^3 n  \,  (\bb E U_n^x(A))^2.  
\end{align*}
Hence
\begin{align*}
\bb P (U_n^x(A) > 0) 
\geq \frac{(\bb E U_n^x(A))^2}{\bb E[(U_n^x(A))^2]}
\geq  \frac{1}{c \log^3 n},
\end{align*}
which, by the definition of $U_n^x(A)$,  implies that 
\begin{align}\label{Pf_LowerBound_Zn_hh}
 \bb P \left( M_{2n}^x(A)  \geq -\frac{3}{2 \alpha} \log n - \Delta \right)  
 \geq  \bb P (U_n^x(A) > 0)   
 \geq  \frac{1}{c \log^3 n}. 
\end{align}
In the same way, one can also prove such a lower bound for  $M_{2n-1}^x(A)$.
The assertion of the lemma follows. 
\end{proof}

Using Lemma \ref{Lem-intermediate001}, we now give a proof of the following upper bound for $M_n^x(A)$. 

\begin{lemma}\label{Lem-Minimal-Low}
Assume conditions \ref{Condi_N},  \ref{Condi-density}, \ref{Condi-density-invariant} and \ref{Condi_ms}. 
Let $x \in \bb P(V)$. 
Then, for any Borel set $A \subseteq \bb P(V)$ and any $b < - \frac{3}{2 \alpha}$,
\begin{align}\label{Pf_Minimal_Low}
\lim_{n \to \infty} \bb P \left( \left. \frac{M_n^x(A)}{ \log n } \leq b   \,  \right| \,  \mathscr S  \right) = 0. 
\end{align} 
\end{lemma}

\begin{proof}
Let $\ee > 0$ and let $\eta_n$ be the first time when the number of particles exceeds $n^{\ee}$, i.e.
$\eta_n = \inf \{ k \geq 1:  ( \sum_{|u| = k} 1 ) \geq [n^{\ee}] \}.$ 
Denote $m = \bb E N$. 
By the Kesten-Stigum theorem, on the system's ultimate survival, 
$\frac{1}{m^n} \sum_{|u| = n} 1$  
converges almost surely as $n\to\infty$ to a strictly positive random variable, say $W >0$. 
Hence $ \limsup_{n\to\infty} \frac{\eta_n}{\log n} = \frac{\ee}{\log m}$ almost surely.
For brevity, denote for any constant $\Delta >0$,
\begin{align*}
& A_n =  \left\{ M_{n}^x(A) \geq  - c \eta_n  - \frac{3}{2 \alpha} \log n - \Delta  \right\}. 
\end{align*}
We will prove that 
\begin{align} \label{main proof 001}
\bb P \left( \liminf_{n \to \infty}  A_n \right) = 1, 
\end{align}
which in turn will imply \eqref{Pf_Minimal_Low}. Indeed, if \eqref{main proof 001} holds, 
then, almost surely,  on the system's survival,  
\begin{align*}
M_{n}^x(A)  \geq  - c \eta_n  - \frac{3}{2 \alpha} \log n - \Delta, 
\end{align*}
which implies that, on the system's survival,  
\begin{align*}
\liminf_{n \to \infty} \frac{M_{n}^x(A)}{\log n} 
\geq  - c \limsup_{n \to \infty} \frac{\eta_n}{\log n} - \frac{3}{2 \alpha} 
 =  - c \frac{\ee}{\log m} -  \frac{3}{2 \alpha}. 
\end{align*}
Since $\ee >0$ can be arbitrary small, we get \eqref{Pf_Minimal_Low}.

In order to prove \eqref{main proof 001}, we introduce sets $C_n$ with the property that if $\bb P(\liminf_{n \to \infty} C_n^c) =1$
then $\bb P(\liminf_{n \to \infty} A_n) =1$,
and use the Borel-Cantelli lemma to show that $\bb P(\limsup_{n \to \infty} C_n) = 0$.

The details are as follows. 
Since $\eta_n$ grows like $\log n$, 
it holds that $\bb P (\limsup_{n \to \infty} \{ \eta_n > n/2\} ) = 0$, so that 
in order to show \eqref{main proof 001}, it is enough to prove that
\begin{align} \label{liminf_Anc}
\bb P \left(  \liminf_{n \to \infty} (A_n \cup \left\{\eta_n > n/2 \right\}) \right)=1.
\end{align}
Denote, additionally, for any constant $\Delta >0$, 
\begin{align*}
& B_n =  \left\{ M_{n}^x(A) <  m_{\eta_n}^x  - \frac{3}{2 \alpha} \log n - \Delta, \eta_n \leq n/2 \right\},  \\
& C_n =  \left\{ \min_{k \in [\frac{n}{2}, n]}  M_{k + \eta_n}^x(A)  <  m_{\eta_n}^x  - \frac{3}{2 \alpha} \log n - \Delta,
\eta_n \leq n/2   \right\}, 
\end{align*}
where $m_{\eta_n}^x$ is defined by \eqref{MinPosition001A} with $A = \bb P(V)$.
Since on the set $\{ \eta_n \leq n/2 \}$, it holds that $\min_{k \in [\frac{n}{2}, n]}  M_{k + \eta_n}^x(A) \leq M_{n}^x(A)$
and therefore $B_n \subset C_n$. 
Note that 
\begin{align*} 
\liminf_{n \to \infty}  (A_n \cup \left\{\eta_n > n/2 \right\})
\  \supset  \  \liminf_{n \to \infty}  B_n^c  
 \cap  \liminf_{n \to \infty}  \left\{   m_{\eta_n}^x  \geq -c\eta_n  \right\}.  
\end{align*}
By Lemma  \ref{lem-liminf of min-001},   it holds that $\bb P (\liminf_{n \to \infty}  \{   m_{\eta_n}^x  \geq -c\eta_n \}) = 1$.  
Since $B_n^c \supset C_n^c$, to prove \eqref{liminf_Anc} it suffices to show that
\begin{align} \label{liminf_Cn_aa}
\bb P \left(  \liminf_{n \to \infty}  C_n^c \right)=1.
\end{align}

As said before, \eqref{liminf_Cn_aa} will be proved by using the Borel-Cantelli lemma. 
For this we give an upper bound for $\bb P(C_n)$. 
Clearly $\bb P (C_n)
\leq  \sum_{k \in [\frac{n}{2}, n]} R_k$, where
\begin{align*}
  R_k=  \bb P  \left\{  M_{k + \eta_n}^x(A)  <  m_{\eta_n}^x  - \frac{3}{2 \alpha} \log n - \Delta, 
                        \, \eta_n \leq  \frac{n}{2} \right\}.
\end{align*}
By the definition of $M_{k + \eta_n}^x(A)$ and $m_{\eta_n}^x$, we have
\begin{align*}
R_k  & =  \bb P \left(   \max_{|u| = k + \eta_n, X^x_u \in A}  S^x_u
   < \min_{|v| = \eta_n} S^x_v - \frac{3}{2 \alpha} \log n - \Delta, \, \eta_n \leq  \frac{n}{2} \right)  \\
& =  \sum_{j = 1}^{[n/2]}  \bb P \left(   \max_{|u| = j + k, X^x_u \in A}  S^x_u
   < \min_{|v| = j} S^x_v - \frac{3}{2 \alpha} \log n - \Delta, \, \eta_n = j \right). 
\end{align*} 
Recall the notion of the shift operator $[\cdot]_v$ from \eqref{eq:shift-operator}. By the cocycle property, for any $v \in \bb T$ and $u \in \bb T_v$,
\begin{align}\label{Cocycle}
  S^x_{vu}  =   S^x_v +  \big[S^{X^x_v}_u \big]_v. 
\end{align}
Since for any $v$ with $|v|=j$, $S^x_v \geq \min_{|v'| = j}  S^x_{v'}$, 
 from \eqref{Cocycle} we have  for any $u \in \bb T_v$ with $|u| = k$, 
\begin{align*}
S^x_{vu}  
\geq  \min_{|v'| = j}  S^x_{v'} +  \big[S^{X^x_v}_u \big]_v. 
\end{align*}
Hence
\begin{align}\label{Inequality-max-uj}
\max_{|v| = j, |u|=k, X^x_{vu} \in A} S^x_{vu}   
\geq  \min_{|v'| = j}  S^x_{v'} 
 + \max_{|v| = j, |u|=k, X^x_{vu} \in A}   \big[S^{X^x_v}_u\big]_v. 
\end{align}
It follows that 
\begin{align}\label{Bound-Rk-01}
R_k   \leq \sum_{j = 1}^{[n/2]}    
  \bb P \Bigg(  \max_{|v| = j, |u|=k, X^x_{vu} \in A}   \big[S^{X^x_v}_u\big]_v
   <  - \frac{3}{2 \alpha} \log n - \Delta,  \eta_n = j   \Bigg).
\end{align}
Since 
\begin{align*} 
\max_{|v| = j, |u|=k, X^x_{vu} \in A}   \big[S^{X^x_v}_u\big]_v = \max_{|v|=j} \big[M_k^{X^x_v}(A)\big]_v, 
\end{align*}
it holds that 
\begin{align*}
 \bb P \left(   \max_{|v| = j, |u|=k, X^x_{vu} \in A}   \big[S^{X^x_v}_u\big]_v   < 
   - \frac{3}{2 \alpha} \log n - \Delta,  \eta_n = j  \right)  
 =  \bb P \left(    \big[M_k^{X^x_v}(A)\big]_v
   <  - \frac{3}{2 \alpha} \log n - \Delta, \,  \forall |v| =j, \, \eta_n =j  \right).
\end{align*}
Using the independence of the branching property of each node, 
from \eqref{Bound-Rk-01} we get, upon conditioning on $\scr{F}_{\eta_n}$ in the last step,
\begin{align*} 
R_k  
 & \leq \sum_{j = 1}^{[n/2]}  \bb P  \left(    \big[M_k^{X^x_v}(A)\big]_v
 <  - \frac{3}{2 \alpha} \log n - \Delta, \,  \forall |v| =j, \, \eta_n =j  \right)  \notag\\
   & = \bb P \left(    \big[M_k^{X^x_v}(A)\big]_v
   <  - \frac{3}{2 \alpha} \log n - \Delta, \, \forall |v| = \eta_n, \, \eta_n \leq [n/2] \right)  \notag\\
   & \leq \bb P \left(    \big[M_k^{X^x_v}(A)\big]_v
   <  - \frac{3}{2 \alpha} \log n - \Delta, \, \forall |v| = \eta_n  \right)  \notag\\ 
 & = \bb E  \left[ \prod_{|v| = \eta_n}  \mathds{1}_{ \big\{ \big[M_k^{X^x_v}(A)\big]_v
 	<  - \frac{3}{2 \alpha} \log n - \Delta \big\}}  \right] \notag\\
 & = \bb E \left[ \prod_{|v|=\eta_n} F(X^x_v)\right].
\end{align*}
Here $F(x)=\bb P(M_k^x <  - \frac{3}{2 \alpha} \log n - \Delta)$, for which we have found a uniform bound in Lemma \ref{Lem-intermediate001}.
Using further the definition of $\eta_n$ and the inequality $\log (1 - t) \leq -t$ for $t \in (0,1)$,  we get
\begin{align*} 
R_k  \leq  \bb E  \left[ \prod_{|u| = \eta_n}   \left( 1 -  \frac{1}{c \log^3 n}  \right)   \right]
\leq  \left( 1 -  \frac{1}{c \log^3 n}  \right)^{[n^\ee]}
\leq  e^{- [n^\ee] \frac{1}{ \log^3 n } }.  
\end{align*}
This implies 
\begin{align*}
\bb P (C_n)  \leq \sum_{k \in [\frac{n}{2}, n]} R_k
 \leq \sum_{k \in [\frac{n}{2}, n]} e^{- [n^\ee] \frac{1}{ \log^3 n } }  
 \leq n  e^{- [n^{\ee/2}]}. 
\end{align*} 
Since $\sum_{n = 1}^{\infty} n  e^{- [n^\ee]/2} < \infty$,  
we get that $\sum_{n = 1}^{\infty} \bb P(C_n) < \infty$.
By the Borel-Cantelli lemma, we have $\bb P ( \limsup_{n \to \infty} C_n) = 0$ and hence \eqref{main proof 001} holds. 
This completes the proof of the lemma. 
\end{proof}

\section{Proof of Theorems \ref{Thm_In_Pro_min} and \ref{Thm_In_Pro_Coeff}}

\subsection{Proof of Theorem \ref{Thm_In_Pro_min}}
The following result is an analog of Lemma \ref{Lem_CondiLLT_cc}. 

\begin{lemma}\label{Lem_CondiLLT-minimal}
Assume condition \ref{Condi-density}, \ref{Condi-density-invariant} 
and $\kappa'(s) = 0$ for some $s \in  I_{\mu}$. 
Then, there exists a constant $c >0$ such that for any $x \in \bb P(V)$, $y \geq 0$, 
$n \geq 1$ and $0 \leq  a < b \leq \sqrt{n} \log n$,
\begin{align*}
 \mathbb{Q}_s^x  \left( y + S_{n}\in [a,b], \,  \tilde{\tau}_{y}>n\right)  
 \leq  c \frac{ (1+y) \wedge n^{1/2} }{n^{3/2}} (b-a + 1) ( b+a+ 1), 
\end{align*}
where $\tilde{\tau}_y = \inf \{ k \geq 1: y + S_k < 0\}$. 
\end{lemma}
Since the proof of Lemma \ref{Lem_CondiLLT-minimal} can be carried out in the same way as that of Lemma \ref{Lem_CondiLLT_cc},
we omit the details. 

The following result is similar to Lemma \ref{Lem-Minimal-Upp}. 

\begin{lemma}\label{Lem-Minimal-Low-beta}
Assume conditions \ref{Condi_N},  \ref{Condi-density}, \ref{Condi-density-invariant} and \ref{Condi_Neg_alpha}. 
Let $x \in \bb P(V)$. 
Then, there exists a constant $c> 0$ such that for any $\ee \in (0, \frac{3}{2 \beta})$ and any Borel set $A \subseteq \bb P(V)$, 
\begin{align*}
I: = \bb P \left(\left. \frac{m_n^x(A)}{\log n}  \leq  \frac{3}{2 \beta} -  \ee  \, \right| \,  \mathscr S \right) 
 \leq c \frac{\log^3 n}{n^{\ee \beta}}. 
\end{align*}
\end{lemma}

\begin{proof}
Let $K >1$. We write $I = I_1 + I_2$, where 
\begin{align*}
I_1  & = \bb P \left(  \left.   \frac{m_n^x(A)}{\log n} \leq  \frac{3}{2 \beta} - \ee, 
  \  \min_{1 \leq i \leq n}  m_i^x(A) \geq  -K    \, \right| \,  \mathscr S \right),   \nonumber\\
I_2  & =  \bb P \left(  \left.    \frac{m_n^x(A)}{\log n}  \leq  \frac{3}{2 \beta} - \ee, 
  \   \min_{1 \leq i \leq n} m_i^x(A) < - K   \, \right| \,  \mathscr S \right), 
\end{align*}
with the notation $m_i^x(A) = \inf \{ S^x_{u|i}: |u| = n,  \,  X_u^x \in A \}$ 
for $x \in \bb P(V)$ and $1 \leq i \leq n$. 
For the first term $I_1$, we have $I_1 \leq \bb E (Z_n^x(A) \, | \,  \mathscr S)$, 
where, for $x \in \bb P(V)$, 
\begin{align*}
Z_n^x(A) = \sum_{|u| = n} 
\mathds{1}_{ \left\{X^x_u \in A, \  \frac{S^x_u}{\log n} \leq  \frac{3}{2 \beta} -  \ee, \  S^x_{u|i} \geq  -K,  \  \forall \,  1 \leq i \leq n \right\} }. 
\end{align*}
By the many-to-one formula \eqref{Formula_many_to_one}, 
the fact that $\mathfrak m (-\beta)  = 1$ (cf.\ condition \ref{Condi_Neg_alpha})
and $r_{-\beta}$ is bounded and strictly positive on $\bb P(V)$ (cf.\ Proposition \ref{prop:Ps}),  
there exists a constant $c>0$ such that for any $x \in \bb P(V)$, 
\begin{align*}
 \bb E(Z_n^x(A))   
& = r_{-\beta}(x) \bb E_{\bb Q_{-\beta}^x}  
 \left[ \frac{1}{r_{-\beta} (X_n)} e^{\beta S_n} 
   \mathds{1}_{ \{ X_n \in A, \,  S_n \leq  (\frac{3}{2 \beta} - \ee) \log n, 
        \   S_i \geq -K,  \  \forall 1 \leq i \leq n \} } \right]  \nonumber\\
&  \leq  c n^{ \frac{3}{2} - \ee \beta} 
    \bb Q_{-\beta}^x \left\{ S_n \leq  \left( \frac{3}{2 \beta} - \ee \right) \log n, 
                    \  S_i \geq  -K, \  \forall 1 \leq i \leq n \right\} \nonumber\\
&  =  c  n^{ \frac{3}{2} - \ee \beta}    
   \bb Q_{-\beta}^x \left\{ K + S_n \in \left[0,  K + \left(\frac{3}{2 \beta} - \ee \right) \log n \right], 
         \,  \tilde{\tau}_K > n  \right\},   
\end{align*}
where $\tilde{\tau}_K = \inf \{ k \geq 1: K + S_k < 0\}$. 
Applying Lemma \ref{Lem_CondiLLT-minimal} and taking $K = a \log n$ with $a>0$ (whose value will be chosen to be sufficiently large), 
we obtain
\begin{align*} 
\bb E(Z_n^x(A)) 
 \leq  c  n^{ \frac{3}{2} - \ee \beta}   (1 + K) \frac{[ K + ( \frac{3}{2 \beta} - \ee) \log n ]^2}{n^{3/2}}  
 \leq c \frac{a^3 \log^3 n}{n^{\ee \beta}}. 
\end{align*}
It follows that 
\begin{align}\label{Mini_I1_001-beta}
I_1 \leq \bb E (Z_n^x(A) \, | \,  \mathscr S) \leq  c \frac{a^3 \log^3 n}{n^{\ee \beta}}. 
\end{align}
For $I_2$, by Markov's inequality, the many-to-one formula \eqref{Formula_many_to_one}
and the fact that $\mathfrak m (-\beta)  = 1$ (cf.\ condition \ref{Condi_Neg_alpha}), we have 
\begin{align}\label{Mini_I2_001-beta}
I_2 & \leq  c \bb P \left( \min_{1 \leq i \leq n} m_i^x(A) <  -K \right)  
 \leq  c \bb P \left(  \min_{1 \leq i \leq n} \min_{|u|=i}  S^x_u <  -K   \right)  \nonumber\\
& \leq  c e^{- \beta K} \bb E \left( e^{ -\beta \min_{1 \leq i \leq n} \min_{|u|=i}  S^x_u} \right)   
 \leq  c e^{- \beta K} \bb E \left( \max_{1 \leq i \leq n} \sum_{|u|=i} e^{-\beta  S^x_u} \right)   \nonumber\\
& \leq  c e^{- \beta K}   \sum_{i = 1}^n  \bb E \left(  \sum_{|u|=i} e^{-\beta S^x_u} \right)  
=  c e^{- \beta K} r_{-\beta}(x)  \sum_{i = 1}^n  \bb E_{\bb Q_{-\beta}^x } \left( \frac{1}{ r_{-\beta}(X_i) } \right)
 \leq  \frac{c}{n^{a \beta - 1}}
 \leq  \frac{c}{n^{a \beta /2}},
\end{align}
by taking $K = a \log n$ and $a>0$ sufficiently large. 
Combining \eqref{Mini_I1_001-beta} and \eqref{Mini_I2_001-beta} ends the proof of Lemma \ref{Lem-Minimal-Low-beta}. 
\end{proof}

The following result is an analog of Lemma \ref{Lem-Minimal-Low} and its proof can be carried out in the same way. 

\begin{lemma}\label{Lem-Minimal-Low-beta-upp}
Assume conditions \ref{Condi_N},  \ref{Condi-density}, \ref{Condi-density-invariant} and \ref{Condi_Neg_alpha}. 
Let $x \in \bb P(V)$. 
Then, for any Borel set $A \subseteq \bb P(V)$ and any $b > \frac{3}{2 \beta}$,
\begin{align*}
\lim_{n \to \infty} \bb P \left( \left. \frac{m_n^x(A)}{ \log n } > b   \,  \right| \,  \mathscr S  \right) = 0. 
\end{align*} 
\end{lemma}

Theorem \ref{Thm_In_Pro_min} follows from Lemmas \ref{Lem-Minimal-Low-beta} and \ref{Lem-Minimal-Low-beta-upp}.

\subsection{Proof of Theorem \ref{Thm_In_Pro_Coeff}}

We first prove \eqref{thm1_Minimal_aa_Coeff} for the coefficients $\langle f, G_u v \rangle$.  
Since $|\langle f, G_u v \rangle| \leq \|G_u v\|$, 
using Theorem \ref{Thm_In_Pro}, we get, for any $\ee > 0$, 
\begin{align*}
\lim_{n \to \infty}  
\bb P \left( \left.   \frac{ \sup_{G_u\cdot x\in A, | u | = n} \log |\langle f, G_u v \rangle| }{\log n} + \frac{3}{2 \alpha}  >  \ee
  \, \right| \,  \mathscr S \right) = 0. 
\end{align*}
On the other hand, we denote $A_n = \left\{ \log |\langle f, G_u v \rangle| - \log \|G_u v\| \leq - \frac{\ee}{2}  \log n \right\}$
and by $A_n^c$ its complement. 
By \cite[Lemma 14.11]{BQ16b}, we have $\bb P(A_n) \leq  n^{-c}$ for some constant $c>0$. 
This, together with the fact that $\bb P (\mathscr S) >0$, implies that 
\begin{align*}  
& \bb P \left( \left.   \frac{ \sup_{G_u\cdot x\in A, | u | = n} \log |\langle f, G_u v \rangle| }{\log n} + \frac{3}{2 \alpha}  <  -\ee
  \, \right| \,  \mathscr S \right)  \nonumber\\
& \leq \frac{1}{n^c} 
   +  \bb P \left( \left.   \frac{ \sup_{G_u\cdot x\in A, | u | = n} \log \|G_u v\| }{\log n} + \frac{3}{2 \alpha}  <  - \frac{\ee}{2}
  \, \right| \,  \mathscr S \right),
\end{align*}
which converges to $0$ as $n \to \infty$, using Theorem \ref{Thm_In_Pro}. 
The proof of \eqref{thm1_Minimal_aa_Coeff} for the coefficients $\langle f, G_u v \rangle$ is complete.

We next prove \eqref{thm1_Minimal_aa_Coeff} for the operator norm $\|G_u\|$.  
Since $\|G_u \| \geq \| G_u v \|$, 
by Theorem \ref{Thm_In_Pro}, we get, for any $\ee > 0$, 
\begin{align*}
\lim_{n \to \infty}  
\bb P \left( \left.   \frac{ \sup_{G_u\cdot x\in A, | u | = n} \log \| G_u \| }{\log n} + \frac{3}{2 \alpha}  <  -\ee
  \, \right| \,  \mathscr S \right) = 0. 
\end{align*}
Denote $A_n = \left\{ \log \| G_u v \| - \log \|G_u \| \leq - \frac{\ee}{2}  \log n \right\}$
and by $A_n^c$ its complement. 
By  \cite[Proposition 3.11]{Hde19}, we see that $\bb P (A_n) \leq n^{-c}$ for some constant $c >0$. 
Then,  we obtain
\begin{align*}  
\bb P \left( \left.   \frac{ \sup_{G_u\cdot x\in A, | u | = n} \log \| G_u \| }{\log n} + \frac{3}{2 \alpha}  >  \ee
  \, \right| \,  \mathscr S \right)  
 \leq \frac{1}{n^c} 
   +  \bb P \left( \left.   \frac{ \sup_{G_u\cdot x\in A, | u | = n} \log \|G_u v\| }{\log n} + \frac{3}{2 \alpha}  > \frac{\ee}{2}
  \, \right| \,  \mathscr S \right),
\end{align*}
which converges to $0$ as $n \to \infty$, using Theorem \ref{Thm_In_Pro}. 
The proof of \eqref{thm1_Minimal_aa_Coeff} for the operator norm $\|G_u\|$ is complete. 

We finally prove \eqref{thm1_Minimal_aa_Coeff} for the spectral radius $\rho(G_u)$. 
Since $\rho(G_u) \leq  \|G_u \|$, 
using the law of large numbers \eqref{thm1_Minimal_aa_Coeff} for the operator norm $\|G_u\|$, 
we get that for any $\ee > 0$, 
\begin{align*}
\lim_{n \to \infty}  
\bb P \left( \left.   \frac{ \sup_{G_u\cdot x\in A, | u | = n} \log \rho(G_n) }{\log n} + \frac{3}{2 \alpha}  > \ee
  \, \right| \,  \mathscr S \right) = 0. 
\end{align*}
Denote $A_n = \left\{ \log \rho(G_n) - \log \|G_u \| \leq - \frac{\ee}{2}  \log n \right\}$
and by $A_n^c$ its complement. 
By  \cite[Lemma 14.13]{BQ16b}, we see that $\bb P (A_n) \leq n^{-c}$ for some constant $c >0$. 
Hence, 
\begin{align*}  
 \bb P \left( \left.   \frac{ \sup_{G_u\cdot x\in A, | u | = n} \log \rho(G_n) }{\log n} + \frac{3}{2 \alpha}  <  -\ee
  \, \right| \,  \mathscr S \right)  
 \leq \frac{1}{n^c} 
   +  \bb P \left( \left.   \frac{ \sup_{G_u\cdot x\in A, | u | = n} \log \|G_u \| }{\log n} + \frac{3}{2 \alpha}  < - \frac{\ee}{2}
  \, \right| \,  \mathscr S \right),
\end{align*}
which converges to $0$ as $n \to \infty$, using the law of large numbers \eqref{thm1_Minimal_aa_Coeff} for the operator norm $\|G_u\|$. 
This finishes the proof of \eqref{thm1_Minimal_aa_Coeff} for the spectral radius $\rho(G_u)$.

%
%

\begin{acks}[Acknowledgments]
The authors would like to thank the referees and the editor for
very constructive comments which contributed to improve the presentation of the paper. 

Hui Xiao is corresponding author. 
\end{acks}
\begin{funding}
%
The authors  were supported by DFG grant ME 4473/2-1. 
Hui Xiao was also supported by the National Natural Science Foundation of China (No. 12288201). 
\end{funding}





\end{document}